\documentclass[a4paper,10pt]{article}

\usepackage[utf8]{inputenc}
\usepackage{frcursive} 
\usepackage[english]{babel}
\usepackage{amsmath}
\usepackage {amsfonts,color}
\usepackage{stmaryrd}
\numberwithin{equation}{section}

\def\vrho{\rho}
\def\prho{{\rho'}}

\usepackage{amsmath}
\usepackage{amssymb}
\usepackage[utf8]{inputenc}
\usepackage{amsfonts}
\usepackage{amsmath}
\newcommand{\cv}{c_{\mathbf Y}} % Constant for convolution inequality
\newcommand{\ch}{c_{\mathbf{HK}}} %Constant for the alpha-stable heat type estimate
\newcommand{\cc}{\mbox{\small{\begin{cursive}c\end{cursive}}}}
\newcommand{\tr}{\textcolor{red}}

\def\R{\mathbb{R}}
\def\1{\mathbf{1}}
\def \E{\mathbb{E}}

\def \N{\mathbb{N}}
\def \I{\mathbb{I}}
\def \ba{\mathbf{a}}

\def\B{{B}}

\def\eps{\varepsilon}

\def\d{{d}}

\title{\textbf{Weak error for SDEs with additive {stable} noise and singular drift: choose the test function in the same space as the drift!}}

\author{Benjamin Jourdain\footnote{CERMICS, ENPC, Institut Polytechnique de Paris, CNRS, Marne-la-Vallée, France \& MATHRISK team-project, Inria Paris, France, benjamin.jourdain@enpc.fr}\phantom{S} and Stéphane Menozzi\footnote{Université d'Evry Val d'Essonne, Paris Saclay, LaMME, UMR CNRS 8071, 23 Boulevard de France, 91037 Evry France, stephane.menozzi@univ-evry.fr}}

\newtheorem{THM}{Theorem} 

\newtheorem{cor}[THM]{Corollary} 
\newtheorem{PROP}[THM]{Proposition}
\newtheorem{lem}[THM]{Lemma}
\newtheorem{rem}[THM]{Remark}

\begin{document}  
\maketitle
\begin{abstract}
We emphasize that for a stochastic differential equation with isotropic stable additive noise and {non Lipschitz} drift, when considering an appropriate discretization scheme and the associated weak error, it is somehow natural to consider a test function having the same spatial regularity as the drift involved. We will in particular focus on drifts belonging to Lebsegue, Hölder or Besov spaces with negative regularity index in their spatial variable. Choosing such a test function allows to improve the convergence rate previously obtained on the densities (for Lebesgue or Hölder drifts) or preserve the rate {for possibly singular generalized test functions} (for Besov spaces with negative regularity). 

\end{abstract}

\section{Introduction}
    For a fixed finite time horizon $T>0$, we are interested in  the Euler-Maruyama dicretization of the SDE
	\begin{equation}\label{SDE}
	d X_t = b(t,X_t)d t + d Z_t, \qquad X_0 = x, \qquad \forall t \in [0,T],
	\end{equation}
	where $Z_t$ is a symmetric isotropic $d$-dimensional $\alpha$-stable process, $\alpha\in (1,2] $ and the drift $b$ is somehow meant to be singular.
	Three main cases will be considered for the drift $b$ that will be {assumed} to belong to a:
	\begin{itemize}
\item[-] time-space Lebesgue space,
\item[-] a Hölder space in spatial variable  (viewed as a Besov space with positive regularity) uniformly in time (the reason why we assume uniformity in time in this second setting is just that the auxilliary results on the transition densities needed to perform our analysis are only available under this assumption and, in this already long paper, we prefer to save the space needed to extend them to $L^\vartheta$ regularity in time),
\item[-] Besov space with negative regularity in space with the Besov norm belonging to some Lebesgue space in time. Then we exclude the Brownian noise case $\alpha=2$ where the related heat kernel {estimates} are not fully available yet (see {\cite{Me:Pa:26}} where some particular case is treated).
	\end{itemize}

We will consider some related Euler schemes $(X^h_t)_{t\in[0,T]}$ with time-step $h=T/n$ for $n\in\N^*$ defined in the corresponding sections. Our starting point to analyse the weak error is the Duhamel formula satisfied by the density $y\mapsto \Gamma(0,x,t,y)$ of $X_t$ with respect to the Lebesgue measure for $t\in (0,T]$. In the first two cases, we have

        \begin{align*}
			\Gamma(0,x,t,y)
			= p_\alpha(t,y-x)-\int_{0}^{ t}\int_{\R^d}\Gamma(0,x,s,z)b(s,z)\cdot\nabla_y  p_\alpha(t-s,y-z)d zd s,
        \end{align*}
        where $z\mapsto p_\alpha(s,z)$ denotes the density of $Z_s$ when $s\in (0,T]$. The case when the drift has some negative regularity in space is similar in spirit but slightly more complicated since the drift is not a function. The density $y\mapsto \Gamma^h(0,x,t,y)$ of $X^h_t$ satisfies an analogous Duhamel formula with of course some slight modifications : the drift $b$ is replaced some regularization $b_h$ depending on $h$ and the first two occurences of $s$ in the integral are replaced by distinct time variables belonging to the same time step. The strategy initiated in \cite{BJ20} when $\alpha=2$ and $b$ is bounded to derive some weak error bound consists in estimating the contribution $\Gamma^h(0,x,t,y)-p_\alpha(t,y-x)+\int_{0}^{ t}\int_{\R^d}\Gamma^h(0,x,s,z)b(s,z)\cdot\nabla_y  p_\alpha(t-s,y-z)d zd s$ and then dealing with $$\int_{0}^{ t}\int_{\R^d}(\Gamma(0,x,s,z)-\Gamma^h(0,x,s,z))b(s,z)\cdot\nabla_y  p_\alpha(t-s,y-z)d zd s$$ by some Gr\"onwall's type argument. The spatial integral amounts to the convolution of the difference between the two densities multiplied by $b(s,\cdot)$ with the gradient of $p_\alpha(t-s,\cdot)$. The product $(\Gamma(0,x,s,\cdot)-\Gamma^h(0,x,s,\cdot))b(s,\cdot)$ is easily estimated when the difference between the two densities is measured in terms of the dual norm of the space where the drift belongs as a function of the spatial variable for fixed time variable. This explains why this choice which amounts to integrate the difference $\Gamma(0,x,t,z)-\Gamma^h(0,x,t,z)$ with respect to some test function in the same space as the drift is particularly appealing in view of the final Gronwall argument. This strategy has been successfully implemented in \cite{BJ20} when $\alpha=2$, where, since the drift coefficient is assumed to be bounded, the difference between the densities is estimated in $L^1$ norm or equivalently the difference between the corresponding distributions estimated in total variation distance. The derived order of convergence is $1/2$. Another possibility explored in \cite{jour:meno:24,fito:jour:meno:25,fito:meno:25} and inspired by \cite{MPZ21} consists in performing the Gronwall's argument on $\sup_{y\in\R^d}\frac{{|}\Gamma(0,x,r,y)-\Gamma^h(0,x,r,y){|}}{\bar p_\alpha (r,y)}$ where $\bar p_\alpha=p_\alpha$ when $\alpha\in (1,2)$ and the variance of the Gaussian density $p_2(r,\cdot)$ is multiplied by a factor greater than one and not depending on $r\in (0,T]$ to define $\bar p_2(r,\cdot)$. For $\alpha\in (1,2)$, this approach has also been successfully extended in \cite{fito:isso:meno:25} to the distributional case {when the drift belongs to a Besov space with negative regularity in space.}
        
 Let us point out that in the considered cases the convergence  rate obtained in the quoted works is strongly linked to the conditions ensuring weak/strong existence and uniqueness for the underlying SDE. Namely, in the Lebesgue setting, i.e. when $b\in L^\vartheta([0,T],L^p(\R^d))$, weak well posedness follows from the condition $0<\alpha-1-(\frac{d}p+\frac{\alpha}{\vartheta})$ (Krylov and Röckner type criterion) which also yields strong uniqueness if $\alpha=2 $ (see e.g. \cite{fito:jour:meno:25},\cite{kryl:rock:05}). In the Hölder setting, $b\in L^\infty([0,T],C^\beta(\R^d,\R^d))$ with $\beta\in (0,1)$, it is known that the SDE is strongly well posed for $\alpha=2 $ and weakly well posed as soon as $0<\alpha-1+\beta $ (for even possible super-critical regimes $\alpha\in (0,1) $ that we do not consider here), see \cite{CZZ21}. Eventually for $b\in L^\vartheta([0,T],B_{p,q}^\beta), \beta<0$ (a precise definition of Besov spaces is provided in Section \ref{SEC_BESOV} below), well posedness of the martingale problem associated with the \textit{formal} SDE\footnote{recall that $b$ is here a distribution} is ensured by the condition $0<\alpha-1-(\frac{d}p+\frac{\alpha}{\vartheta})+2\beta$, see \cite{chau:meno:22}. We mention that there is a kind of \textit{continuity} in the conditions guaranteeing  well posedness in the three cases with respect to the integrability exponents $p,\vartheta $ and the regularity parameter $\beta $ (the Lebesgue case being associated with $\beta=0 $). The factor $2$ for $\beta<0 $ is due to technical constraints (paraproduct type rules).
 
The corresponding convergence rates for the error on the densities established in the quoted papers actually corresponds to the \textit{margin} or \textit{gap} to the inequality in the former conditions. Precisely, it reads $\gamma/\alpha $ with:
 \begin{trivlist}
 \item [-]$\gamma=\alpha-1-(\frac{d}p+\frac{\alpha}{\vartheta}) $ in the Lebesgue case,
 \item [-]$\gamma=\alpha-1+\beta$ in the Hölder case,
 \item [-]$(\gamma-\varepsilon)/\alpha,\ \varepsilon>0 $ with $\gamma=\alpha-1-(\frac{d}p+\frac{\alpha}{\vartheta})+2\beta $ in the Besov case of negative regularity, where the loss of $\varepsilon $ comes here  again from a technical constraint intrinsic to the distributional setting. {Still in the distributional setting, we can mention the recent paper \cite{hao:wu:26} in which the authors obtain a similar convergence rate for the considered case  $p=\vartheta=\infty $ therein for the total variation distance between the laws but can consider more general stable driving {noises}, with non degenerate spectral measures including e.g. cylindrical ones}.\end{trivlist}

       The convergence in the Hölder case reflects how the additional regularity $\beta>0 $ improves the {rate with respect to the Lebesgue case with $\vartheta=p=\infty$.} On the other hand, for the  Lebesgue case or  drifts in Besov spaces with negative regularity,  the convergence rates deteriorate when the parameters are almost critical.
        For $b\in L^\vartheta([0,T],L^p(\R^d))$, this is not satisfactory especially when compared with the results by Lê and Ling in \cite{le:ling:25} who obtained, in the Brownian setting $\alpha=2$, for any parameters satisfying $\gamma>0$, a convergence rate of order $1/2$ (up to a logarithmic factor) for the strong error. Their approach is based on stochastic sewing lemma techniques particularly well-suited to analyze the behavior of irregular functionals of the Brownian path.  In \cite{hao:le:ling:24,hao:le:ling:26}, the authors deal with the weak error and go back to the choice of the test function in the space where the drift belongs as a function of the spatial variables for fixed time. This space is adapted to the more complicated kinetic equation with Brownian noise that they consider. {When $\vartheta=\infty$ and $p>d$,} they prove weak convergence  with order $1/2$ of a semi-discrete Euler scheme (the time-variable of the drift coefficient is not discretized). This strongly suggests that 
        the order of convergence derived with the test function approach is in general larger than the one derived with spatial supremum of the ratio. In this paper, we investigate the test function approach in the three settings presented at the beginning of the introduction. It turns out that
 \begin{itemize}
 \item in the Lebesgue and H\"older settings, we improve the orders of convergence derived respectively in \cite{jour:meno:24,fito:jour:meno:25} and in \cite{fito:meno:25} to $\frac{\alpha-1}\alpha-\frac 1\vartheta$  and $1-\frac{(1-2\beta)^+}\alpha$ (note that we suppose $\vartheta=\infty$ in the second setting),
 \item in the negative regularity setting, we preserve the order of convergence derived in \cite{fito:isso:meno:25} when estimating the pointwise difference $\Gamma-\Gamma^h$ in the supremum norm, but we are able to integrate possibly singular test functions,  {in the spirit of \cite{guyo:06}, where Guyon though considered non-degenerate Brownian diffusions with somehow smooth coefficients. }
 \item as suggested by the Lebesgue case, the important point is to choose the test function with same regularity exponent $\beta$ in space as the drift coefficient (in the definition of the Besov space). Indeed, in the particular case $\beta=0$ which corresponds to the Lebesgue spaces, we are able to take advantage of the regularization by the convolution with $\nabla p_\alpha(t-r,\cdot)$ to estimate the error in $L^\rho$ with $\rho$ possibly larger than the conjugate exponent $\frac p{p-1}$ of $p$. This enables us to deal with the important case when the drift is the sum of contributions in distinct Lebesgue spaces which, to our knowledge, has not been covered so far.

 \end{itemize}

\subsection*{About the driving noise}
We will recall here some properties of the isotropic $\alpha $-stable driving noise, with $\alpha\in (1,2] $. In particular the behavior of the densities is dramatically different in the pure jump case $\alpha\in (1,2) $, for which the tails of the marginal densities have heavy tails, and the Brownian case, in which the marginal densities are Gaussian (and thus square exponentially thin tailed). In the following, we will denote by $p_\alpha(t,\cdot) $ the density of $Z_t$ at time $t>0$ and set for some constant $c>1$,
\begin{equation}
   \forall t\in (0,T],\;\bar p_\alpha(t,x)=\begin{cases}
                                                 t^{-\frac d\alpha}\left(1+t^{-\frac 1\alpha}|x|\right)^{-(d+\alpha)}\mbox{ if }\alpha\in (1,2),\\
                                                 p_2(ct,\cdot)\mbox{ if }\alpha=2.
   \end{cases}\label{eq:barpal}
\end{equation} Note that $p_2(t,\cdot)$ stands for the centered Gaussian density with covariance matric $t I_d$.  The following controls hold.

       \begin{PROP}[Pointwise Density estimates for the driving noise]\label{prop-controls-palpha} There exists $C\ge 1,\bar c\ge 1$ s.t. for all $0\textcolor{black}{\le} s<t\le T $, $x,w \in \R^d$, and any muti-index $\ba \in \N^d $, $|\ba|\in \{0,1, 2\} $,  
        $\theta\in \{0,1\} $,
    \begin{align}
    |\partial_t^\theta\partial_x^\ba p_\alpha (t-s,w-x)|&\le C (t-s)^{-(\theta+\frac{|\ba|+d}{\alpha})}\left(1+\frac{\left| w-x\right|}{(t-s)^{\frac{1}{\alpha}}} \right)^{-(d+\alpha+|\ba|)},\ \alpha\in (1,2),\notag\\
        |\partial_t^\theta\partial_x^\ba p_2(t-s,w-x)|&\le C(t-s)^{-(\theta+\frac{|\ba|+d}{\alpha})} \bar p_2(t-s,w-x). \label{GD_BOUNDS}
    \end{align}
 In particular, for $|\ba|\in \{1,2\} $ (i.e. when spatial derivatives are considered) and $\rho\in (0,1) $ 
     \begin{align}\label{drift-smoothing-iso-noise}
        |w-y|^\rho|\partial_t^\theta\partial_{\textcolor{black}{w}}^\ba  p_\alpha(t-s,y-w)|\le C (t-s)^{-\theta+\frac{\rho-|\ba|}{\alpha}} \bar p_\alpha(t-s,y-w).
    \end{align}
    and, for $y\in \R^d$ s.t. $|y-y'|\lesssim (t-s)^{\frac{1}{\alpha}}$, 
    \begin{align}\label{drift-smoothing-iso-noise-diag}
    	|w-y|^\rho|\partial_t^\theta\partial_w^\ba p_\alpha\left(t-s,{y-w+(y'-y)}\right)|\le C (t-s)^{-\theta+\frac{\rho-|\ba|}{\alpha}} \bar p_\alpha(t-s,y-w).
    \end{align}
	\color{black}
\end{PROP}
\begin{rem}
Let us emphasize that, for a pure-jump isotropic stable driving noise, the derivatives of the marginal densities actually have better decay properties at infinity. This precisely the meaning of \eqref{drift-smoothing-iso-noise} in this case, which in turn allows to absorb a possibly unbounded contribution of order $\rho\in (0,1] $ yielding a time-gain in $t^{\frac \rho\alpha} $ (parabolic space-time correspondance in the $\alpha $ stable setting).
\end{rem}
    For notational simplicity we will use from now on for two quantities $A$ and $B$ the symbol $A\lesssim B $ whenever there exists a constant $C:=C(\textcolor{black}{d},b,\alpha,T)$ s.t. $A\le CB $. Namely,
    \begin{equation}\label{DEF_LESSSIM}
A\lesssim B \Longleftrightarrow \exists C:=C(\textcolor{black}{d},b,\alpha,T),\ A\le CB.
    \end{equation}
\section{Drift in a single Lebesgue space}\label{SEC_LP}

\subsection{Standing assumptions and discretization scheme}
We will here first focus on the case where $b$ in \eqref{SDE} belongs to a time-space Lebesgue space $b\in L^\vartheta([0,T],L^p(\R^d))$ under the Krylo-Röckner type condition
\begin{align}
\label{COND_KR}
\frac{d}{p}+\frac{\alpha}{\vartheta}<\alpha-1,\ p\ge 2.
\end{align}
In that case, it is well known from \cite{kryl:rock:05} that equation \eqref{SDE} admits a unique strong solution for $\alpha=2 $. Without additional regularity condition, weak well-posedness was derived in \cite{fito:jour:meno:25} under \eqref{COND_KR} for $\alpha\in (1,2) $. We recall that in the pure jump case some additional mild smoothness properties are needed to guarantee strong well-posedness, see \cite{xie:zhan:20}.
In that same reference, Euler type approximation schemes were as well introduced, see also \cite{jour:meno:24} for the Brownian case. 
Defining, for a given $T>0$, {$h:=T/n$, $n\in \N^*$ and $t_k:=kh$}, we consider the dynamics:  
\begin{equation}\label{euler-scheme}
			X_{t_{k+1}}^h = X_{t_k}^h + (Z_{t_{k+1}}-Z_{t_{k}})+ hb_h(U_k,X_{t_k}^h),
\end{equation}
where 		
		\begin{align}
			{b}_h (t,y) &= \mathbb{I}_{t\geq h,|b(t,y)|>0}\frac{\min \left\{ |b(t,y)|,Bh^{\frac{1}{\alpha}-1}\right\}}{|b(t,y)|} b(t,y), &(t,y)\in [0,T] \times \R^d, \label{cutoff}
		\end{align}
and the random variables $ (U_k)_{k\in \{0,\cdots,n-1\}}$ are independent and such that $U_{k}$ is uniformly distributed on $[t_k,t_{k+1}]$. Namely, in the above scheme, the drift is truncated in space and randomized in time.

For the analysis, it might as well be useful to consider the related continuous time extension of the above scheme which writes:
\begin{equation}\label{scheme-interpo_LP}
		d X^h_t=d Z_t+b_h(U_{\lfloor t/h\rfloor},X^h_{\tau^h_t}) d t,
\end{equation}
where $\tau_t^h=t_k $ for $t\in [t_k,t_{k+1}) $ stands for the largest discretization time smaller than $t$.

In \cite{jour:meno:24}, \cite{fito:jour:meno:25}, it was as well shown that the time marginals of the SDEs and the scheme admit at positive time densities which satisfy both a Duhamel type representation and some Aronson like estimates. Namely, denoting by $\Gamma(0,x,t,\cdot),\Gamma^h(0,x,t,\cdot)$  the respective densities at time $t$ of the SDE and the Euler scheme starting from $x$ at time 0 the following Duhamel type expansions hold.
\begin{PROP}[Duhamel representation and density estimates for the SDE and the Euler scheme]\label{prop-main-estimates_LP}
		Assume \eqref{COND_KR}. 
                The density $\Gamma(s,x,t,\cdot) $ of the unique weak solution to Equation \eqref{SDE} starting from $x$ at time $s\in [0,T)$ admits the following  Duhamel representation: for all $\textcolor{black}{t\in (s, T],\  y\in \R^d }$,
		\begin{align}
			\Gamma(s,x,t,y)
			= p_\alpha(t-s,y-x)-\int_{s}^{ t}\E_{s,x}\left[b(r,X_r)\cdot\nabla_y  p_\alpha(t-r,y-X_r)\right]d r.\label{duhamel-Diff}
		\end{align} 
Similarly, for the Euler scheme $X^h$ evolving according to \eqref{scheme-interpo_LP}, the conditional distribution of $X_t^h$  given $X_{t_k}^h=x$ with $k\in \llbracket 0,n-1\rrbracket ,\ t\in (t_k,T] $ admits a density $\Gamma^h(t_k,x,t,\cdot) $ which again enjoys a Duhamel type representation: for all $y\in \R^d $,
		\begin{align}
			\Gamma^h(t_k,x,t,y)
			= p_\alpha(t-t_k,y-x)-\int_{t_k}^{ t}\E_{t_k,x}\left[b_h(U_{\lfloor r/h\rfloor},X^h_{\tau_r^h})\cdot\nabla  p_\alpha(t-r,y-X^h_r)\right]d  r.\label{duhamel-scheme}
		\end{align} 
\end{PROP}
It was shown as well that those densities enjoy the following Aronson type upper-bound. 
There exists finite constants $C:=C(T,\alpha,p,r,d)$ s.t. for all $0\le s,t_k<t\le T,\ x,y\in \R^d $,
\begin{align}\label{ARONSON_UPPER_LP}
  \Gamma(s,x,t,y)\le C\bar p_\alpha(t-s,y-x)\mbox{ and  }\Gamma^h(t_k,x,t,y)\le C\bar p_\alpha(t-t_k,y-x),
\end{align}
{with $\bar p_\alpha$ defined in \eqref{eq:barpal}.

The main result of the quoted papers was then that:
\begin{align}
\label{RES_JM_LP}
 | (\Gamma-\Gamma^h)(t_k,x,t,y)|\le Ch^{\frac \gamma \alpha}\bar p_\alpha(t-t_k,y-x), \ \gamma=\alpha-1-\frac dp-\frac{\alpha}{\vartheta},
\end{align}
up to a modification of $C$. The bound in \eqref{RES_JM_LP} in particular made the so-called margin to singularity associated with \eqref{COND_KR}. In other words, in the supremum norm for the densities, the obtained error rate corresponded to the remaining space in the strict inequality \eqref{COND_KR}. We think that \eqref{RES_JM_LP} is somehow \textit{sharp} for the approach considered (Duhamel type expansions). Nevertheless, the analysis below precisely takes advantage of an additional test function in the error analysis through appropriate duality type arguments. This in fact allows to improve the convergence rate from \eqref{RES_JM_LP}. This is precisely the purpose of Theorem \ref{thmsinglebdrift} in the next section. It can be seen as a kind of generalization of a result in \cite{hao:le:ling:24} where a semi-discretization scheme was considered (i.e. the concrete  time discretization was not considered and some boundedness in time was also assumed).

\subsection{Main results and proofs}
We denote by $\prho=\frac\vrho{\vrho-1}\in[1,+\infty]$ the conjugate exponent of $\rho\in[1,+\infty]$.\begin{THM}\label{thmsinglebdrift}
  We suppose that $b\in L^\vartheta([0,T],L^p(\R^d))$ under the Krylov and R\"ockner condition  $$\frac d p+\frac{\alpha}{\vartheta}<\alpha-1\mbox{ and }{p\ge 2},\ \alpha\in (1,2].$$
  For each $\vrho\in[p', p]$ such that $\frac d \vrho+\frac{\alpha}{\vartheta}\le \alpha-1$ 
  % \textcolor{red}{Attention: besoin de $\rho'\le p$, qui n'est pas impliqué par $\frac d\rho+\frac 2 r\le 1$ lorsque $d=1$ mais qui  est impliqué par $\rho\ge 2$}
  , there exists $C<\infty$ such that for each $h=T/n$ with $n\in\N^*$,
  $$\forall (t,x)\in (0,T]\times\R^d,\;\|\Gamma^h(0,x,t,\cdot)-\Gamma(0,x,t,\cdot)\|_{L^{\vrho'}}\le C h^{\frac{\alpha- 1}\alpha-\frac 1\vartheta}t^{-\frac d\alpha(\frac 1\vrho+\frac 1p)}.$$
  \end{THM}
\begin{rem}The larger {$\vartheta$}, the larger the derived order of convergence $\frac{\alpha-1}\alpha-\frac 1\vartheta$ is. On the other hand, the larger $p$, the more flexible is the choice of the index $\rho$ and the smaller is the singularity of the estimation in the time variable $t$.  Note that the choice $\rho=p$ is always possible.
\end{rem}
We will need the following estimations of the Gaussian and stable densities  consequence of \cite[Lemma 2.5]{jour:meno:24} and \cite[Lemma 4]{fito:jour:meno:25} respectively.
\begin{PROP}\label{PROP_INT_LP_HK}
  For $u\ge 1$, $\theta\in\{0,1\}$ and $k\in\N$, there exists $C<\infty$ such that 
  \begin{align}
    \forall t\in(0,T],\;\|\partial_t^\theta\nabla^kp_\alpha(t,\cdot)\|_{L^u}+\|\partial_t^\theta\nabla^k\bar p_\alpha(t,\cdot)\|_{L^u}\le Ct^{-(\theta+\frac k\alpha +\frac d{\alpha u'})}\label{estigradplu}.\end{align}For $k\in\N$, there exists $C<\infty$ such that
  \begin{equation}
   \forall (t,y)\in(0,T]\times \R^d,\;|\nabla^kp_\alpha(t,{y})|\le Ct^{-\frac k\alpha}\bar p_\alpha(t,y)\label{majogradp2}.
  \end{equation}

  For $k\in\{0,1\}$, there exists $C<\infty$ such that \begin{align}
\mbox{for }0<s<t\le T\mbox{ and }y\in\R^d,\;|\nabla^kp_\alpha(t,y)-\nabla^kp_\alpha(s,y)|\le C\frac{(t-s)\wedge s}{s^{1+\frac k \alpha}}(1+(t/s)^{\frac d\alpha})\bar p_\alpha(t,y).\label{estidifftempsgaus}
  \end{align}
 
\end{PROP}
From \eqref{ARONSON_UPPER_LP}, we derive the following result.
\begin{cor}
   Let $u\ge 1$. There exists $C<\infty$ such that $0\le s<t\le T,\ x\in \R^d$,
\begin{align}\|\Gamma(s,x,t,\cdot)\|_{L^u}\le C (t-s)^{-\frac d{\alpha u'}}.\label{estitranslu}\end{align}
There exists $C<\infty$ such that for $h=\frac{T}{n},\ n\in \mathbb N^*$, $0\le t_k<t\le T,\ x\in \R^d$
\begin{align}\|\Gamma^h(t_k,x,t,\cdot)\|_{L^u}\le C (t-t_k)^{-\frac d{\alpha u'}}.\label{estitranslusch}\end{align}\end{cor}

We rely on the following decomposition of the discretization error which takes into account that $b_h$ vanishes on $[0,h)\times\R^d$ :\begin{align}
 &\Gamma^h(0,x,t,y)-\Gamma(0,x,t,y)=\Delta^1_t(y)+\Delta^2_t(y)+\Delta^3_t(y)+\Delta^4_t(y)+\Delta^5_t(y)+\Delta^6_t(y)
   \mbox{ where }\label{decomperrt}\\
 &\Delta^1_t(y)=\1_{\{t\ge 3h\}}\int_h^{\tau^h_t-h} ds  \int_{\R^d}[\Gamma(0,x,\tau^h_s,z)-\Gamma^h(0,x,\tau^h_s,z)] b_h(s,z) \cdot \nabla p_\alpha(t-s,y-z)dz,\notag
  \\ &\Delta^2_t(y)=\1_{\{t\ge 3h\}}\int_{h}^{\tau^h_t-h} ds\int_{\R^d}[\Gamma(0,x,s,z)-\Gamma(0,x,\tau^h_s,z)]b_h(s,z)\cdot\nabla p_\alpha(t-s,y-z) dz,\notag\\&\Delta^3_t(y)=\1_{\{t\ge 3h\}}\int_{h}^{\tau^h_t-h}ds\E\bigg[b_h(U_{\lfloor s/h\rfloor},X_{\tau^h_s}^h)\cdot (\nabla p_\alpha(t-s,y-X_{\tau^h_s}^h)-\nabla p_\alpha(t-s,y-X_s^h))\bigg],\notag
  \\ &\Delta^4_t(y)=\1_{\{t\ge 3h\}}\int_{h}^{\tau^h_t-h}ds\E\bigg[b_h(U_{\lfloor s/h\rfloor},X_{\tau^h_s}^h)\cdot (\nabla p_\alpha(t-U_{\lfloor s/h\rfloor},y-X_{\tau^h_s}^h)-\nabla p_\alpha(t-s,y-X_{\tau^h_s}^h))\bigg],\notag\\
 &\Delta^5_t(y)=\int_{(\tau^h_t-h)\vee 0}^tds\E\bigg[ b(s,X_s)\cdot\nabla p_\alpha(t-s,y-X_s)-b_h(U_{\lfloor s/h\rfloor},X_{\tau^h_s}^h))\cdot\nabla p_\alpha(t-s,y-X_{s}^h))\bigg],\notag\\
  &\Delta^6_t(y)=\1_{\{t\ge 2h\}}\int_{0}^{\tau^h_t-h}ds\int_{\R^d}\Gamma(0,x,s,z)(b(s,z)-b_h(s,z))\cdot\nabla p_\alpha(t-s,y-z)dz.\notag\end{align}

In view of this decomposition, we see that we will have to estimates the $L^{\vrho'}(\R^d)$ norm of terms of the form $\int_{\R^d}\gamma(z)\beta(r,z).\nabla p_\alpha(t-s,\cdot-z)dz$ with $\gamma$ typically equal to some transition density term, $r\in[0,T]$ and  $\beta$ equal to $b$ or $b_h$ and therefore such that $\beta(r,\cdot)$ belongs to $L^p(\R^d)$. According to Young's and H\"older's inequalities, \eqref{estigradplu} and $|b_h|\le|b|$, for  $u,v\ge 1$ such that $\frac 1u+\frac 1v=\frac 1{\prho}+\frac 1{p'}$,
%%%%% De S. pour S.: Détail du Young Hölder:
%%% On part de \rho ' et le v de la norme du gradient est fixé. D'où $1+1/rho'=1/e+1/p $ où e est l'exposant qui va avoir le produit, $1/e=1/u+1/v $
%%% et donc en effet $1/p'+1/\rho'=1/u+1/v $
\begin{align}
   &\left\|\int_{\R^d}\gamma(z)\beta(r,z).\nabla p_\alpha(t-s,\cdot-z)dz\right\|_{L^\prho}\notag\\
   \le& \|\gamma\|_{L^u}\|\beta(r,\cdot)\|_{L^p}\|\nabla p_\alpha(t-s,\cdot)\|_{L^v}\le \|\gamma\|_{L^u}\|b(r,\cdot)\|_{L^p}(t-s)^{-(\frac 1\alpha+\frac d{\alpha v'})}.\label{estuv}
\end{align}
We will mainly use this inequality with $(u,v)=(\prho,p')$ (typically when $s\le \frac t2$) and with $(u,v)=(\frac{p\prho}{p-\prho},1)$ (typically when $s\ge \frac t2$).

Let us first check that the estimation  
\begin{equation}
   \exists C_{(2-6)}<\infty,\;\forall t\in[0,T],\;\sum_{i=2}^6\|\Delta^i_t(\cdot)\|_{L^\prho}\le C^{\frac{\vartheta-1} \vartheta}_{(2-6)}h^{\frac{\alpha-1}\alpha-\frac 1\vartheta}t^{-\frac d\alpha(\frac 1\vrho+\frac 1p)},\label{esti2-6}
 \end{equation} that we are next going to establish is enough to conclude the proof {by} a discrete Gronwall argument. For $t\ge h$, we have $\|\Gamma(0,x,t,\cdot)-\Gamma^h(0,x,t,\cdot)\|_{L^\prho}\le \|\Delta^1_t(\cdot)\|_{L^\prho}+Ct^{-\frac d\alpha(\frac 1\vrho+\frac 1p)}h^{\frac{\alpha-1}\alpha-\frac 1\vartheta}$ and by \eqref{estuv} applied with $(u,v)=(\prho,p')$ then H\"older's inequality, we get
\begin{align}
  \|\Delta^1_t(\cdot)\|_{L^\prho}
                          &\le \int_h^t ds \|\Gamma(0,x,\tau^h_s,\cdot)-\Gamma^h(0,x,\tau^h_s,\cdot)\|_{L^\prho}\|b(s,\cdot)\|_{L^p}(t-s)^{-(\frac 1\alpha+\frac d{\alpha p})}ds\notag\\
  &\le C\|b\|_{L^\vartheta-L^p}\left(\int_h^t\|\Gamma(0,x,\tau^h_s,\cdot)-\Gamma^h(0,x,\tau^h_s,\cdot)\|^{\frac{\vartheta}{\vartheta-1}}_{L^\prho}(t-s)^{-(\frac 1\alpha+\frac{d}{\alpha p})\times\frac{\vartheta}{\vartheta-1}}ds\right)^{\frac{\vartheta-1}{\vartheta}}.\label{estidel1}
\end{align}By \eqref{estitranslu} and \eqref{estitranslusch}, the function $$f(t)=\left(t^{\frac d \alpha(\frac 1\vrho+\frac 1 p)}\|\Gamma(0,x,t,\cdot)-\Gamma^h(0,x,t,\cdot)\|_{L^{\prho}}\right)^{\frac{\vartheta}{\vartheta-1}}$$ is bounded on $(0,T]$ by some constant not depending on $h$.
 We deduce from \eqref{decomperrt}, \eqref{esti2-6}, \eqref{estidel1} and the inequality $
{s}/{\tau^h_s}<2$ valid for $s\ge h$ that 
\begin{align}
   f(t)\le C_{(2-6)}h^{(\frac{\alpha-1}\alpha-\frac 1\vartheta)\times\frac{\vartheta}{\vartheta-1}}+Ct^{\frac d\alpha(\frac 1\vrho+\frac 1p)\times\frac{\vartheta}{\vartheta-1}}\int_h^tf(\tau^h_s)s^{-\frac d \alpha(\frac 1\vrho+\frac 1p)\times\frac{\vartheta}{\vartheta-1}}(t-s)^{-(\frac 1\alpha+\frac{d}{\alpha p})\times\frac{\vartheta}{\vartheta-1}}ds.
  \label{majopregron}\end{align}
Note that since $\frac dp+\frac{\alpha}{\vartheta}<\alpha-1$ and $\frac d\vrho+\frac{\alpha}{\vartheta}\le \alpha-1$, $\frac 1\alpha+\frac d{\alpha p}+\frac 1\vartheta<1$ and $\frac d 2(\frac 1\vrho+\frac 1 p)+\frac{\alpha}{\vartheta}<\alpha-1$ so that $(\frac 1\alpha+\frac{d}{\alpha p})\times\frac{\vartheta}{\vartheta-1}<1$ and $\frac d \alpha (\frac 1\vrho+\frac 1 p)+\frac 2{\vartheta}<2\times \frac{\alpha-1}\alpha\le 1$ which implies ${\frac d \alpha (\frac 1\vrho+\frac 1 p)\times\frac{\vartheta}{\vartheta-1}<1}$. The statement of Theorem \ref{thmsinglebdrift} then follows from the next lemma, the proof of which is postponed at the end of the section.

\begin{lem}[{Gronwall type lemma}]\label{lemgron}
  Let $f:[0,T]\to\R_+$ be bounded and such that for $h=T/N$ with $N\in\N$,
  \begin{equation}
   \forall t\in[0,T],\;f(t)\le \kappa +\lambda t^{a_1}
   \int_h^{h\vee t} f(\tau^h_s)s^{-a_1}(t-s)^{-a_2}ds,\label{eq:gron}
 \end{equation}
 where $\kappa,\lambda\in [0,+\infty)$ and $a_1,a_2\in {[0,1)}$. Then $\sup_{t\in[0,T]}f(t)\le C\kappa$ for some constant $C(\lambda,a_1,a_2,T)$ not depending on $h$.
\end{lem}

Let us now prove \eqref{esti2-6} by successively estimating $\Delta^i_t(y)$ for $i\in\{6,5,4,3,2\}$. {We start with the term $\Delta^6_t(y)$ associated with the truncation of the drift}.
When $t< 2h$, then $\Delta^6_t(y)=0$. Otherwise, $\tau^h_t-h\ge h$ and we have $\Delta^6_t(y)=\Delta^{61}_t(y)+\Delta^{62}_t(y)$ where \begin{align}
\Delta^{61}_t(y)&=\int_0^h \int_{\R^d } \Gamma(0,x,s,z) b(s,z)\cdot \nabla p_\alpha(t-s,y-z) ds, \notag\\
\Delta^{62}_t(y)&= \int_h^{\tau^h_t-h} \int_{\R^d } \Gamma(0,x,s,z) (b(s,z)-b_h(s,z))\cdot \nabla p_\alpha(t-s,y-z) ds.\label{EXP_DENS_DELTA_2}
                                                                 \end{align}
                                                                                                                                                                                
                                                                                                                                                                                   By \eqref{estuv} with $(u,v)=(p',\prho)$ and \eqref{estitranslu} then H\"older's inequality and the inequalities $\frac d{\alpha p}\times\frac{\vartheta}{\vartheta-1}<1$ and $\frac d{\alpha p}<\frac{\alpha-1}\alpha$ consequence of \eqref{COND_KR}, we get that for $t\ge 2h$
                                                                                                                                             
\begin{align*}
  \|\Delta^{61}_t(\cdot)\|_{L^{\prho}}&\lesssim\int_0^h \|\Gamma(0,x,s,\cdot)\|_{L^{p'}}\|b(s,\cdot)\|_{L^p}(t-s)^{-(\frac 1\alpha+\frac d{\alpha\vrho})}ds\lesssim t^{-(\frac 1\alpha+\frac d{\alpha \vrho})}\int_0^h s^{-\frac d{\alpha p}}\|b(s,\cdot)\|_{L^p}ds\\&\lesssim t^{-(\frac 1\alpha+\frac d{\alpha\vrho})}\|b\|_{L^\vartheta-L^p}\left(\int_0^h s^{-\frac d{\alpha p}\times\frac{\vartheta}{\vartheta-1}}ds\right)^{\frac{\vartheta-1}{\vartheta}}\\
                                   &\lesssim\|b\|_{L^\vartheta-L^p}t^{-(\frac 1\alpha+\frac d{\alpha\vrho})}{h^{1-(\frac d{\alpha p}+\frac 1\vartheta)}} \le {C\|b\|_{L^\vartheta-L^p}h^{\frac{\alpha-1}{\alpha}-\frac 1\vartheta} t^{-\frac d{\alpha }(\frac 1{\vrho}+\frac 1 p)}},
\end{align*}
where we used $d/p\le 1 $ (and therefore $h^{-\frac 1\alpha} \ge h^{-\frac{d}{\alpha p}}$) for the last inequality.
           On the other hand, 
$$|b(s,z)-b_h(s,z)|\le |b(s,z)|\1_{\{|b(s,z)|\ge Bh^{-\frac{\alpha-1}\alpha}\}}\le \frac{h^{\frac {\alpha-1}\alpha}}{B} |b(s,z)|^2,$$ so that $\|b(s,\cdot)-b_h(s,\cdot)\|_{L^{p/2}}\le \frac{h^{\frac {\alpha-1}\alpha}}{B}\|b(s,\cdot)\|_{L^p}^2$.  With \eqref{estitranslu}, Young's and H\"older's inequalities, we deduce that for $u,v\ge 1$ such that $\frac 1 u+\frac 1 v=1+\frac 1\prho-\frac 2 p$,

\begin{align*}
\|\Delta^{62}_t(\cdot)\|_{L^{\vrho'}}&\lesssim h^{\frac {\alpha-1}\alpha}\int_h^{\tau^h_t-h} \|\Gamma(0,x,s,\cdot)\|_{L^u}\|b(s,\cdot)\|^{2}_{L^p}\|\nabla p_\alpha(t-s,\cdot)\|_{L^v}ds\\&\lesssim h^{\frac {\alpha-1}\alpha}\int_h^{\tau^h_t-h} s^{-\frac d{\alpha u'}}\|b(s,\cdot)\|^{2}_{L^p}(t-s)^{-(\frac 1\alpha+\frac d{\alpha v'})}ds\\&\lesssim h^{\frac {\alpha-1}\alpha}\|b\|^{2}_{L^\vartheta-L^p}\left(\int_h^{\tau^h_t-h} s^{-\frac d{\alpha u'}\times \frac {\vartheta}{\vartheta-2}}(t-s)^{-(\frac 1\alpha+\frac d{\alpha v'})\times \frac {\vartheta}{\vartheta-2}}ds\right)^{\frac{\vartheta-2}{\vartheta}}
\end{align*}
Choosing $u=(p/2)'$ and $v=\prho$ and using that, by \eqref{COND_KR} and $\alpha \in (1,2]$, $\frac {2d }{\alpha p}\times \frac r{\vartheta-2}<1$, we get for $t\ge 2h$,
\begin{align*}
  \|\Delta^{62}_t(\cdot)\|_{L^{\vrho'}}&\lesssim h^{\frac {\alpha-1}\alpha}\|b\|^{2}_{L^\vartheta-L^p}\left(\int_h^{\tau^h_t-h} s^{-\frac {2d}{\alpha p}\times \frac {\vartheta}{\vartheta-2}}(t-s)^{-(\frac 1\alpha+\frac d{\alpha\vrho})\times \frac {\vartheta}{\vartheta-2}}ds\right)^{\frac{\vartheta-2}{\vartheta}}\\
  &\lesssim h^{\frac {\alpha-1}\alpha }\|b\|^{2}_{L^\vartheta-L^p}\left(\1_{\{\frac d{\vrho}+\frac{2\alpha }\vartheta<\alpha-1\}}t^{-\frac{2d}{\alpha p}+(\frac{\alpha-1}\alpha-\frac{d}{\alpha \vrho}-\frac{2}\vartheta)}+\1_{\{\frac d{\vrho}+\frac{2\alpha }\vartheta=\alpha-1\}}t^{-\frac{2d}{\alpha p}}(\ln(T/h))^{\frac{\vartheta-2}\vartheta}\right.\\
  &\phantom{\lesssim h^{\frac {\alpha-1}\alpha }\|b\|^{2}_{L^\vartheta-L^p}(}\left.+\1_{\{\frac d{\vrho}+\frac{2\alpha }\vartheta>\alpha-1\}}t^{-\frac{2d}{\alpha p}}h^{\frac {\alpha-1}\alpha-(\frac d{\alpha\vrho}+\frac {2}\vartheta)}\right).
\end{align*}
Since $\frac {\alpha-1}\alpha-(\frac d{\alpha\rho}+\frac 1\vartheta)\ge 0$ and $\frac 2p\le (\frac 1\vrho+\frac 1p)$, we deduce that $\|\Delta^{62}_t(\cdot)\|_{L^\prho}\lesssim h^{\frac{\alpha-1}\alpha-\frac 1\vartheta}t^{-\frac d\alpha(\frac 1\rho+\frac 1p)}$ but in the case $\vartheta=\infty$ and $\frac d{\rho}=\alpha-1$ when we have some additional factor $\ln(T/h)$. To get rid of this  logarithmic factor, it is enough to choose $\frac 1{u'}=\frac{1-\tilde \lambda}\vrho+\frac{1+\tilde \lambda}p$ and $\frac 1{v'}=\frac{\tilde \lambda}\vrho+\frac{1-\tilde \lambda}{p}$ for some $\tilde \lambda\in[0,1)$, which is possible when $p>2$ since then $\frac{1-\tilde \lambda}\vrho+\frac{1+\tilde \lambda}p$ goes to $\frac 2p<1$ as $\tilde \lambda\to 1-$. Note that we cannot have $p=2$ and $\frac d{\vrho}=\alpha-1$, since, by \eqref{COND_KR}, the first equality implies $\alpha>1+d/2$ which is only feasible if $d=1$ and $\alpha>3/2 $, so that $\rho=1/(\alpha-1)<2$ by the second equality, which contradicts the condition $\rho\ge p'=2$ that we suppose.

With the above estimations of $\|\Delta^{61}_t(\cdot)\|_{L^{\prho}}$, we conclude that
$$\forall t\in[0,T],\;\|\Delta^{6}_t(\cdot)\|_{L^{\prho}}\lesssim h^{\frac{\alpha -1}{\alpha}-\frac 1\vartheta}t^{-\frac d\alpha(\frac 1\rho+\frac 1p)}.$$

Let us turn to $\Delta^{5}_t(y) $, error associated with the last time steps. Using that $b_h$ vanishes on $[0,h)\times\R^d$, we get that $\Delta^5_t(y)=\Delta^{51}_t(y)+\Delta^{52}_t(y)+\Delta^{53}_t(y)$ with
\begin{align*}
   \Delta^{51}_t(y)&=\1_{\{t<2h\}}\int_{0}^{t}ds\int_{\R^d}\Gamma(0,x,s,z)b(s,z)\cdot\nabla p_\alpha(t-s,y-z)dz,\\\Delta^{52}_t(y)&=\1_{\{t\ge 2h\}}\int_{\tau^h_t-h}^tds\int_{\R^d}\Gamma(0,x,s,z)b(s,z)\cdot\nabla p_\alpha(t-s,y-z)dz,\\\Delta^{53}_t(y)&=-\frac{\1_{\{t>h\}}} h\int_{(\tau^h_t-h)\vee h}^t ds\int_{\tau^h_s}^{\tau^h_s+h}dr\E\left[b_h(r,X^h_{\tau_s^h})\cdot\nabla p_\alpha(t-s,y-X^h_s)\right].
\end{align*}
By \eqref{estuv} with $(u,v)=(\prho,p')$ and \eqref{estitranslu}, we have using that $\frac d{\alpha\vrho}\times\frac{\vartheta}{\vartheta-1}\le \frac{\alpha-1}\alpha<1$ since $\frac d\vrho+\frac{\alpha} r\le \alpha-1$ and $(\frac 1\alpha+\frac d{\alpha p})\frac{\vartheta}{\vartheta-1}<1$ by \eqref{COND_KR}
\begin{align*}
   \|\Delta^{51}_t(\cdot)\|_{L^\prho}&\lesssim \1_{\{t<2h\}}\int_0^{t}\|\Gamma(0,x,s,\cdot)\|_{L^\prho}\|b(s,\cdot)\|_{L^p}(t-s)^{-(\frac 1\alpha+\frac d{\alpha p})}ds\\&\lesssim \1_{\{t<2h\}}\|b\|_{L^\vartheta-L^p}\left(\int_0^{t}s^{-\frac d{\alpha \vrho}\times\frac{\vartheta}{\vartheta-1}}(t-s)^{-(\frac 1\alpha +\frac d{\alpha p})\frac{\vartheta}{\vartheta-1}}ds\right)^{\frac{\vartheta-1}{\vartheta}}\\&\lesssim\1_{\{t<2h\}}\|b\|_{L^\vartheta-L^p}t^{\frac {\alpha-1}\alpha-\frac 1\vartheta-\frac d \alpha(\frac 1\vrho+\frac 1 p)}\lesssim \|b\|_{L^\vartheta-L^p}h^{\frac{\alpha-1}\alpha-\frac 1\vartheta}t^{-\frac d \alpha(\frac 1\vrho+\frac 1 p)}.
\end{align*}

For $t\in[h,T]$ and $s\in[h,t]$, we have using \eqref{ARONSON_UPPER_LP} and \eqref{majogradp2}
 \begin{align*}
  \E\left[\left|b_h(r,X^h_{\tau_s^h})\cdot\nabla p_\alpha(t-s,y-X^h_s))\right|\right]&=\int_{\R^d\times\R^d}\Gamma^h(0,x,\tau^h_s,z)|b_h(r,z)|\Gamma^h(\tau^h_s,z,s,w)|\nabla p_\alpha(t-s,y-w)|dwdz\notag\\&\le
C\int_{\R^d\times\R^d}\Gamma^h(0,x,\tau^h_s,z)|b_h(r,z)|\bar{p}_\alpha(s-\tau^h_s,w-z)\frac{\bar{p}_\alpha(t-s,y-w)}{(t-s)^{\frac 1\alpha}}dwdz\notag\\&=
C\int_{\R^d}\Gamma^h(0,x,\tau^h_s,z)|b_h(r,z)|\frac{\bar{p}_\alpha(t-\tau^h_s,y-z)}{(t-s)^{\frac 1\alpha}}dz.\notag
\end{align*}
With Young's and H\"older's inequalities, {i.e. reasoning like in the derivation of \eqref{estuv} with $(u,v)=(\frac{p\prho}{p-\prho},1)$}, and \eqref{estitranslusch}, we deduce that
\begin{align*}
   \left\|\E\left[\left|b_h(r,X^h_{\tau_s^h})\cdot\nabla p_\alpha(t-s,\cdot-X^h_s))\right|\right]\right\|_{L^\prho}&\lesssim \|\Gamma^h(0,x,\tau^h_s,\cdot)\|_{L^{\frac{p\prho}{p-\prho}}}\|b(r,\cdot)\|_{L^p}\|\bar{p}_\alpha(t-\tau^h_s,\cdot)\|_{L^1}(t-s)^{-\frac 1\alpha}\\&\lesssim (\tau^h_s)^{-\frac d\alpha(\frac 1\rho+\frac 1 p)}\|b(r,\cdot)\|_{L^p}(t-s)^{-\frac 1\alpha}.
\end{align*}
Therefore, we have  using that $(\tau^h_t-h)\vee h> \frac t3$ 
\begin{align*}
  \|\Delta^{53}_t(\cdot)\|_{L^\prho}&\lesssim \frac{\1_{\{t>h\}}} h\int_{(\tau^h_t-h)\vee h}^t (\tau^h_s)^{-\frac d\alpha(\frac 1{\vrho}+\frac 1p)}(t-s)^{-\frac 1 \alpha}ds\int_{\tau^h_s}^{\tau^h_s+h} \|b(r,\cdot)\|_{L^p}dr\le \frac C h t^{-\frac d\alpha(\frac 1{\vrho}+\frac 1p)}h^{\frac{\alpha-1}\alpha}\|b\|_{L^\vartheta-L^p}h^{1-\frac 1\vartheta}\\
  &\lesssim  h^{\frac{\alpha-1}\alpha-\frac 1\vartheta}t^{-\frac d\alpha(\frac 1{\vrho}+\frac 1p)}\|b\|_{L^\vartheta-L^p}.
\end{align*}On the other hand, by \eqref{estuv} with $(u,v)=(\frac{p\prho}{p-\prho},1)$ and \eqref{estitranslu}, we get using that $\tau^h_t-h>t/3 $ when $t\ge 2h$ 
\begin{align*}
  \|\Delta^{52}_t(\cdot)\|_{L^{\prho}}&\lesssim \1_{\{t\ge 2h\}}\int_{\tau^h_t-h}^ts^{-\frac d\alpha(\frac 1\rho+\frac 1p)}\|b(s,\cdot)\|_{L^p}(t-s)^{-\frac 1\alpha}ds\lesssim t^{-\frac d\alpha(\frac 1\rho+\frac 1p)}\|b\|_{L^\vartheta-L^p}\left(\int_{\tau^h_t-h}^t(t-s)^{-\frac 1\alpha\times\frac{\vartheta}{\vartheta-1}}\right)^{\frac{\vartheta-1}{\vartheta}}\\
                                      &\lesssim \|b\|_{L^\vartheta-L^p}h^{\frac{\alpha-1}\alpha-\frac 1\vartheta} t^{-\frac d\alpha (\frac 1\rho+\frac 1p)}.\end{align*}With the above estimations of $\|\Delta^{51}_t(\cdot)\|_{L^\prho}$ and $\|\Delta^{53}_t(\cdot)\|_{L^\prho}$, we conclude that 
                                      $$\|\Delta^{5}_t(\cdot)\|_{L^\prho}\lesssim \|b\|_{L^\vartheta-L^p}h^{\frac{\alpha-1}\alpha-\frac 1\vartheta}t^{-\frac d\alpha(\frac 1{\vrho}+\frac 1p)}.$$

For $i\in\{2,3,4\}$, since $\Delta^i_t(y)=0$ when $(t,y)\in[0,3h)\times\R^d$, it is enough to estimate $\|\Delta^i_t(\cdot)\|_{L^\prho}$ when $t\ge 3h$. We therefore suppose from now on that $t\ge 3h$, which implies that $h\le \tau^h_{t/2}\le \tau^h_t-h$ and $\tau^h_{t/2}>t/4$.  
By \eqref{estidifftempsgaus}, for $r\in[0,\tau^h_t-h)$ and $s\in[\tau^h_r,\tau^h_r+h)$ so that $\frac{t-r\wedge s}{t- r\vee s}\le 2$, we have
\begin{align*}
  &|\nabla p_\alpha(t-r,\cdot)-\nabla p_\alpha(t-s,\cdot)|\le C h(t-r\vee s)^{-(1+\frac 1\alpha)}\bar p_\alpha(t-r\wedge s,\cdot)\\&\mbox{ and }\|\nabla p_\alpha(t-r,\cdot)-\nabla p_\alpha(t-s,\cdot)\|_{L^1}\le C  h(t-r\vee s)^{-(1+\frac 1\alpha)}.
\end{align*}
Therefore, using \eqref{estitranslusch}, $(t-r\vee s)^{-(1+1/\alpha)}\le 2^{1+1/\alpha}(t-r)^{-(1+1/\alpha)}$ and {$r\le 2\tau_r^h $}
 for $r\in[h,\tau^h_t-h)$ and $s\in[\tau^h_r,\tau^h_r+h]$, we get, {reasoning like in the derivation of \eqref{estuv} with $(u,v)=(\frac{p\prho}{p-\prho},1)$ and using \eqref{estitranslu}}, we get for the term $\Delta^{4}_t(\cdot)$ involving the time sensitivity of the stable heat kernel,
\begin{align*}
  \|\Delta^{4}_t(\cdot)\|_{L^\prho}&\le\frac 1h\int_{h}^{\tau^h_t-h}dr\int_{\tau^h_r}^{\tau^h_r+h}ds\|\Gamma^h(0,x,\tau^h_r,\cdot)\|_{L^{\frac{p\prho}{p-\prho}}}\|b(r,\cdot)\|_{L^p}\|\nabla p_\alpha(t-r,\cdot)-\nabla p_\alpha( t-s,\cdot)\|_{L^1}\\
                                   &\lesssim h \int_{h}^{\tau^h_t-h}r^{-\frac d\alpha (\frac 1\rho+\frac 1 p)}\|b(r,\cdot)\|_{L^p}(t-r)^{-(1+\frac 1\alpha)}dr\\&\lesssim h\left(t^{-(1+\frac 1\alpha)}\int_{h}^{\tau^h_{t/2}}r^{-\frac d\alpha (\frac 1\rho+\frac 1 p)}\|b(r,\cdot)\|_{L^p}dr+ t^{-\frac d\alpha(\frac 1{\vrho}+\frac 1p)}\int_{\tau^h_{t/2}}^{\tau^h_t-h}\|b(r,\cdot)\|_{L^p}(t-r)^{-(1+\frac 1\alpha)}dr\right)\\&\lesssim h\|b\|_{L^\vartheta-L^p}\left(t^{-(1+\frac 1\alpha)}\left(\int_{h}^{\tau^h_{t/2}}r^{-\frac d\alpha (\frac 1\rho+\frac 1 p)\frac{\vartheta}{\vartheta-1}}dr\right)^{\frac{\vartheta-1}{\vartheta}}+ t^{-\frac d\alpha(\frac 1{\vrho}+\frac 1p)}\left(\int_{\tau^h_{t/2}}^{\tau^h_t-h}(t-r)^{-(1+\frac 1\alpha)\times \frac{\vartheta}{\vartheta-1}}dr\right)^{\frac{\vartheta-1}{\vartheta}}\right)\\&\lesssim h\|b\|_{L^\vartheta-L^p}\left(t^{-(\frac 1\alpha+\frac 1\vartheta+\frac d\alpha(\frac 1\rho+\frac 1 p))}+ t^{-\frac d\alpha(\frac 1{\vrho}+\frac 1p)}h^{-(\frac 1\alpha+\frac 1\vartheta)}\right)\lesssim \|b\|_{L^\vartheta-L^p}h^{\frac {\alpha-1}\alpha-\frac 1\vartheta}t^{-\frac d\alpha(\frac 1{\vrho}+\frac 1p)}.
\end{align*}

The estimation of the term $\Delta^3_t(\cdot)$ {associated with the local transition of the Euler scheme on one time step} relies on the next lemma {inspired from \cite[Lemma 6.4]{hao:le:ling:24} or \cite[Lemma 4.5]{hao:le:ling:26}} and the proof of which is postponed.

\begin{lem}\label{lemfdg}For $t\in [h,T]$,
   \begin{align}
   \E[|f(X^h_{\tau^h_t})(g(X^h_t)-g(X^h_{\tau^h_t}))|]\le C h^{\frac 1\alpha}(\tau^h_t)^{-\frac d{\alpha p}}\|f\|_{L^p}\|\nabla g\|_{L^\infty}.\label{efdg1}
\end{align}
\begin{align}
   |\E[f(X^h_{\tau^h_t})(g(X^h_t)-g(X^h_{\tau^h_t}))]|\le C\|f\|_{L^p}\left(h^{1-\frac 1\vartheta}(\tau^h_t)^{-\frac{2d}{\alpha p}}\|b\|_{L^\vartheta-L^p}\|\nabla g\|_{L^\infty}+h(\tau^h_t)^{-\frac d\alpha(\frac 1{\vrho}+\frac 1 p)}\|\Delta^{\frac \alpha 2} g\|_{L^{\vrho}}\right)
 ,\label{efdg2}
\end{align}
where $\Delta^{\frac\alpha 2}:=-(-\Delta)^{\frac \alpha 2} $ denotes the generator of the driving process $Z$.
                                                                       \end{lem}

We have $\Delta^3_t(y)=\Delta^{31}_t(y)+\Delta^{32}_t(y)$ where
\begin{align*}
   \Delta^{31}_t(y)&=\frac 1h\int_{h}^{\tau^h_{t/2}} ds\int_{\tau^h_s}^{\tau^h_s+h}dr\E\left[b_h(r,X^h_{\tau_s^h})\cdot(\nabla p_\alpha(t-s,y-X^h_{\tau^h_s})-\nabla p_\alpha(t-s,y-X^h_s))\right],\\
  \Delta^{32}_t(y)&                 =\frac 1 h\int_{\tau^h_{t/2}}^{\tau^h_t-h} ds\int_{\tau^h_s}^{\tau^h_s+h}dr\E\left[b_h(r,X^h_{\tau_s^h})\cdot(\nabla p_\alpha(t-s,y-X^h_{\tau^h_s})-\nabla p_\alpha(t-s,y-X^h_s))\right].\end{align*}
Let $\varphi\in L^\vrho(\R^d)$ be {an arbitrary} test function. By \eqref{efdg1} applied with $f=b_h(r,\cdot)$ and $g=\nabla p_\alpha(t-s,\cdot)\star\varphi$ and \eqref{estigradplu}, 
we have
\begin{align*}
   &\left|\int_{\R^d}\E\left[b_h(r,X^h_{\tau_s^h})\cdot(\nabla p_\alpha(t-s,y-X^h_s)-\nabla p_\alpha(t-s,y-X^h_{\tau^h_s}))\right]\varphi(y)dy\right|\\&\hspace{1cm}=\left|\int_{\R^d}\E\left[b_h(r,X^h_{\tau_s^h})\cdot(\nabla p_\alpha(t-s,\cdot)\star\varphi(X^h_{\tau^h_s})-\nabla p_\alpha(t-s,\cdot)\star\varphi(X^h_{s}))\right]\right|\\&\hspace{1cm}\lesssim h^{\frac 1\alpha}(\tau^h_s)^{-\frac d{\alpha p}}\|b(r,\cdot)\|_{L^p}\|{\nabla^2} p_\alpha(t-s,\cdot)\|_{L^\prho}\|\varphi\|_{L^\vrho}\lesssim h^{\frac 1\alpha }(\tau^h_s)^{-\frac d{\alpha p}}\|b(r,\cdot)\|_{L^p}(t-s)^{-({\frac{2}{\alpha }}+\frac d{\alpha \vrho})}\|\varphi\|_{L^\vrho}.
\end{align*}
Since $\varphi$ is arbitrary, we deduce that
\begin{align}
   \left\|\E\left[b_h(r,X^h_{\tau_s^h})\cdot(\nabla p_\alpha(t-s,\cdot-X^h_{\tau^h_s})-\nabla p_\alpha(t-s,\cdot-X^h_s))\right]\right\|_{L^\prho}\lesssim h^{\frac 1\alpha}(\tau^h_s)^{-\frac d{\alpha p}}\|b(r,\cdot)\|_{L^p}(t-s)^{-({\frac{2}\alpha}+\frac d{2\vrho})}.\label{est31}
\end{align}
On the other hand, by \eqref{efdg2} applied with $f=b_h(r,\cdot)$ and $g=\nabla p_\alpha(t-s,\cdot)\star\varphi$ and \eqref{estigradplu},
                                we have for $(r,s)\in [0,T]\times [h,t]$
                                \begin{align*}&\left|\E\left[b_h(r,X^h_{\tau_s^h})\cdot(\nabla p_\alpha(t-s,\cdot)\star\varphi(X^h_{\tau^h_s})-\nabla p_\alpha(t-s,\cdot)\star\varphi(X^h_s))\right]\right|\\&\hspace{1.5cm}\lesssim \|b(r,\cdot)\|_{L^p}\left(h^{1-\frac 1\vartheta}(\tau^h_s)^{-\frac{2d}{\alpha p}}\|b\|_{L^\vartheta-L^p}\|\nabla^2p_\alpha(t-s,\cdot)\star\varphi \|_{L^\infty}+h(\tau^h_s)^{-\frac d\alpha (\frac 1{\vrho}+\frac 1 p)}\|\Delta^{{\frac \alpha 2}}\nabla p_\alpha(t-s,\cdot)\star\varphi\|_{L^{\vrho}}\right)\\&\hspace{1.5cm}\lesssim \|b(r,\cdot)\|_{L^p}\|\varphi\|_{L^{\vrho}}\left(h^{1-\frac 1\vartheta}(\tau^h_s)^{-\frac{2 d}{\alpha p}}\|b\|_{L^\vartheta-L^p}\|\nabla^2p_\alpha(t-s,\cdot) \|_{L^{\prho}}+h(\tau^h_s)^{-\frac d\alpha(\frac 1{\vrho}+\frac 1 p)}\|\partial_t\nabla p_\alpha(t-s,\cdot)\|_{L^1}\right)\\&\hspace{1.5cm}\lesssim \|b(r,\cdot)\|_{L^p}\|\varphi\|_{L^p}\left(h^{1-\frac 1\vartheta}(\tau^h_s)^{-\frac{2d}{\alpha p}}\|b\|_{L^\vartheta-L^p}(t-s)^{-(\frac2\alpha+\frac d{\alpha \vrho})}+h(\tau^h_s)^{-\frac d\alpha(\frac 1{\vrho}+\frac 1 p)}(t-s)^{-(1+\frac 1\alpha)}\right)                  ,\end{align*}
                                                                                 so that
\begin{align}
   &\left\|\E\left[b_h(r,X^h_{\tau_s^h})\cdot(\nabla p_\alpha(t-s,\cdot-X^h_{\tau^h_s})-\nabla p_\alpha(t-s,\cdot-X^h_s))\right]\right\|_{L^\prho}\notag\\&\hspace{3cm}\lesssim \|b(r,\cdot)\|_{L^p}\left(h^{1-\frac 1\vartheta}(\tau^h_s)^{-\frac{2d}{\alpha p}}\|b\|_{L^\vartheta-L^p}(t-s)^{-(\frac{2}{\alpha}+\frac d{\alpha \vrho})}+h(\tau^h_s)^{-\frac d\alpha(\frac 1{\vrho}+\frac 1 p)}(t-s)^{-(1+\frac 1\alpha)}\right).\label{est32}
\end{align}
                                                                                 By \eqref{est31} and Minkowski's inequality,
\begin{align}
          \|\Delta^{31}_t(\cdot)\|_{L^\prho}
          &\lesssim  h^{-1+\frac 1\alpha}\int_{r=h}^{\tau^h_{t/2}}(\tau^h_r)^{-\frac d{\alpha p}}\|b(r,\cdot)\|_{L^p}\int_{s=\tau^h_r}^{\tau^h_r+h}(t-s)^{-(\frac 2\alpha+\frac{d}{\alpha\vrho})}dsdr\notag\\
                                             &\lesssim  h^{\frac 1\alpha}(t-\tau^h_{t/2})^{-(\frac 2\alpha+\frac{d}{\alpha \vrho})}\int_{h}^{\tau^h_{t/2}+h}r^{-\frac d{\alpha p}}\|b(r,\cdot)\|_{L^p}dr\notag\\&\lesssim  h^{\frac 1\alpha }t^{-(\frac 2\alpha+\frac{d}{\alpha\vrho})}\|b\|_{L^\vartheta-L^p}\left(\int_{h}^{\tau^h_{t/2}+h}r^{-\frac d{\alpha p}\times\frac{\vartheta}{\vartheta-1}}dr\right)^{\frac{\vartheta-1}{\vartheta}}\lesssim  h^{\frac 1\alpha}t^{{1-(\frac 2\alpha}+\frac 1\vartheta+\frac{d}\alpha(\frac 1{\vrho}+\frac 1{p}))}\|b\|_{L^\vartheta-L^p}\notag\\
                                    &\lesssim h^{\frac{\alpha-1}\alpha-\frac 1\vartheta}t^{-\frac{d}\alpha(\frac 1{\vrho}+\frac 1{p})}\|b\|_{L^\vartheta-L^p}       \label{estid31}.\end{align}
                                           By \eqref{est32} and the inequality $\int_{\tau^h_s}^{\tau^h_s+h}\|b(r,\cdot)\|_{L^p}dr\le\|b\|_{L^\vartheta-L^p}h^{1-\frac 1\vartheta}$ deduced from H\"older's inequality, we have
\begin{align*}
  \|\Delta^{32}_t(\cdot)\|_{L^\prho}\lesssim& \|b\|_{L^\vartheta-L^p}h^{-\frac 1 \vartheta}t^{-\frac{2d}{\alpha p}}\int_{s=\tau^h_{t/2}}^{\tau^h_t-h}(t-s)^{-(\frac 2\alpha+\frac d{\alpha\vrho})}\int_{r=\tau^h_s}^{\tau^h_s+h}\|b(r,\cdot)\|_{L^p}drds\notag\\&+ t^{-\frac d\alpha(\frac 1{\vrho}+\frac 1 p)}\int_{s=\tau^h_{t/2}}^{\tau^h_t-h}(t-s)^{-(1+\frac 1\alpha)}\int_{r=\tau^h_s}^{\tau^h_s+h}\|b(r,\cdot)\|_{L^p}drds\notag\\
  \lesssim& \|b\|^2_{L^\vartheta-L^p}h^{1-\frac 2 \vartheta}t^{-\frac{2d}{\alpha p}}\int_{\tau^h_{t/2}}^{\tau^h_t-h}(t-s)^{-(\frac 2\alpha+\frac d{\alpha\vrho})}ds+C\|b\|_{L^\vartheta-L^p}h^{1-\frac 1 \vartheta}t^{-\frac d\alpha(\frac 1{\vrho}+\frac 1 p)}\int_{\tau^h_{t/2}}^{\tau^h_t-h}(t-s)^{-(1+\frac 1\alpha)}ds\notag\\
  \lesssim & \|b\|^2_{L^\vartheta-L^p}h^{2-\frac 2\alpha-\frac 2 r-\frac d{\alpha\vrho}}t^{-\frac{2d}{\alpha p}}+C\|b\|_{L^\vartheta-L^p}h^{\frac {\alpha-1}\alpha-\frac 1 \vartheta}t^{-\frac d\alpha(\frac 1{\vrho}+\frac 1 p)}\notag\\
   \lesssim & \|b\|^2_{L^\vartheta-L^p}h^{\frac{\alpha-1}\alpha-\frac 1 r+(1-\frac 1\alpha-\frac 1\vartheta-\frac d{\alpha\vrho})}t^{-\frac{2d}{\alpha p}}+C\|b\|_{L^\vartheta-L^p}h^{\frac {\alpha-1}\alpha-\frac 1 \vartheta}t^{-\frac d\alpha(\frac 1{\vrho}+\frac 1 p)}\notag\\
  \lesssim &\|b\|_{L^\vartheta-L^p}(1+\|b\|_{L^\vartheta-L^p})h^{\frac{\alpha-1}\alpha -\frac 1 \vartheta}t^{-\frac d\alpha(\frac 1{\vrho}+\frac 1 p)}\label{estid32},
\end{align*}
where we used $-\frac 1\rho\le-\frac 1 p$ and $\frac d{\alpha \vrho}+\frac 1 r\le 1- \frac 1\alpha$ for the last inequality.
                   With \eqref{estid31}, we conclude that 
\begin{align}
   \forall t\in[0,T],\;\|\Delta^3_t(\cdot)\|_{L^{\prho}}\lesssim h^{\frac{\alpha-1}\alpha-\frac 1\vartheta}t^{-\frac d\alpha(\frac 1{\vrho}+\frac 1p)}\|b\|_{L^\vartheta-L^p}(1+\|b\|_{L^\vartheta-L^p}).
\end{align}

The estimation of $\|\Delta^2_t(\cdot)\|_{L^\prho}$ relies on that of
$\left\|\Gamma(0,x,t,\cdot)-\Gamma(0,x,\tau^h_t,\cdot)\right\|_{L^{\prho}}$ for $h\le t\le T$ given in the next lemma, the proof of which is postponed.\begin{lem}[Forward time sensitivity in Lebesgue norm for the density of the SDE]\label{lemerrfaibschemtt}
  For $\vrho\ge p'$ such that $\frac d\vrho\le 1$ and $t\in[h,T]$,
  \begin{align}
   \left\|\Gamma(0,x,t,\cdot)-\Gamma(0,x,\tau^h_t,\cdot)\right\|_{L^{\prho}}\le Ch^{\frac {\alpha-1}\alpha-\frac 1\vartheta}t^{-(\frac{\alpha-1}\alpha-\frac 1\vartheta+\frac d{\alpha \vrho})}(1+\|b\|_{L^\vartheta-L^p}).
    \label{errfaibschemtt}
  \end{align}
\end{lem}
Let $u$ be such that $\frac 1 u=\frac 12\left(\frac{1}{\rho'}+\frac{1}{p'}\right)$. Then since $\rho\in[p',p]$, $\frac 1 {u'}=\frac 12\left(\frac{1}{\rho}+\frac{1}{p}\right)\le \frac 1{p'}$ and since $\frac d\rho\le 1$, $\frac d{u'}\le 1$. Using \eqref{estuv} with $v=u$ and $|b_h|\le |b|$ then \eqref{errfaibschemtt} and H\"older's inequality in time, we obtain
\begin{align*}
  \|\Delta^2_t(\cdot)\|_{L^{\prho}}&\lesssim \int_h^{\tau^h_t-h}\|\Gamma(0,x,s,\cdot)-\Gamma(0,x,\tau^h_s,\cdot)\|_{L^{u}}\|b(s,\cdot)\|_{L^p}(t-s)^{-(\frac 1\alpha+\frac{d}{2 \alpha}(\frac 1\rho+\frac 1p))}ds\\
                                      &\lesssim  h^{\frac{\alpha-1}\alpha-\frac 1\vartheta}(1+\|b\|_{L^\vartheta-L^p})\int_h^{\tau^h_t-h}s^{-(\frac{\alpha-1}\alpha-\frac 1\vartheta+\frac d{2 \alpha }(\frac 1\vrho+\frac 1 p))}\|b(s,\cdot)\|_{L^p}(t-s)^{-(\frac 1\alpha+\frac{d}{2\alpha }(\frac 1\rho+\frac 1p))}ds
  \\&\lesssim  h^{\frac{\alpha-1}\alpha-\frac 1\vartheta}(1+\|b\|_{L^\vartheta-L^p})\|b\|_{L^\vartheta-L^p}\left(\int_h^{\tau^h_t-h}s^{-(\frac{\alpha-1}\alpha-\frac 1\vartheta+\frac d{2\alpha}(\frac 1\vrho+\frac 1 p)) \frac{\vartheta}{\vartheta-1}}(t-s)^{-(\frac 1\alpha +\frac{d}{2\alpha }(\frac 1\rho+\frac 1p))\frac{\vartheta}{\vartheta-1} }ds\right)^{\frac{\vartheta-1}{\vartheta}}
  \\& \lesssim  h^{\frac{\alpha-1}\alpha-\frac 1\vartheta}t^{-\frac d\alpha(\frac 1\rho+\frac 1p)}(1+\|b\|_{L^\vartheta-L^p})\|b\|_{L^\vartheta-L^p},
\end{align*}
where we used that \eqref{COND_KR} and $\frac d\vrho+\frac{\alpha}{\vartheta}\le \alpha-1$ imply that 
$(\frac 1\alpha+\frac{d}{2\alpha }(\frac 1\rho+\frac 1p))\frac{\vartheta}{\vartheta-1}<1$.
% \textcolor{red}{Attention besoin de $\frac d{2\vrho}<\frac 12$ pour singularité intégrable en $0$. Perte d'un facteur $\ln(T/h)$ lorsque $r=\infty$ et $\frac d{2\vrho}=\frac 12$. Pour se débarrasser de ce facteur, on choisit $\frac 1 u=\frac 1v=\frac 12\left(\frac{1}{\rho'}+\frac{1}{p'}\right)$. Alors $\frac 1{u'}=\frac 1{v'}=\frac 12\left(\frac{1}{\rho}+\frac{1}{p}\right)$ et on a $\left(\int_h^ts^{-(\frac 12-\frac 1\vartheta+\frac d{4}(\frac 1\vrho+\frac 1p))\times \frac{\vartheta}{\vartheta-1}}(t-s)^{-(\frac 12+\frac d{4}(\frac 1\vrho+\frac 1p))\times\frac{\vartheta}{\vartheta-1} }ds\right)^{\frac{\vartheta-1}{\vartheta}}$ OK car $\frac dp+\frac 2\vartheta<1$ et $\frac d\rho+\frac 2r\le 1\Rightarrow (\frac 12+\frac d{4}(\frac 1\vrho+\frac 1p))\times\frac{\vartheta}{\vartheta-1}<1$. }

\noindent {\bf Proof of Lemma \ref{lemfdg}.}
We have 
\begin{align}
   \E[|f(X^h_{\tau^h_t})(g(X^h_t)-g(X^h_{\tau^h_t}))|]= \int_{\R^d}\Gamma^h(0,x,\tau^h_t,y)|f(y)|\E[|g(y+Z_t-Z_{\tau^h_t}+b_h(U_{\lfloor t/h\rfloor},y)(t-\tau^h_t))-g(y)|]dy.\label{efg}
\end{align}
Since $|b_h|\le Bh^{-{\frac {\alpha-1}\alpha}}$, we have
\begin{align*}
                                                                                      \sup_{y\in\R^d}\E[|g(y+Z_t-Z_{\tau^h_t}+b_h(U_{\lfloor t/h\rfloor},y)(t-\tau^h_t))-g(y)|]&\le \|\nabla g\|_{L^\infty}\sup_{y\in\R^d}\E[|Z_t-Z_{\tau^h_t}+b_h(U_{\lfloor t/h\rfloor},y)(t-\tau^h_t) |]\\&\lesssim h^{\frac 1\alpha}\|\nabla g\|_{L^\infty},
\end{align*}
and we deduce with H\"older's inequality and \eqref{estitranslusch} that
\begin{align*}
   \E[|f(X^h_{\tau^h_t})(g(X^h_t)-g(X^h_{\tau^h_t}))|]|&\lesssim h^{\frac 1\alpha}\|\nabla g\|_{L^\infty}\int_{\R^d}\Gamma^h(0,x,\tau^h_t,y)|f(y)|dy\lesssim h^{\frac 1\alpha}\|\nabla g\|_{L^\infty}\|f\|_{L^p}\|\Gamma^h(0,x,\tau^h_t,\cdot)\|_{L^{p'}}\\&\lesssim h^{\frac 1\alpha}(\tau^h_t)^{-\frac d{\alpha p}}\|f\|_{L^p}\|\nabla g\|_{L^\infty}.
\end{align*}
On the other hand, since
\begin{align*}
  \left|g(y+Z_t-Z_{\tau^h_t}+b_h(U_{\lfloor t/h\rfloor},y)(t-\tau^h_t))-g(y+Z_t-Z_{\tau^h_t})\right|\le \|\nabla g\|_{L^\infty}|b_h(U_{\lfloor t/h\rfloor},y)|h,\end{align*}
and $|b_h|\le|b|$, we have using H\"older's inequality and \eqref{estitranslusch}
\begin{align*}
   &\E[|f(X^h_{\tau^h_t})(g(X^h_{\tau^h_t}+Z_t-Z_{\tau^h_t}+b_h(U_{\lfloor t/h\rfloor},X^h_{\tau^h_t})(t-\tau^h_t))-g(X^h_{\tau^h_t}+Z_t-Z_{\tau^h_t}))|]\\&\hspace{1cm}\le \|\nabla g\|_{L^\infty}\int_{\tau^h_t}^{\tau^h_t+h}\int_{\R^d}\Gamma^h(0,x,\tau^h_t,y)|f(y)||b(s,y)|dyds\\&\hspace{1cm}\le \|\nabla g\|_{L^\infty}\|f\|_{L^p}\|\Gamma^h(0,x,\tau^h_t,\cdot)\|_{L^{(p/2)'}}\int_{\tau^h_t}^{\tau^h_t+h}\|b(s,\cdot)\|_{L^p}ds\lesssim \|\nabla g\|_{L^\infty}\|f\|_{L^p}\|b\|_{L^\vartheta-L^p}h^{1-\frac 1\vartheta}(\tau^h_t)^{-\frac{2d}{\alpha p}}.
\end{align*}
Setting ${\cal I}=\E[f(X^h_{\tau^h_t})(g(X^h_{\tau^h_t}+Z_t-Z_{\tau^h_t})-g(X^h_{\tau^h_t}))]=\int_{\R^d}\Gamma^h(0,x,\tau^h_t,y)f(y)\E[g(y+Z_t-Z_{\tau^h_t})-g(y)]dy$, we deduce that
\begin{align}
  |\E[f(X^h_{\tau^h_t})(g(X^h_t)-g(X^h_{\tau^h_t}))]|&\le |{\cal I}|+C\|\nabla g\|_{L^\infty}\|f\|_{L^p}\|b\|_{L^\vartheta-L^p}h^{1-\frac 1\vartheta}(\tau^h_t)^{-\frac{2d}{\alpha p}}.\label{decompefg2}
\end{align}
By Itô's formula, there exists $c_\alpha>0$ (for $\alpha=2 $, $c_\alpha=2$) s.t.
\begin{align*}
   c_\alpha{\cal I}&=\int_{\tau^h_t}^t\int_{\R^d}\Gamma^h(0,x,\tau^h_t,y)f(y)\E[\Delta^{\frac \alpha2} g(y+Z_s-Z_{\tau^h_t})
  ]dyds\\
  &=\int_{\tau^h_t}^t\int_{\R^d}\Gamma^h(0,x,\tau^h_t,y)f(y)p_\alpha(s-\tau^h_t,\cdot)\star\Delta^{\frac \alpha2} g(y)
  ]dyds,\end{align*}
 where in the above computations, we have denoted with a slight abuse of notation $\Delta^{\frac{\alpha}2}:=-(-\Delta)^{\frac \alpha2} $, the generator of the driving isotropic $\alpha $-stable process. 
  
Since, {from the Hölder inequality}, 
$$\int_{\R^d}\Gamma^h(0,x,\tau^h_t,y)|f(y)p_\alpha(s-\tau^h_t,\cdot)\star\Delta^{\frac \alpha 2} g(y)|dy\le \|f\|_{L^p}\|\Delta^{\frac \alpha 2} g\|_{L^{\vrho}}\|\Gamma^h(0,x,\tau^h_t,\cdot)\|_{L^{\frac{p\prho}{p-\prho}}},$$ 
we deduce with \eqref{estitranslusch} that
\begin{align*}
  |{\cal I}|\le \frac{t-\tau^h_t}{c_\alpha}\|f\|_{L^p}\|\Delta^{\frac \alpha 2} g\|_{L^{\vrho}}(\tau^h_t)^{-\frac d \alpha (\frac 1{\vrho}+\frac 1p)}.\end{align*}
Plugging this estimation in \eqref{decompefg2}, we conclude that
\begin{align*}
   |\E[f(X^h_{\tau^h_t})(g(X^h_t)-g(X^h_{\tau^h_t}))]|
                                                      &\lesssim \|f\|_{L^p}\left(h^{1-\frac 1\vartheta}(\tau^h_t)^{-\frac{2 d}{\alpha p}}\|b\|_{L^\vartheta-L^p}\|\nabla g\|_{L^\infty}+h(\tau^h_t)^{-\frac d \alpha(\frac 1{\vrho}+\frac 1p)}\|\Delta^{\frac \alpha 2} g\|_{L^{\vrho}}\right).
\end{align*}
\hfill$\Box$

\noindent {\bf Proof of Lemma \ref{lemerrfaibschemtt}.}
Let $t\in[h,T]$ so that $\tau^h_t\ge h$. By the Duhamel representation \eqref{duhamel-Diff} of the density of the diffusion, we have
\begin{align*}
  \Gamma(0,x,t,y)-\Gamma(0,x,\tau^h_t,y)={\cal E}^1(y)+{\cal E}^2(y)+{\cal E}^3(y)+{\cal E}^4(y)\end{align*}
where
\begin{align*}
  {\cal E}^1(y)&=\1_{\{t\ge 2h\}}\int_0^{\tau^h_t-h}\int_{\R^d}\Gamma(0,x,s,z)b(s,z)\cdot(\nabla p_\alpha(\tau^h_t-s,y-z)-\nabla p_\alpha(t-s,y-z))dzds\\{\cal E}^2(y)&=\int_{(\tau^h_t-h)\vee 0}^{\tau^h_t}\int_{\R^d}\Gamma(0,x,s,z)b(s,z)\cdot\nabla p_\alpha(\tau^h_t-s,y-z)dzds\\
  {\cal E}^3(y)&=-\int_{(\tau^h_t-h)\vee 0}^t\int_{\R^d}\Gamma(0,x,s,z)b(s,z)\cdot\nabla p_\alpha(t-s,y-z)dzds\\
  {\cal E}^4(y)&=p_\alpha({t,}y-x)-p_\alpha(\tau^h_t,y-x)
\end{align*}
Let us first suppose that $t\ge 2h$ so that $\tau^h_t-h\ge h$ and $\tau^h_t-h>t/3$.  On the one hand, by \eqref{estuv} with $(u,v)=(\frac{p\prho}{p-\prho},1)$ and \eqref{estitranslu},\begin{align*}
                                                                                    \|{\cal E}^3(\cdot)\|_{L^\prho}&\lesssim \int_{\tau^h_t-h}^t\|\Gamma(0,x,s,\cdot)\|_{L^{\frac{p\prho}{p-\prho}}}\|b(s,\cdot)\|_{L^p}(t-s)^{-\frac 1\alpha}ds\lesssim t^{-\frac d\alpha(\frac 1{\vrho}+\frac 1p)}\int_{\tau^h_t-h}^t\|b(s,\cdot)\|_{L^p}(t-s)^{-\frac 1\alpha}ds\\&\lesssim t^{-\frac d\alpha(\frac 1{\vrho}+\frac 1p)}\|b\|_{L^\vartheta-L^p}\left(\int_{\tau^h_t-h}^t(t-s)^{-\frac 1\alpha\times\frac{\vartheta}{\vartheta-1}}ds\right)^{\frac{\vartheta-1}{\vartheta}}\lesssim h^{\frac{\alpha-1}\alpha-\frac 1\vartheta}t^{-\frac d\alpha(\frac 1{\vrho}+\frac 1p)}\|b\|_{L^\vartheta-L^p}. \end{align*}In the same way, $\|{\cal E}^2(\cdot)\|_{L^\prho}\lesssim h^{\frac{\alpha-1}{\alpha}-\frac 1\vartheta}t^{-\frac d\alpha(\frac 1{\vrho}+\frac 1p)}\|b\|_{L^\vartheta-L^p}$. For $s\in[0,\tau^h_t-h]$, by \eqref{estidifftempsgaus} and since $t-s<2(\tau^h_t-s)$, we have\begin{align*}|\nabla p_\alpha(t-s,\cdot)-\nabla p_\alpha(\tau^h_t-s,\cdot)|\lesssim \frac{t-\tau^h_t}{(\tau^h_t-s)^{1+\frac 1\alpha}}\left(1+\frac{(t-s)^{\frac{d}{\alpha}}}{(\tau^h_t-s)^{\frac{d}{\alpha}}}\right)\bar{p}_\alpha(t-s,\cdot)\lesssim  h(t-s)^{-(1+\frac 1\alpha)}\bar{p}_\alpha(t-s,\cdot).
                          \end{align*}
With H\"older's and Young's inequalities, \eqref{estigradplu} and \eqref{estitranslu}, we deduce that
\begin{align*}
  \|{\cal E}^1(\cdot)\|_{L^\prho}&\lesssim h\int_0^{\tau^h_t-h}\|\Gamma(0,x,s,\cdot)\|_{L^{\frac{p\prho}{p-\prho}}}\|b(s,\cdot)\|_{L^p}(t-s)^{-(1+\frac 1\alpha)}\|\bar{p}_\alpha(t-s,\cdot)\|_{L^1}ds \\&\lesssim h\int_0^{\tau^h_t-h}s^{-\frac d\alpha(\frac 1\vrho+\frac 1p)}\|b(s,\cdot)\|_{L^p}(t-s)^{-(1+\frac 1\alpha)}ds
  \\&\lesssim h\|b\|_{L^\vartheta-L^p}\left(t^{-(1+\frac 1\alpha)\times\frac{\vartheta}{\vartheta-1}}\int_0^{t/3}s^{-\frac d\alpha(\frac 1{\vrho}+\frac 1p)\frac{\vartheta}{\vartheta-1}}ds+t^{-\frac d\alpha(\frac 1{\vrho}+\frac 1p)\frac{\vartheta}{\vartheta-1}}\int_{t/3}^{\tau^h_t-h}(t-s)^{-(1+\frac 1\alpha)\times\frac{\vartheta}{\vartheta-1}}ds\right)^{\frac{\vartheta-1}{\vartheta}}\\
  &\lesssim h\|b\|_{L^\vartheta-L^p}\left(t^{-(\frac 1\alpha+\frac 1\vartheta+\frac d\alpha(\frac 1{\vrho}+\frac 1p))}+t^{-\frac d\alpha(\frac 1{\vrho}+\frac 1p)}h^{-(\frac 1\alpha+\frac 1\vartheta)}\right)\lesssim h^{\frac {\alpha-1}\alpha-\frac 1\vartheta}t^{-\frac d\alpha(\frac 1{\vrho}+\frac 1p)}\|b\|_{L^\vartheta-L^p},
\end{align*}
where we used that \eqref{COND_KR} and $\frac d\vrho\le 1$ imply that $\frac d\alpha(\frac 1\vrho+\frac 1p)\frac{\vartheta}{\vartheta-1}<1$.

Let now $t\in[h,2h)$. Then ${\cal E}^1$ vanishes and by \eqref{estuv} with $(u,v)=(\frac{p\prho}{p-\prho},1)$ and \eqref{estitranslu},
\begin{align*}
                                                                                    \|{\cal E}^3(\cdot)\|_{L^\prho}&\lesssim \int_{0}^ts^{-\frac d\alpha(\frac 1{\vrho}+\frac 1p)}\|b(s,\cdot)\|_{L^p}(t-s)^{-\frac 1\alpha}ds\lesssim \|b\|_{L^\vartheta-L^p}\left(\int_{0}^ts^{-\frac d\alpha(\frac 1{\vrho}+\frac 1p)\frac{\vartheta}{\vartheta-1}}(t-s)^{-\frac 1\alpha\times\frac{\vartheta}{\vartheta-1}}ds\right)^{\frac{\vartheta-1}{\vartheta}}\\&\lesssim\|b\|_{L^\vartheta-L^p}t^{\frac{\alpha-1}\alpha-\frac 1\vartheta-\frac d \alpha(\frac 1\vrho+\frac 1p)}\lesssim \|b\|_{L^\vartheta-L^p}h^{\frac{\alpha-1}{\alpha}-\frac 1\vartheta}t^{-\frac d \alpha(\frac 1\vrho+\frac 1p)}. \end{align*}
                                                                                  In the same way, $\|{\cal E}^{{2}}(\cdot)\|_{L^\prho}\lesssim \|b\|_{L^\vartheta-L^p}h^{\frac{\alpha-1}\alpha-\frac 1\vartheta}t^{-\frac d \alpha(\frac 1\vrho+\frac 1p)}$. We conclude that
                                                                                  $$\forall t\in[h,T],\;\|{\cal E}^1(\cdot)\|_{L^\prho}+\|{\cal E}^2(\cdot)\|_{L^\prho}+\|{\cal E}^3(\cdot)\|_{L^\prho}\lesssim \|b\|_{L^\vartheta-L^p}h^{\frac{\alpha-1}\alpha-\frac 1\vartheta}t^{-\frac d \alpha(\frac 1\vrho+\frac 1p)}.$$

For $t\in[h,T]$, by \eqref{estidifftempsgaus} and the inequality $t<2\tau^h_t$,                    \begin{align*}|p_\alpha(t,\cdot)-p_\alpha(\tau^h_t,\cdot)|\lesssim  {h(\tau^h_t)^{-1}}\bar{p}_\alpha(t,\cdot)\lesssim  h^{\frac{\alpha-1}\alpha}t^{-\frac {\alpha-1}\alpha}\bar{p}_\alpha(t,\cdot).
\end{align*}
With \eqref{estigradplu}, we deduce that $\|{\cal E}^4(\cdot)\|_{L^\prho}\lesssim  h^{\frac{\alpha-1}\alpha}t^{-(\frac{\alpha-1}\alpha+\frac d{\alpha\vrho})}\lesssim  h^{\frac {\alpha-1}\alpha-\frac 1\vartheta}t^{-(\frac{\alpha-1}\alpha-\frac 1\vartheta+\frac d{\alpha\vrho})}$. We conclude by remarking that the Krylov and R\"ockner condition \eqref{COND_KR} ensures that $\frac{\alpha-1}\alpha-\frac 1\vartheta>\frac d{\alpha p}$.
\hfill$\Box$

\noindent {\bf Proof of Lemma \ref{lemgron}.}
Let $\Lambda=\lambda \times\left(1\vee\int_0^1s^{-a_1}(1-s)^{-a_2}ds\right)$. We have 
\begin{align*}
  \forall t\in[0,T],\;f(t)\le \kappa+\Lambda t^{1-a_2}\sup_{s\in[h,t]}f(s).
\end{align*}
Let $\theta$ be such that ${\frac{\Lambda 2^{a_1}\theta^{1-a_2}}{1-a_2}=\frac 12}$. Since $\Lambda\theta^{1-a_2}\le\frac 12$, we deduce that $\sup_{t\in[0,\theta\wedge T]}f(t)\le 2\kappa$. This thus gives the claim if $T\le\theta$. When $\theta<T$, we conclude by checking by induction on $j$ that $$\forall j\in\{0,\cdots,\lfloor T/\theta\rfloor\},\;\sup_{t\in [0,(j+1)\theta\wedge T]}f(t)\le 2\kappa\frac{{\widetilde \Lambda}^{j+1}-1}{\widetilde \Lambda-1},$$
where $\widetilde \Lambda=1\vee\left(2\Lambda T^{1-a_2}\right)$ and 
with convention that the last ratio is equal to $j+1$ when $\widetilde \Lambda=1$.
The property holds for $j=0$ and supposing that it holds at rank $j-1$ with $1\le j\le \lfloor T/\theta\rfloor$, we deduce from \eqref{eq:gron} and  $\lambda\le \Lambda$ that for $t\in[j\theta,(j+1)\theta\wedge T]$,
\begin{align*}
  f(t)&\le \kappa+\lambda \sup_{s\in[0,j\theta]}f(s) t^{a_1}\int_0^{(\tau^h_{j\theta}+h)\wedge t}s^{-a_1}(t-s)^{-a_2}ds+\Lambda \sup_{s\in[j\theta,(j+1)\theta]}(t/s)^{a_1}\int_{(\tau^h_{j\theta}+h)\wedge t}^tf(\tau^h_s)(t-s)^{-a_2}ds\\&\le \kappa+\Lambda
T^{1-a_2}\sup_{s\in[0,j\theta]}f(s)+\frac{\Lambda 2^{a_1}\theta^{1-a_2}}{1-a_2}\sup_{s\in[j\theta,(j+1)\theta\wedge T]}f(s)\\&
  = \kappa+\Lambda T^{1-a_2}\sup_{s\in[0,j\theta]}f(s)+\frac 12\sup_{s\in[j\theta,(j+1)\theta\wedge T]}f(s),
\end{align*}
where we used the definition of $\theta$ for the equality. We deduce that \begin{align*}\sup_{s\in[j\theta,(j+1)\theta\wedge T]}f(s)&\le 2 \kappa+2\Lambda T^{1-a_2}\sup_{s\in[0,j\theta]}f(s)\\\sup_{s\in[0,(j+1)\theta\wedge T]}f(s)&\le 2 \kappa+\widetilde \Lambda \sup_{s\in[0,j\theta]}f(s)\le 2 \kappa\left(1+\widetilde \Lambda\frac{{\widetilde \Lambda}^{j}-1}{\widetilde \Lambda-1}\right)
\end{align*} and conclude that the induction property holds at rank $j+1$.\hfill$\Box$
\subsection{Drift sum of contributions in distinct Lebesgue spaces}\label{secsomdrift}
We suppose that $b=\sum_{i=1}^I b_i$ with $I\in \N $, $b_i\in L^{\vartheta_i}([0,T],L^{p_i}(\R^d))$ for some couple $(p_i,\vartheta_i)$ satisfying the Krylov and R\"ockner condition  $$\frac d {p_i}+\frac {\alpha} {\vartheta_i}<\alpha -1\mbox{ and }p_i\ge 2.$$
Note for instance that if $\exists \delta>0,\;\forall (t,x)\in[0,T]\times \R^d$, $|b(t,x)|=|x|^{-\delta}$, then $b$ does not belong to any Lebesgue space because the integrability at the origin and at infinity in space lead to complementary conditions on $\delta$. Nevertheless, when $\delta<(\alpha-1)\wedge \frac d 2$ then $b$ satisfies the assumptions made in the current section with $I=2$ since  $b_1(t,x)=\1_{\{|x|\le 1\}}b(t,x)$ belongs to $L^\infty([0,T],L^{p_1}(\R^d))$ for some $p_1>\frac d{\alpha-1}\vee 2$ while $b_2(t,x)=\1_{\{|x|>1\}}b(t,x)$ belongs to $L^\infty([0,T],L^\infty(\R^d))$.

Like in \cite{fito:jour:meno:25}, it can be checked that \eqref{SDE} admits a unique weak solution (strong when $\alpha=2$) with a transition density $\Gamma$ which enjoys the Duhamel representation \eqref{duhamel-Diff}. It can also be checked that the Euler scheme \eqref{scheme-interpo_LP} admits a transition density satisfying the Duhamel formula \eqref{duhamel-scheme} and the heat kernel estimates \eqref{ARONSON_UPPER_LP} still hold.
\begin{THM}
{Let $I\in \N $ be fixed}.  We suppose that $b=\sum_{i=1}^I b_i$ with $b_i\in L^{\vartheta_i}([0,T],L^{p_i}(\R^d))$ for some couple $(p_i,\vartheta_i)$ satisfying the Krylov and R\"ockner condition  $$\frac d {p_i}+\frac {\alpha} {\vartheta_i}<\alpha -1\mbox{ and }p_i\ge 2.$$
We also assume that $\frac d {p_\wedge}+\frac \alpha {\vartheta_\wedge}\le \alpha-1$ where $p_\wedge=\min_{1\le i\le I}p_i$ and $\vartheta_\wedge=\min_{1\le i\le I}\vartheta_i$. 

Then $p_\wedge\in\left\{\vrho\in[p_\wedge', p_\wedge]:\frac d \vrho+\frac \alpha {\vartheta_\wedge}\le \alpha-1\right\}$ and for any $\rho$ in this set, there exists $C<\infty$ such that for each $h=T/n$ with $n\in\N^*$,
  $$\forall (t,x)\in (0,T]\times\R^d,\;\|\Gamma^h(0,x,t,\cdot)-\Gamma(0,x,t,\cdot)\|_{L^{\vrho'}}\le C h^{\frac{\alpha-1}\alpha-\frac 1{\vartheta_\wedge}}t^{-\frac d\alpha(\frac 1\vrho+\frac 1{p_\wedge})}.$$
\end{THM}
\begin{rem}
   Of course, $\max_{1\le i\le I}\left(\frac d {p_i}+\frac \alpha {\vartheta_i}\right)\le \frac d {p_\wedge}+\frac \alpha {\vartheta_\wedge}$ with strict inequality unless $(p_\wedge,\vartheta_\wedge)=(p_i,\vartheta_i)$ for some  $i\in\{1,\cdots,I\}$. So it is possible that $\max_{1\le i\le I}\left(\frac d {p_i}+\frac \alpha {\vartheta_i}\right)<\alpha-1$ while $\frac d {p_\wedge}+\frac \alpha {\vartheta_\wedge}=\alpha-1$.
\end{rem}
Since $p_i\ge 2$ for each $i\in\{1,\cdots,I\}$, we have $p_\wedge \ge 2$  so that $p'_\wedge \le 2$ and the interval $[p_\wedge', p_\wedge]$ is non empty and contains $p_\wedge$. Since $\frac d {p_\wedge}+\frac \alpha {\vartheta_\wedge}\le \alpha-1$, we deduce that $p_\wedge\in\left\{\vrho\in[p_\wedge', p_\wedge]:\frac d \vrho+\frac \alpha {\vartheta_\wedge}\le \alpha-1\right\}$. Let us now consider $\rho$ in this set.
We still rely on the decomposition \eqref{decomperrt} of the discretization error. Note that the estimations of $\|\Delta^j_t(\cdot)\|_{L^\prho}$ for $j\in\{2,\cdots,5\}$ and $\|\Delta^{61}_t(\cdot)\|_{L^\prho}$ in the proof of Theorem \ref{thmsinglebdrift} only rely on the inequality $|b_h|\le |b|$ and the fact that $\|b\|_{L^\vartheta-L^p}<\infty$ with $(p,r)$ satisfying \eqref{COND_KR} and $\frac{d}{\rho}+\frac{\alpha}{\vartheta}\le \alpha-1$. Since we now have $|b_h|\le |b|\le \sum_{i=1}^I|b_i|$ with $b_i\in L^{\vartheta_i}-L^{p_i}$ for $(p_i,\vartheta_i)$ satisfying \eqref{COND_KR} and $\frac{d}{\rho}+\frac \alpha{\vartheta_i}\le \alpha-1$, we get in the same way that
\begin{equation}
   \sum_{j=2}^5\|\Delta^j_t(\cdot)\|_{L^\prho}+\|\Delta^{61}_t(\cdot)\|_{L^\prho}\le C\sum_{i=1}^Ih^{\frac{\alpha-1}\alpha-\frac 1{\vartheta_i}}t^{-\frac d\alpha(\frac 1\rho+\frac 1{p_i})}.\label{estiprelim}
\end{equation}
To deal with the {truncation error} $\|\Delta^{62}_t(\cdot)\|_{L^\prho}$, we first remark that
\begin{align*}
   |b(s,z)-b_h(s,z)|\le |b(s,z)|\1_{\{|b(s,z)|\ge Bh^{-\frac {\alpha-1}\alpha}\}}\le \sum_{i=1}^I|b_i(s,z)|\sum_{j=1}^I\1_{\{|b_j(s,z)|\ge Bh^{-\frac{\alpha-1}\alpha}/I\}}\le \frac{Ih^{\frac{\alpha-1}\alpha}} B \sum_{i,j=1}^I|b_i(s,z)||b_j(s,z)|.
\end{align*}
As a consequence, using that the assumptions ensure that $\frac d \alpha(\frac 1 {p_i}+\frac 1{p_j})\frac{\vartheta_i\vartheta_j}{\vartheta_i\vartheta_j-\vartheta_i-\vartheta_j}<1$ and $\frac{\alpha-1}\alpha-\frac d{\alpha\rho}-\frac 1{\vartheta_i\wedge \vartheta_j}\ge 0$ for $i,j\in\{1,\cdots,I\}$, we get that for $t\ge 2h$
\begin{align*}& \|\Delta^{62}_t(\cdot)\|_{L^{\vrho'}}\lesssim h^{\frac{\alpha-1}\alpha}\sum_{i,j=1}^I\int_h^{\tau^h_t-h}\|\Gamma(0,x,s,\cdot)\|_{L^{\left(\frac{p_ip_j}{p_i+p_j}\right)'}}\|b_i(s,\cdot)\|_{L^{p_i}}\|b_j(s,\cdot)\|_{L^{p_j}}\|\nabla p_\alpha(t-s,\cdot)\|_{L^{\rho'}}ds\\&\lesssim h^{\frac{\alpha-1}\alpha}\sum_{i,j=1}^I\int_h^{\tau^h_t-h}(t-s)^{-\frac d \alpha (\frac 1 {p_i}+\frac 1{p_j})}\|b_i(s,\cdot)\|_{L^{p_i}}\|b_j(s,\cdot)\|_{L^{p_j}}(t-s)^{-(\frac 1\alpha+\frac d{\alpha\rho})}ds\\&\lesssim h^{\frac{\alpha-1}\alpha}\sum_{i,j=1}^I\|b_i\|_{L^{\vartheta_i}-L^{p_i}}\|b_j\|_{L^{\vartheta_j}-L^{p_j}}\left(\int_h^{\tau^h_t-h}(t-s)^{-\frac d \alpha(\frac 1 {p_i}+\frac 1{p_j})\frac{\vartheta_i\vartheta_j}{\vartheta_i\vartheta_j-\vartheta_i-\vartheta_j}}(t-s)^{-(\frac 1\alpha+\frac d{\alpha\rho})\frac{\vartheta_i\vartheta_j}{\vartheta_i\vartheta_j-\vartheta_i-\vartheta_j}}ds\right)^{1-(\frac 1{\vartheta_i}+\frac 1{\vartheta_j})}\\
                                                                                           &\lesssim h^{\frac{\alpha-1}\alpha}\sum_{i,j=1}^I\|b_i\|_{L^{\vartheta_i}-L^{p_i}}\|b_j\|_{L^{\vartheta_j}-L^{p_j}}\bigg(\1_{\{\frac d{\alpha\vrho}+\frac 1{\vartheta_i}+\frac 1{\vartheta_j}<\frac{\alpha-1}\alpha\}}t^{-\frac{d}\alpha(\frac 1{p_i}+\frac 1{p_j})+(\frac{\alpha-1}\alpha-\frac{d}{\alpha\vrho}-\frac 1{\vartheta_i}-\frac 1{\vartheta_j})}\\&\hspace{2cm}+\1_{\{\frac d{\alpha\vrho}+\frac 1{\vartheta_i}+\frac 1{\vartheta_j}=\frac{\alpha-1}\alpha\}}t^{-\frac d \alpha(\frac 1 {p_i}+\frac 1{p_j})}(\ln(T/h))^{1-(\frac 1{\vartheta_i}+\frac 1{\vartheta_j})}+\1_{\{\frac d{\alpha\vrho}+\frac 1{\vartheta_i}+\frac 1{\vartheta_j}>\frac{\alpha-1}\alpha\}}t^{-\frac d \alpha(\frac 1 {p_i}+\frac 1{p_j})}h^{\frac{\alpha-1}\alpha-(\frac d{\alpha\vrho}+\frac1{\vartheta_i}+\frac 1{\vartheta_j})}\bigg)
 \\ &\lesssim \sum_{i,j=1}^Ih^{\frac{\alpha-1}{\alpha}-\frac 1{\vartheta_i\vee \vartheta_j}}(\ln(T/h))^{\1_{\{\vartheta_i=\vartheta_j=\infty,\rho=d/(\alpha-1)\}}}t^{-\frac d \alpha(\frac 1 {p_i}+\frac 1{p_j})}\|b_i\|_{L^{\vartheta_i}-L^{p_i}}\|b_j\|_{L^{\vartheta_j}-L^{p_j}}.                                                                           \end{align*}
    To get rid of the logarithmic factor when $\vartheta_i=\vartheta_j=\infty$ and $\rho=d/(\alpha-1)$, we replace from the first inequality $\|\Gamma(0,x,s,\cdot)\|_{L^{\left(\frac{p_ipj}{p_i+p_j}\right)'}}\|\nabla p_2(t-s,\cdot)\|_{L^{\rho'}}$ by  $\|\Gamma(0,x,s,\cdot)\|_{L^{u}}\|\nabla p_2(t-s,\cdot)\|_{L^{v}}$ with $\frac 1{u'}=\frac{1-\lambda}\vrho+\frac{1+\lambda}2(\frac 1{p_i}+\frac 1{p_j})$ and $\frac 1{v'}=\frac \lambda\vrho+\frac{1-\lambda}2(\frac 1{p_i}+\frac 1{p_j})$ for some $\lambda\in[0,1)$, which is possible when $p_i\wedge p_j>2$ since then $\frac{1-\lambda}\vrho+\frac{1+\lambda}2(\frac 1{p_i}+\frac 1{p_j})$ goes to $\frac 1{p_i}+\frac 1{p_j}<1$ as $\lambda\to 1-$. Note that we cannot have $p_i\wedge p_j=2$ and $\vrho=d/(\alpha-1)$, since then $p_\wedge= 2$ and  $d=1$, $\alpha>\frac 32$ so that $\rho<2 $, which contradicts the condition $\rho\ge p_\wedge'= 2$ that we suppose.

With \eqref{estiprelim} and since $\rho\le p_\wedge$, we deduce that\begin{equation}
   \exists C_{(2-6)}<\infty,\;\forall t\in[0,T],\;\sum_{i=2}^6\|\Delta^i_t(\cdot)\|_{L^\prho}\le C^{\frac{\vartheta_\wedge-1}{\vartheta_\wedge}}_{(2-6)}h^{\frac{\alpha-1}\alpha-\frac 1{\vartheta_\wedge
}}t^{-\frac d \alpha(\frac 1\vrho +\frac 1 {p_\wedge
})}.\label{esti2-6bis}
 \end{equation}
Let $a\in\left(\max_{1\le i\le I}\left(\frac d{\alpha p_i}+\frac 1{\vartheta_i}\right)+\frac 1\alpha-\frac 1{\vartheta_\wedge},1-\frac 1{\vartheta_\wedge}\right)$ where the left boundary is smaller than the right one according to the Krylov and R\"ockner condition \eqref{COND_KR} satisfied by each couple $(p_i,\vartheta_i)$.
Reasoning like in the above derivation of \eqref{estidel1}, we get for $t\ge h$
\begin{align*}
              &\|\Delta^1_t(\cdot)\|_{L^\prho}\le C\sum_{i=1}^I\|b_i\|_{L^{\vartheta_i}-L^{p_i}}\left(\int_h^t\|\Gamma(0,x,\tau^h_s,\cdot)-\Gamma^h(0,x,\tau^h_s,\cdot)\|^{\frac {\vartheta_i}{{\vartheta_i}-1}}_{L^\prho}(t-s)^{-(\frac 1\alpha+\frac{d}{\alpha{p_i}})\times\frac {\vartheta_i}{{\vartheta_i}-1}}ds\right)^{\frac{{\vartheta_i}-1}{\vartheta_i}}\\
&\le C\sum_{i=1}^I\|b_i\|_{L^{\vartheta_i}-L^{p_i}}\left(\int_h^t\|\Gamma(0,x,\tau^h_s,\cdot)-\Gamma^h(0,x,\tau^h_s,\cdot)\|^{\frac {\vartheta_\wedge}{{\vartheta_\wedge}-1}}_{L^\prho}(t-s)^{-\frac {a\vartheta_\wedge}{{\vartheta_\wedge}-1}}ds\right)^{\frac{{\vartheta_\wedge}-1}{\vartheta_\wedge}}\left(\int_h^t(t-s)^{(a-\frac 1\alpha-\frac d{\alpha p_i})\frac{\vartheta_i\vartheta_\wedge}{\vartheta_i-\vartheta_\wedge}}ds\right)^{\frac{\vartheta_i-\vartheta_\wedge}{\vartheta_i\vartheta_\wedge}}\\             
&\le C\sum_{i=1}^I\|b_i\|_{L^{\vartheta_i}-L^{p_i}}T^{a+\frac 1{\vartheta_\wedge}-(\frac 1\alpha+\frac d{\alpha p_i}+\frac 1{\vartheta_i})}\left(\int_h^t\|\Gamma(0,x,\tau^h_s,\cdot)-\Gamma^h(0,x,\tau^h_s,\cdot)\|^{\frac {\vartheta_\wedge}{{\vartheta_\wedge}-1}}_{L^\prho}(t-s)^{-\frac {a\vartheta_\wedge}{{\vartheta_\wedge}-1}}ds\right)^{\frac{{\vartheta_\wedge}-1}{\vartheta_\wedge}}\\&\le C\left(\int_h^t\|\Gamma(0,x,\tau^h_s,\cdot)-\Gamma^h(0,x,\tau^h_s,\cdot)\|^{\frac {\vartheta_\wedge}{{\vartheta_\wedge}-1}}_{L^\prho}(t-s)^{-\frac {a\vartheta_\wedge}{{\vartheta_\wedge}-1}}ds\right)^{\frac{{\vartheta_\wedge}-1}{\vartheta_\wedge}}.
\end{align*}
We deduce that the function $f(t)=\left(t^{\frac d \alpha(\frac 1\vrho +\frac 1 {p_\wedge
    })}\|\Gamma(0,x,t,\cdot)-\Gamma^h(0,x,t,\cdot)\|_{L^{\rho'}}\right)^{\frac{\vartheta_\wedge}{{\vartheta_\wedge}-1}}$ which is bounded on $[0,T]$ by \eqref{ARONSON_UPPER_LP} satisfies
$$\forall t\in[0,T],\;f(t)\le C_{(2-6)}h^{(\frac{\alpha-1}\alpha-\frac 1{\vartheta_\wedge})\times \frac{\vartheta_\wedge}{\vartheta_\wedge-1}}+Ct^{\frac d\alpha(\frac 1\vrho+\frac 1{p_\wedge})\times\frac{\vartheta_\wedge}{\vartheta_\wedge-1}}\int_{h}^{h\vee t}f(\tau^h_s)s^{-\frac d \alpha(\frac 1\vrho+\frac 1{p_\wedge})\times\frac{\vartheta_\wedge}{\vartheta_\wedge-1}}(t-s)^{-\frac{a\vartheta_\wedge}{\vartheta_\wedge-1}}ds.$$
Since  $\frac d{p_\wedge}+\frac \alpha{\vartheta_\wedge}\le \alpha-1$ with strict inequality when $\vartheta_\wedge=\infty$ and $\frac d{\rho}+\frac \alpha{\vartheta_\wedge}\le \alpha-1$, we have  $\frac d \alpha(\frac 1 {p_\wedge}+\frac 1\rho)+\frac 2{\vartheta_\wedge}\le 2\times\frac{\alpha-1}\alpha\le 1$ where the first inequality is strict when $\vartheta_\wedge=\infty$ so that $\frac d\alpha(\frac 1\vrho+\frac 1{p_\wedge})\times\frac{\vartheta_\wedge}{\vartheta_\wedge-1}<1$. Since $\frac{a\vartheta_\wedge}{\vartheta_\wedge -1}<1$, we conclude with
Lemma \ref{lemgron}.

\section{Definitions and useful tools about Besov spaces}\label{SEC_BESOV}
In this paragraph, we set definitions/notations and remind technical preliminaries - reviewed or directly established in \cite{chau:jabi:meno:22-1} - that will be used throughout the {remaining of the} present paper.
 
 From here on, we denote  by $B^\beta_{\ell,m}, {(\ell,m,\beta)\in[1,+\infty]^2\times\R }$  the Besov space with regularity index $\beta$ and integrability \textcolor{black}{parameters} $ \ell,m$ (\textcolor{black}{see e.g.  \cite{athr:butk:mytn:20}, \cite{chau:meno:22}, \cite{lema:02}} for some related applications and the dedicated monograph \cite{trie:83} by Triebel). 
 We use the thermic characterization  through the isotropic stable heat kernel for its definition (see e.g. Section 2.6.4 in \cite{trie:83}). 
\textcolor{black}{Namely, denoting by ${\mathcal S'}( \R^d) $ the dual space of the Schwartz class ${\mathcal S}( \R^d) $ and by ${\cal F},{\cal F}^{-1}$ the Fourier and inverse Fourier transforms},
  \[
 B^\beta_{\ell,m}=\left\{f\in\textcolor{black}{\mathcal S'}( \R^d)\,:\,\| f\|_{B^\beta_{\ell,m}}:=\|\mathcal F^{-1}(\phi)\star f\|_{\textcolor{black}{L^\ell}}+\mathcal T_{\ell,m}^\beta(f)<\infty\right\},
 \]
 \begin{align}
\mathcal T_{\ell,m}^\beta(f):=&\left\{
\begin{aligned}
&
\left(\int_0^T\,\frac{dv}{v}v^{(n-\beta/\textcolor{black}{\alpha})m}\|\partial^n_v p_\alpha(v,\cdot)*f\|^m_{L^\ell}\right)^{\frac 1m}\,\text{for}\,1\le m<\infty,\\
&\sup_{v\in(0,T]}\left\{v^{(n-\beta/\textcolor{black}{\alpha})}\|\partial^n_v p_\alpha(v,\cdot)*f\|_{L^\ell}\right\}\,\text{for}\,m=\infty,
\end{aligned}
\right. 
\label{HEAT_CAR}
\end{align}
  $n$ being any non-negative integer (strictly) greater than $\beta/\textcolor{black}{\alpha}$, the function $\phi$ being % a $\mathcal C^\infty_0$-function 
  { an infinitely differentiable function with compact support} such that $\phi(0)\neq 0$, and \textcolor{black}{$p_\alpha(v,\cdot)$ denoting the density function at time $v$ of the $d$-dimensional isotropic stable process}, which is here as well precisely the one of the driving noise in \eqref{SDE}. We can refer e.g. to the discussion in Section 2.6.4. and the general characterization in Section 2.5.1 of \cite{trie:83} for this characterization of Besov spaces. Note that we choose $T$ as the upper bound of the integral for the thermic part of the norm while this bound is usually set to $1$. Our choice which is more convenient for the computations below leads to an equivalent norm. It will be in particular useful to derive some heat kernel estimates on the density of the driving noise itself in Section \ref{SEC_TEC}.

We now list some properties that we will thoroughly use in the analysis.\\\noindent$\bullet$ {\it Embeddings between Lebesgue and $B^0_{\ell,m}$-spaces (\cite[Prop. 2.1]{Sawano-18})}:
\begin{equation}
\forall 1 \le \ell \le \infty,\qquad B^0_{\ell,1}\hookrightarrow L^\ell \hookrightarrow B_{\ell,\infty}^0.\label{EMBEDDING}\tag{$\mathbf E_1$}
\end{equation}

\noindent$\bullet$ {\it Product rule:} for all ${\beta} \in \R$, $(\ell,m)\in [1,+\infty]^2$ and $\rho>\max\Big({\beta},-{\beta}\Big)$, $\forall (f,g)\in B_{\infty,\infty}^\rho \times B_{\ell,m}^{\beta} $,
	\begin{equation}\label{PR}
		\|f \cdot g\|_{ \B_{\ell,m}^{\beta}} \le \|f \|_{\B_{\infty,\infty}^\rho} \|g\|_{ \B_{\ell,m}^{\beta}}.\tag{${\mathbf P}$}
	\end{equation}
	See Theorem 4.37 in \cite{Sawano-18} for a proof.\\
\noindent$\bullet$ {\it Duality inequality:} for all ${\beta} \in \R$, $(\ell,m)\in [1,+\infty]^2$, with $m'$ and $\ell'$ respective conjugates of $m$ and $\ell$, and $(f,g)\in B_{\ell,m}^{\beta} \times B_{\ell',m'}^{-{\beta}}$,
	\begin{equation}\label{dual-ineq}
		\textcolor{black}{|\langle f,g\rangle_{B_{\ell,m}^{{\beta}},B_{\ell',m'}^{-{\beta}}}|:=}\left| \int f(y)g(y) \mathrm{d}y \right| \leq \Vert f \Vert_{B_{\ell,m}^{\beta}} \Vert g \Vert_{B_{\ell',m'}^{-{\beta}}},\tag{${\mathbf D}$}
	\end{equation}
	\textcolor{black}{where the duality pairing is from now on denoted in integral form for notational convenience}.
	We refer to Proposition 6.6 in \cite{lema:02} for a proof.\\

\noindent$\bullet$ {\it Young inequality (\cite[Theorem 3]{Burenkov-90} or \cite{Sawano-18}):} for all ${\beta} \in \R$, $(\ell,m)\in [1,+\infty]^2$, for any $\delta\in \R$ and for $(\ell_1,\ell_2)\in [1,\infty]^2$ and $(m_1,m_2)\in (0,\infty]^2$ such that 
	$$1+\frac{1}{\ell} = \frac{1}{\ell_1} + \frac{1}{\ell_2} \qquad  \text{and}\qquad  \frac{1}{m}\leq \frac{1}{m_1}+\frac{1}{m_2},$$
	there exists $\cv$ a universal constant depending only on $d$ such that, for $f\in B_{\ell_1,m_1}^{{\beta} - \delta}$ and $g\in B_{\ell_2,m_2}^\delta$,
	\begin{equation}
		\label{YOUNG} \Vert f\star g \Vert_{B_{\ell,m}^{\beta}} \leq \cv \Vert f \Vert_{B_{\ell_1,m_1}^{{\beta} - \delta}} \Vert g \Vert_{B_{\ell_2,m_2}^\delta}.\tag{$\mathbf Y $}
	\end{equation}

\noindent $\bullet$ \emph{Besov norm of heat kernel}. For parameters $\beta\in \R^*, \ell,m\in [1,+\infty] $, $\theta\in \{0,1\} $ and a multi-index $\mathbf a$ {with length $|\mathbf a|=\sum_{i=1}^da_i\le 3$}
, there exists $\ch:=C(\alpha,\textcolor{black}{\ell},m,\beta,d)$ s.t. for all $s\in (0,T]$,
\begin{equation}\label{SING_STABLE_HK}
\big\| \partial_s^\theta\partial^{\mathbf a} p_{\alpha}(s,\cdot) \big\|_{
B^{\beta}_{\ell,m}}\le \frac{\ch}{(s\wedge 1)^{{\big(\theta+\frac {|\mathbf a|}\alpha+[\frac \beta\alpha+\frac d{\alpha}(1-\frac 1\ell)]\big)_+}}}(1+\I_{m<\infty,\theta+\frac{|\mathbf a|}\alpha +[\frac \beta\alpha+\frac d{\alpha}(1-\frac 1\ell)]=0}|\ln(s)|). \tag{$\mathbf{HK} $}\end{equation}
The inequality \eqref{SING_STABLE_HK} will be crucial for the analysis below and is proved in Section \ref{SEC_TEC} when $\beta<0$ and in \cite[Lemma 12]{chau:meno:22} when $\beta> 0$.\\
$\bullet$ Smooth approximation of the interaction kernel and associated uniform-control properties:
\begin{PROP}\cite[Proposition 2]{chau:jabi:meno:22-1}\label{PROP_APPROX}
Let  $b\in L^\vartheta((t,T],B_{p,q}^\beta)$ and {\color{black} $\beta\in (-1,0] $}, $1\le p,q \le \infty $.
There exists a sequence of \textcolor{black}{time-space} smooth bounded functions $(b^\varepsilon)_{\varepsilon >0} $ s.t.
\begin{equation*}%\label{SMOOTH_APP_GEN}
\|b-b^\varepsilon\|_{L^{\bar \vartheta}((t,T],B_{p,q}^{\tilde \beta})} \underset{\varepsilon \rightarrow 0}{\longrightarrow} 0,\quad \forall \tilde \beta<\beta,
\end{equation*}
with $\bar \vartheta=\vartheta $ if $\vartheta<+\infty $ and for any $\bar \vartheta<+\infty $ if $\vartheta=+\infty$. Moreover, there exists $ \cc\ge 1$, $\displaystyle \sup_{\varepsilon>0}|b^\varepsilon|_{L^{\bar \vartheta}((t,T],B_{p,q}^{\beta})}\le \textcolor{black}{\cc} |b|_{L^{\bar \vartheta}((t,T],B_{p,q}^{\beta})}$.

If $p,q,\vartheta<+\infty$ it then also holds, see e.g. \cite{lema:02}, that 
\begin{equation*}%\label{SMOOTH_APPR_FINITE}
\|b-b^\varepsilon\|_{L^{ \vartheta}((t,T],B_{p,q}^{\beta})} \underset{\varepsilon \rightarrow 0}{\longrightarrow} 0.
\end{equation*}
\end{PROP}

\section{Hölder case}\label{SEC_HOLDER}
We first restrict in this section to the case $\beta\in (0,1) $ assuming that $b$ in \eqref{SDE} belongs to $L^\infty\big([0,T],C^\beta(\R^d)\big) $.

{Starting from a given point $X_0=x\in \R^d$}, the Euler scheme is then defined for $k\ge 0$ {and $t_k := kh $, where $h=T/n$ with $n\in\N^*$ stands for the time step}, as follows:
	\begin{equation}\label{euler-scheme_H}
		X_{t_{k+1}}^h = X_{t_k}^h + hb(U_k,X_{t_k}^h)+ (Z_{t_{k+1}}-Z_{t_{k}}),\ k\in \N,
	\end{equation}
	the $(U_k)_{k\in \N} $ being independent random variables, also independent  from the driving noise, uniformly distributed  on the time interval $[t_k,t_{k+1}]$.
        
	With the previous notations, the continuous time dynamics of the scheme writes: 
	\begin{equation}\label{scheme-interpo}
		d X^h_t=b(U_{\lfloor t/h\rfloor},X^h_{\tau^h_t}) d t+ d Z_t,
	\end{equation}
{where again $\tau_s^h := h \lfloor \frac{s}{h}\rfloor$}.

Under the current assumptions it is known that both the diffusion and its related Euler scheme have densities, respectively denoted by $ \Gamma $ and $\Gamma^h $, in positive time. In particular, it has been proved in \cite{fito:meno:25} that the following Duhamel representation holds for the densities.
 \begin{PROP}[Duhamel representations for the densities of the SDE and the Euler scheme]%\label{prop-main-estimates-D}
		The density $\Gamma(s,x,t,\cdot) $ of the unique weak solution to Equation \eqref{SDE} starting from $x$ at time $s\in [0,T)$ admits the Duhamel representation \eqref{duhamel-Diff}.
        
Similarly, for $k\in \llbracket 0,n-1\rrbracket ,\ t\in (t_k,T] $, the density of $X_t^h$ admits, conditionally to $X_{t_k}^h=x$, a transition density $\Gamma^h(t_k,x,t,\cdot) $, which again enjoys a Duhamel type representation: for all $y\in \R^d $,
		\begin{align}
			\Gamma^h(t_k,x,t,y)
			= p_\alpha(t-t_k,y-x)-\int_{t_k}^{ t}\E_{t_k,x}\left[b(U_{\lfloor r/h\rfloor},X^h_{\tau_r^h})\cdot\nabla_y  p_\alpha(t-r,y-X^h_r)\right]dr.\label{duhamel-schemeb}
		\end{align} 
\end{PROP}
We first recall some important properties on the densities in the considered case. Most of them come from \cite{fito:meno:25}, but some need to be refined.

\begin{PROP}[A priori controls on the densities in the H\"older case]
\label{PROP_DENS_HOLDER}

Set $\gamma:=\alpha-1+\beta$. 
\begin{itemize}
\item Aronson type bounds.\\
There exists $C:=C(\textcolor{black}{d},b,\alpha,T)<\infty$ s.t. for all $0\le s<t\le T,\ x,y,\in \R^d$,
\begin{align}
			\label{ARONS_DIFF}
			\Gamma(s,x,t,y)\le C\bar p_\alpha(t-s,y-x).
			\end{align}
{Similarly, there exists $C<\infty$ such that for $h=\frac{T}{n},\ n\in \mathbb N^*$, $0\le t_k<t\le T$, $x,y\in\R^d$,}
                 		\begin{align}\label{ARONS_SCHEME}
			&\Gamma^h(t_k,x,t,y)\le C \bar p_\alpha (t-t_k,y-x).
		\end{align}

\item Forward regularity of the density of the diffusion. 
For $\varepsilon \in (0,\textcolor{black}{\gamma\wedge 1}] $, there exists $C_\varepsilon :=C_\varepsilon(\textcolor{black}{d},b,\alpha,T)$ s.t. for all $0\le s<t\le T,\ x,y,w\in \R^d$ s.t.  $|y-w|\le (t-s)^{\frac 1\alpha} $,
\begin{align}
			\label{holder-space-gamma}
			|\Gamma(s,x,t,y)-\Gamma(s,x,t,w)|\le \textcolor{black}{C_\varepsilon}\left(\frac{|y-w|}{(t-s)^{\frac1\alpha}} \right)^{\textcolor{black}{\gamma_\varepsilon}}\bar p_\alpha(t-s,w-x),\ \gamma_\varepsilon:=(\gamma\wedge 1)-\varepsilon.
\end{align}
\item Forward regularity of the density of the scheme.
		For all $0< t_j<t_k\le T $, $x,y,w\in \R^d $, $|y-w|\le (t_k-t_j)^{\frac 1\alpha} $,
		\begin{align}
			\label{holder-space-gammah}
			|\Gamma^h(t_j,x,t_k,y)-\Gamma^h(t_j,x,t_k,w)|\le \textcolor{black}{C_\varepsilon} \left(\frac{|y-w|\textcolor{black}{+h^{\frac 1\alpha}}}{(t_k-t_j)^{\frac1 \alpha}} \right)^{\textcolor{black}{\gamma_\varepsilon}}\bar p_\alpha(t_k-t_j,w-x).
			\end{align}
If additionally $\alpha-1>\beta $ it holds that,
		for all $0< t_j<t_k\le T $, $x,y,w\in \R^d $, $|y-w|\le (t_k-t_j)^{\frac 1\alpha} $,
		\begin{align}
			\label{holder-space-gamma_SPEC}
			|\Gamma^h(t_j,x,t_k,y)-\Gamma^h(t_j,x,t_k,w)|\le C \left(\frac{|y-w|}{(t_k-t_j)^{\frac1 \alpha}} \right)^{\beta}\bar p_\alpha(t_k-t_j,w-x).
			\end{align}
\item Besov norms involving the drift and the density of the diffusion.
There exists a constant $C\ge 1$ s.t.	 for 
 $s\in [h,t] $, $2h\le t\le T $,	and $\Psi\in B_{\infty,\infty}^\beta $,	
		\begin{align}
			\label{BESOV_NORM_1}
			\|\Gamma(0,x,s,\cdot)\Psi\|_{B_{1,\infty}^{\beta}}\le C(1+s^{-\frac \beta\alpha})\|\Psi\|_{B_{\infty,\infty}^\beta},
			\end{align}
\item Besov norms involving the drift and the density of the {Euler scheme}.
Assume that $\alpha-1>\beta $. There exists a constant $C\ge 1$ s.t.	 for  
$s\in [h,t] $, $2h\le t\le T $,	and $\Psi\in B_{\infty,\infty}^\beta $,	
		\begin{align}
			\label{BESOV_NORM_2}
			\|\Gamma^h(0,x,\tau_s^h,\cdot)\Psi\|_{B_{1,\infty}^{\beta}}\le C(1+s^{-\frac \beta\alpha})\|\Psi\|_{B_{\infty,\infty}^\beta}.
			\end{align}
\end{itemize}
\end{PROP}

The inequalities \eqref{ARONS_DIFF}-\eqref{holder-space-gammah} were established in \cite{fito:meno:25}. 
The inequalities \eqref{holder-space-gamma_SPEC}, \eqref{BESOV_NORM_1} and \eqref{BESOV_NORM_2} will be proved in Section \ref{SEC_TEC} below.

\begin{rem}
{Pay attention that the control \eqref{BESOV_NORM_2} for the scheme is valid under the condition $\alpha-1>\beta $, whereas the corresponding one \eqref{BESOV_NORM_1} for the diffusion remains valid in any case. This additional constraint is due to the fact that, for the current analysis, we precisely need \eqref{holder-space-gamma_SPEC} to handle the Hölder continuity at any scale. Let us mention that the additional time step in the r.h.s. of \eqref{holder-space-gammah}  is somehow the price to pay to handle the difference of the time arguments in the Duhamel representation \eqref{duhamel-schemeb} of the density of the scheme.}
\end{rem}

The main result of this section is the following theorem.

\begin{THM}[Weak error in the Hölder case]\label{THM_HOLDER}
Let $b\in L^\infty([0,T],C^\beta(\R^d)) $ with $\beta\in (0,{1)}$. 
Then, there exists a finite constant $C_{\beta}$ s.t. for all {$\varphi\in C^\beta(\R^d)=B_{\infty,\infty}^\beta(\R^d)$}, $x\in \R^d, t\in [2h,T]$,
\begin{align}
&\left|{\E_{0,x}\left[\varphi(X_t^h)-\varphi(X_t)\right]}\right|\le 
C_{\beta}h^{1-\frac{(1- 2\beta)}{\alpha}}(1+t^{-\frac{\beta}{\alpha}})\|\varphi\|_{B_{\infty,\infty}^\beta},{\mbox{ when }\beta\in (0,1/2),}\label{BD_W_ERR_HOLDER_1}\\
&\left|{\E_{0,x}\left[\varphi(X_t^h)-\varphi(X_t)\right]}\right|\le 
C_{\beta}h(1+|\ln(h)|\I_{\beta= \frac 12})(1+t^{-\frac{1}{2\alpha}})\|\varphi\|_{B_{\infty,\infty}^\beta}{\mbox{ when }\beta\in [1/2 ,1)}.
\label{BD_W_ERR_HOLDER_2}\end{align}
\end{THM}
{
\begin{rem}
We carefully mention that we are not able to exploit the additional regularity when $ \beta\in (1/2,1) $ in the above theorem and we somehow could have restricted to $\beta\in (0,1/2^+) $, since when $b\in L^\infty([0,T],C^\beta(\R^d)) $ for $\beta\in ( \frac 12,1) $ it also holds that $b\in L^\infty([0,T],C^{\frac 12^+}(\R^d)) $. The case $\beta=\frac 12 $ induces an additional logarithmic correction and as soon as $\beta>\frac 12 $ we have the  {optimal} convergence rate of order 1. Observe as well that, up to the logarithm, there is a continuity in $\beta $ for the rates in \eqref{BD_W_ERR_HOLDER_1} and \eqref{BD_W_ERR_HOLDER_2}.
\end{rem}
}
\color{black}

\subsection{Proof of Theorem \ref{THM_HOLDER}}

\subsubsection{Decomposition of the error}
% Convention
%% Delta_5_B Gronwall
%% Delta_2_B Reg temps
%% Delta_3_B sensi HK en temps
%%

Note carefully that for $x\in \R^d,t\in (2h,T] $ and $\varphi\in C^\beta=B_{\infty,\infty}^\beta $ we decompose the error as follows:
    \begin{align}
{\E_{0,x}\left[\varphi(X_t^h)-\varphi(X_t)\right]}=&\langle (\Gamma^h-\Gamma)(0,x,\cdot,\cdot),\varphi\rangle_{B_{1,1}^{-\beta}, B_{\infty,\infty}^{\beta}}\notag\\	 
=   &\int_{}\big(\Gamma^h(0,x,t,y)-\Gamma(0,x,t,y)\big) \varphi(y) d y\notag\\
=      &    \int_h^{\tau_t^h-h} \E_{0,x} \bigg[ b(s,X_{\tau_s^h})\cdot \nabla P_{t-s}^\alpha \varphi(X_{\tau_s^h})
        -b(s,X_{\tau_s^h}^h)\cdot \nabla  P_{t-s}^\alpha \varphi(X_{\tau_s^h}^h) \bigg] d s\notag\\%Gronwall
&   + \int_h^{\tau_t^h-h} \E_{0,x} \left[ b(s,X_s)\cdot \nabla P_{t-s}^\alpha \varphi(X_s)-b(s,X_{\tau_s^h})\cdot \nabla  P_{t-s}^\alpha \varphi (X_{\tau_s^h})
        \right] d s\notag\\% Sensi temps dens
                &  + \int_h^{\tau_t^h-h} \E_{0,x} \left[b(U_{\lfloor s/h\rfloor},X_{\tau_s^h}^h)\cdot \left( \nabla P_{t-s}^\alpha \varphi(X_{\tau_s^h}^h)-\nabla  P_{t-s}^\alpha \varphi(X_s^h)\right)\right]d s\notag\\%Sensi temps Euler
		&  +\int_h^{\tau_t^h-h} \E_{0,x} \left[b(U_{\lfloor s/h\rfloor},X_{\tau_s^h}^h)\cdot \left( \nabla P_{t-U_{\lfloor s/h\rfloor}}^\alpha\varphi(X_{\tau_s^h}^h)-\nabla  P_{t-s}^\alpha \varphi(X_{\tau_s^h}^h)\right)\right]d s\notag\\%Sensi temps HK
	&+\int_{\tau_t^h-h}^t \E_{0,x} \left[ b(s,X_s)\cdot \nabla P_{t-s}^\alpha \varphi(X_s)-b(U_{\lfloor s/h\rfloor},X_{\tau_s^h}^h)\cdot\nabla P_{ t-s}^\alpha\varphi (X_s^h)  \right] d s\notag\\% Last time step
  &+\int_0^h \E_{0,x} \left[ b(s,X_s)\cdot \nabla P_{t-s}^\alpha \varphi(X_s)-b(U_0,{x})\cdot\nabla P_{t-s}^\alpha \varphi(t-s,X_s^h) \right] d s\notag\\% First time step
        =:&\Delta^1 + \Delta^2  + \Delta^3 + \Delta^4+\Delta^5 +\Delta^6,\label{DECOUP_ERR}
    \end{align}
    where we emphasize {the partial cancellation between $\Delta^4$ and $\Delta^1$ coming from}:
	\begin{align}
	&\int_h^{\tau_t^h-h} \E_{0,x} \left[b(U_{\lfloor s/h\rfloor},X_{\tau_s^h}^h)\cdot \nabla  P_{t-U_{[\lfloor s/h\rfloor]}}^\alpha \varphi(X_{\tau_s^h}^h)\right]d s\nonumber\\
	&\qquad \qquad =\sum_{i=1}^{\lfloor t/h\rfloor-1}\frac{1}{h}\int_{t_i}^{t_{i+1}} \int_{t_i}^{t_{i+1}}
	\E_{0,x} \left[b(r,X_{t_i}^h)\cdot \nabla  P_{t-r}^\alpha \varphi(X_{t_i}^h)\right] d s d r\nonumber\\
	&\qquad \qquad =\int_h^{\tau_t^h{-h}} \E_{0,x} \left[b(s,X_{\tau_s^h}^h)\cdot \nabla  P_{t-s}^\alpha \varphi (X_{\tau_s^h}^h)\right]d s.\label{convenient-fubini}
	\end{align}
Note that  the first time step contribution $\Delta^6$ and the cutoff contribution $\Delta^6_t(y)$ in \eqref{decomperrt} are of different nature. Also note that  we assume that $t\ge 2h$ to slightly simplify the argument but could deal with the case $t\in [0,2h)$ like we did in the proof of Theorem \ref{thmsinglebdrift} when the drift belongs to Lebesgue spaces.

To proceed, we will rely on the following proposition which quantifies the gain associated with the smoothness of the considered test function {and is a direct consequence of \eqref{SING_STABLE_HK}.}
\begin{PROP}[Smoothing effect of the heat semigroup]
Let $\varphi \in B_{\infty,\infty}^\beta$, $\beta\in (0,1) $.  {Then, for $\mathbf a \in \N^{d}$ such that $1\le |\mathbf a |\le 3$,}  there exists $c\ge 0$ s.t. for $t>0$,
\begin{align}
\label{GAIN_TEST_FUNC}
\|\partial^{\mathbf a} P^\alpha_t \varphi\|_{L^\infty}\le ct^{-\frac{|\mathbf a|}\alpha+\frac \beta\alpha}\|\varphi\|_{B_{\infty,\infty}^\beta}.
\end{align}
\end{PROP}

\subsubsection{Analysis of the various contributions in \eqref{DECOUP_ERR} when $\alpha-1>\beta $.}
We will here first assume, in order to benefit from the control \eqref{BESOV_NORM_2} that $\beta<\alpha-1 $, which is e.g. always satisfied in the Brownian case. The remaining cases, which impose to exploit the former pointwise control for the error on the densities established in \cite{fito:meno:25} will be developed in Section \ref{MORE_GEN_FOR_PROOF} below. We also specify that in the proof below, we will assume $\beta\in (0,1)$. This will in particular emphasize that we were not able to exploit the additional regularity beyond the threshold $1/2$.

Let us first consider the first time-step contribution. Write from \eqref{GAIN_TEST_FUNC}:
\begin{align}
|\Delta^6|\le h\|b\|_{L^\infty}\|\varphi\|_{B_{\infty,\infty}^{\beta}} t^{-\frac{1-\beta}{\alpha}}{\le h^{\frac{\alpha\wedge (\alpha-1+2\beta)}{\alpha}} \|b\|_{L^\infty}\|\varphi\|_{B_{\infty,\infty}^{\beta}}t^{-\frac{\beta\wedge \frac 12}\alpha}} \label{CTR_DELTA_5}.
\end{align}

{For the last time-step contribution $\Delta^5$, we use the equality
\begin{align}\label{CONDITIONING_FOR_EULER_SCHEME}
\E[\nabla P_{t-s}^\alpha\varphi(X_s^h)|X_{\tau_s^h}^h,U_{\lfloor s/h\rfloor}]=\nabla P_{t-{\tau_s^h}}^\alpha\varphi(X_{\tau_s^h}^h+b(U_{\lfloor s/h\rfloor},X_{\tau_s^h}^h)(s-\tau_s^h)),
\end{align}
consequence of the independence of the increments of the driving noise and the harmonicity of the stable heat kernel,}
to write:
\begin{align}
|\Delta^5|\le& \int_{\tau_t^h-h}^t \Big|\int_{\R^d} \Gamma (0,x,s,z) b(s,z)\cdot \nabla P_{t-s}^\alpha \varphi(z) dz\Big|ds\notag\\
&+  \frac{1}h\int_{\tau_t^h-h}^t\int_{\tau_s^h}^{\tau_s^h+h} \Big|\int_{\R^d } \Gamma^h (0,x,\tau_s^h,z)b(r,z)\cdot\nabla P_{ t-\tau_s^h}^\alpha\varphi (z+b(r,z)(s-\tau_s^h)) dz  \Big|dr d s\notag\\
=:&\Delta^{51}+\Delta^{52}.\label{SPLIT_DELTA_6}
\end{align}

For $\Delta^{51} $,  
by \eqref{dual-ineq},
\begin{align}
|\Delta^{51}|\le& \int_{\tau_t^h-h}^t  \|\Gamma (0,x,s,\cdot) b(s,\cdot)\|_{B_{1,\infty}^{{\beta\wedge \frac 12}}} \|\nabla P_{t-s}^\alpha \varphi\|_{ B_{\infty,1}^{-{(\beta\wedge \frac 12)}}}ds\notag\\
\le&\int_{\tau_t^h-h}^t \|\Gamma (0,x,s,\cdot) b(s,\cdot)\|_{ B_{1,\infty}^{{\beta\wedge \frac 12}}} \|\nabla p_\alpha(t-s,\cdot)\|_{ B_{1,1}^{{-(\beta\wedge \frac 12)-\beta}}} \|\varphi\|_{ B_{\infty,\infty}^{{\beta}}}ds\notag\\
\lesssim& \int_{\tau_t^h-h}^t (1+s^{-\frac{{\beta\wedge \frac 12}}\alpha})(t-s)^{-\frac{(1-2\beta)\vee 0}\alpha}(1+|\ln(t-s)|\I_{{\beta= \frac 12}}) \|\varphi\|_{ B_{\infty,\infty}^{\beta}} ds \notag\\ 
\lesssim& h^{\frac{\alpha\wedge (\alpha-1+2\beta) }\alpha}(1+|\ln(h)|\I_{{\beta=\frac 12}})(1+t^{-\frac {{\beta\wedge \frac 12}}\alpha})\|\varphi\|_{B_{\infty,\infty}^\beta},\label{CTR_DELTA_61}
\end{align}
using as well \eqref{YOUNG} for the second inequality, \eqref{SING_STABLE_HK} and 
\eqref{BESOV_NORM_1} for the last but one inequality.
Note here that we have performed a dichotomy in function of $\beta$. If $\beta<\frac 12 $ then the total gain w.r.t. the control is the full exponent $\beta $, in the other case $\beta\ge \frac 12 $ (bigger H\"older exponent), we get to order 1 up to an additional logarithmic factor {for $\beta=\frac 12 $}. The point is precisely that we have to keep a non positive exponent for the {non} homogeneous Besov norm of the gradient of the heat kernel (see equation \eqref{SING_STABLE_HK}
}
).
{To deal with $\Delta^{52}$, we introduce a regularization $b_m:=b\star\psi_{\frac 1m}$, $m\in \N^*$, of $b$ by spatial convolution with a family $(\psi_{\varepsilon}(\cdot)=\varepsilon^{-d}\psi(\cdot/\varepsilon))_{\varepsilon> 0} $ of mollifyers.  Then
\begin{align}
  &\exists C<\infty,\;\forall m\in\N^*,\;\|b_m-b\|_{L^\infty-L^\infty}\le C m^{-\beta}\label{eq:bmbinf}\\
  &\exists C<\infty,\;\;ds\mbox{ a.e },\;\forall m\in\N^*,\;\sup_{y\neq y'\in\R^d}\frac{|b_m(s,y)-b_m(s,y')|}{|y-y'|}\le C m^{1-\beta}.\label{eq:bmlip}
\end{align}}

We write:
\begin{align*}
\Delta^{52}
=&
\frac{1}h\int_{\tau_t^h-h}^t\int_{\tau_s^h}^{\tau_s^h+h} \Big|\int_{\R^d } \Gamma^h (0,x,\tau_s^h,z)b(r,z)\cdot\nabla P_{ t-\tau_s^h}^\alpha\varphi (z+b_m(r,z)(s-\tau_s^h)) dz  \Big|dr \Big]d s\\
&+\frac{1}h\int_{\tau_t^h-h}^t\int_{\tau_s^h}^{\tau_s^h+h} \Big|\int_{\R^d } \Gamma^h (0,x,\tau_s^h,z)b(r,z)\\
&\cdot[\nabla P_{ t-\tau_s^h}^\alpha\varphi (z+b(r,z)(s-\tau_s^h))-\nabla P_{ t-\tau_s^h}^\alpha\varphi (z+b_m(r,z)(s-\tau_s^h))] dz  \Big|dr \Big]d s\\
=:&\Delta^{521}_m+\Delta^{522}_m.
\end{align*}
The previous regularization is performed in order to {ensure that $z\mapsto \varphi(z+b_m(r,z)(s-\tau_s^h)) $ is $\beta $-H\"older continuous according to \eqref{eq:bmlip}}.
Note first that, still by duality of Besov spaces, from the arguments used for $\Delta^{51} $ we obtain:
\begin{align}
|\Delta^{521}_m|\le& \frac 1h\int_{\tau_t^h-h}^t  \int_{\tau_s^h}^{\tau_s^h+h} \|\Gamma^h (0,x,\tau_s^h,\cdot) b(r,\cdot)\|_{  B_{1,\infty}^{{\beta\wedge \frac 12}}} \|\nabla P_{t-\tau_s^h}^\alpha \varphi(\cdot+b_{m}(r,\cdot)(s-\tau_s^h))\|_{ B_{\infty,1}^{-{(\beta\wedge \frac 12)}}} dr ds\notag\\
\le&\frac 1h\int_{\tau_t^h-h}^t \int_{\tau_s^h}^{\tau_s^h+h}\|\Gamma^h(0,x,\tau_s^h,\cdot) b(r,\cdot)\|_{ B_{1,\infty}^{{\beta\wedge \frac 12}}} \|\nabla p^\alpha(t-\tau_s^h,\cdot)\|_{  B_{1,1}^{{-{(\beta\wedge \frac 12)}-\beta}}} \|\varphi(\cdot+b_{m}(r,\cdot)(s-\tau_s^h))\|_{ B_{\infty,\infty}^{{\beta}}} dr ds\notag\\
\lesssim& \int_{\tau_t^h-h}^t (1+(\tau_s^h)^{-\frac{{\beta\wedge \frac 12}}\alpha})(t-\tau_s^h)^{-\frac{(1-2\beta)\vee 0}\alpha} (1+|\ln(t-s)|\I_{{\beta= \frac 12}})\|\varphi\|_{ B_{\infty,\infty}^{{\beta}}}(1+ (m^{1-\beta}h)^{{\beta}}) ds\notag\\
\lesssim & h^{\frac{\alpha \wedge(\alpha-1+2\beta) }\alpha}(1+|\ln(h)|\I_{{\beta= \frac 12}})(1+t^{-\frac{{\beta\wedge \frac 12}}\alpha})(1+(m^{1-\beta}h)^{{\beta} })\|\varphi\|_{B_{\infty,\infty}^\beta},\label{CTR_DELTA_61M}
\end{align}
{using \eqref{BESOV_NORM_2} for the last but one inequality}.
Observe then that, 
\begin{align*}
|\Delta^{522}_m|\le& \|b\|_{L^\infty}\int_{\tau_t^h-h}^{t}\int_{\R^d}\bar p_\alpha(\tau_s^h,z-x) \Big|\int_{\R^d}\int_0^1   D^2p_\alpha (t-\tau_s^h,z+[(1-\lambda)b_m+\lambda b](r,z)(s-\tau_s^h)-y)\\
&\times (b-b_m)(r,z)(s-\tau_s^h)\varphi(y)   d\lambda  dy \Big|dz ds.
\end{align*}
{Using now a cancellation argument in the inner integral w.r.t. to $dy$, i.e. observe that
\begin{align*}
\int_{\R^d}  D^2p_\alpha (t-\tau_s^h&,z+[(1-\lambda)b_m+\lambda b](r,z)(s-\tau_s^h)-y)dy\\
=&\int_{\R^d}  D_w^2p_\alpha (t-\tau_s^h,w-y)dy\Big|_{w=z+[(1-\lambda)b_m+\lambda b](r,z)(s-\tau_s^h)}\\
=& D_w^2\int_{\R^d}  p_\alpha (t-\tau_s^h,w-y)dy\Big|_{w=z+[(1-\lambda)b_m+\lambda b](r,z)(s-\tau_s^h)}=0,
\end{align*}
we get} {also using \eqref{eq:bmbinf}, the H\"older continuity of $\varphi$, \eqref{drift-smoothing-iso-noise} and \eqref{SING_STABLE_HK},}
\begin{align*}
|\Delta^{522}_m|\lesssim& \int_{\tau_t^h-h}^{t}\int_{\R^d}\bar p_\alpha(\tau_s^h,z-x) \Big| \int_{\R^d}\int_0^1   D^2p_\alpha (t-\tau_s^h,z+[(1-\lambda)b_m+\lambda b](r,z)(s-\tau_s^h)-y)\\
&\times [\varphi(y)-\varphi(z+[(1-\lambda)b_m+\lambda b](r,z)(s-\tau_s^h))]  d\lambda  dy \Big| |(b-b_m)(r,z)|(s-\tau_s^h) dz ds\\
&\lesssim m^{-\beta} \|\varphi\|_{B_{\infty,\infty}^\beta} \int_{\tau_t^h-h}^{t} (s-\tau_s^h) (t-\tau_s^h)^{-\frac 2\alpha+\frac{\beta}{\alpha}}ds\\
&\lesssim m^{-\beta} \|\varphi\|_{B_{\infty,\infty}^\beta} \int_{\tau_t^h-h}^{t}  (t-\tau_s^h)^{1-\frac 2\alpha+\frac{\beta}{\alpha}}ds.
\end{align*}
Since $\alpha-2+\beta>-\alpha \iff 2(\alpha-1)+\beta>0$, we eventually get:
\begin{align*}
|\Delta^{522}_m|
&\lesssim m^{-\beta} \|\varphi\|_{B_{\infty,\infty}^\beta} h^{\frac{2(\alpha-1)+\beta}{\alpha}}.
\end{align*}

Choosing thus $h m^{1-\beta}\simeq 1 \iff m=h^{-\frac{1}{1-\beta}}  $, in order to have a bounded contribution in $m$ in \eqref{CTR_DELTA_61M}, yields:
\begin{align*}
|\Delta^{522}_m|\lesssim&  h^{\frac{2(\alpha-1)+\beta}{\alpha}}h^{\frac {\beta}{1-\beta}}  \|\varphi\|_{B_{\infty,\infty}^\beta}
\lesssim h^{\frac{\alpha \wedge(\alpha-1+2\beta)}\alpha}\|\varphi\|_{B_{\infty,\infty}^\beta}.
\end{align*}
Indeed, {since $\alpha>1>1-\beta $, for $h\le 1 $},
$$h^{\frac {\beta}{1-\beta}}\le h^{\frac{\beta}{\alpha}} \Rightarrow h^{\frac{2(\alpha-1)+\beta}{\alpha}}h^{\frac {\beta}{1-\beta}}\le h^{\frac{2(\alpha-1)+2\beta}{\alpha}}\le h^{\frac{\alpha-1+2\beta}{\alpha}}{\le  h^{\frac{\alpha \wedge(\alpha-1+2\beta)}\alpha}}.
$$
This eventually yields, together with \eqref{CTR_DELTA_61M}, $|\Delta^{52}|\lesssim h^{\frac{\alpha\wedge (\alpha-1+2\beta) }{\alpha}}(1+|\ln(h)|\I_{{\beta= \frac 12}})(1+t^{-\frac{{\beta\wedge \frac 12}}\alpha})\|\varphi\|_{B_{\infty,\infty}^\beta} $. {Exploiting as well \eqref{CTR_DELTA_61}} gives:
\begin{align}
|\Delta^5|\le h^{\frac{\alpha\wedge (\alpha-1+2\beta)}{\alpha}}(1+|\ln(h)|\I_{{\beta= \frac 12}}) (1+t^{-\frac{{\beta\wedge \frac 12}}{\alpha}})\|\varphi\|_{B_{\infty,\infty}^{\beta}}\label{CTR_DELTA_6}.
\end{align}

Let us turn to $\Delta^3$. Write from \eqref{CONDITIONING_FOR_EULER_SCHEME}:
\begin{align}
\Delta^3:=&\int_h^{\tau_t^h-h} \E_{0,x} \left[b(U_{\lfloor s/h\rfloor},X_{\tau_s^h}^h)\cdot \left( \nabla P_{t-s}^\alpha \varphi(X_{\tau_s^h}^h)-\nabla  P_{t-s}^\alpha \varphi(X_s^h)\right)\right]d s\notag\\
=&\int_h^{\tau_t^h-h} \E_{0,x} \left[b(U_{\lfloor s/h\rfloor},X_{\tau_s^h}^h)\cdot \left( \nabla P_{t-s}^\alpha \varphi(X_{\tau_s^h}^h)-\nabla  P_{t-\tau_s^h}^\alpha \varphi(X_{\tau_s^h}^h+b(U_{\lfloor s/h\rfloor},X_{\tau_s^h}^h)(s-\tau_s^h))\right)\right]d s\notag\\
=&\int_h^{\tau_t^h-h} \E_{0,x} \left[b(U_{\lfloor s/h\rfloor},X_{\tau_s^h}^h)\cdot \left( \nabla P_{t-s}^\alpha \varphi(X_{\tau_s^h}^h)-\nabla  P_{t-\tau_s^h}^\alpha \varphi(X_{\tau_s^h}^h)\right)\right]d s\notag\\
&-\int_h^{\tau_t^h-h} \E_{0,x} \left[b(U_{\lfloor s/h\rfloor},X_{\tau_s^h}^h){\cdot}(\nabla  P_{t-\tau_s^h}^\alpha \varphi(X_{\tau_s^h}^h+b(U_{\lfloor s/h\rfloor},X_{\tau_s^h}^h)(s-\tau_s^h))-\nabla  P_{t-\tau_s^h}^\alpha \varphi(X_{\tau_s^h}^h))\right]ds\notag\\
=:&\Delta^{31}-\Delta^{32}.\label{DECOUP_Delta_3}
\end{align}
{In the above equation we have somehow split the time and space sensitivities for the gradient of the heat semigroup applied to the test function $\varphi $}.

Write first {from the duality inequality \eqref{dual-ineq}}: 

\begin{align}
|\Delta^{31}|\le&\int_h^{\tau_t^h-h} \frac 1h \int_{\tau_s^h}^{\tau_s^h+h} \Big|\int_{\R^d} \Gamma^h(0,x,\tau_s^h,z)b(r,z) \cdot \left( \nabla P_{t-s}^\alpha \varphi(z)-\nabla  P_{t-\tau_s^h}^\alpha \varphi(z)\right) dz \Big|dr ds,\notag\\
\lesssim& \int_h^{\tau_t^h-h} \frac 1h \int_{\tau_s^h}^{\tau_s^h+h} \|\Gamma^h(0,x,\tau_s^h,\cdot)b(r,\cdot)\|_{ B_{1,\infty}^{{\beta}}}\|\nabla P_{t-s}^\alpha \varphi-\nabla  P_{t-\tau_s^h}^\alpha \varphi\|_{ B_{\infty,1}^{-{\beta}}}dr ds\notag\\
\lesssim & \int_h^{\tau_t^h-h}  (1+(\tau_s^h)^{-\frac{{\beta}}\alpha}) \|\nabla p_{\alpha}(t-s,\cdot)-\nabla p_{\alpha}(t-\tau_s^h,\cdot)\|_{ B_{1,1}^{-2\beta}}ds \|\varphi\|_{B_{\infty,\infty}^\beta}\notag\\
\lesssim & h\int_h^{\tau_t^h-h}  (1+(\tau_s^h)^{-\frac{{\beta}}\alpha}) \int_0^1\|\partial_v\nabla p_{\alpha}(t-v,\cdot)|_{v_\lambda=\tau_s^h+\lambda (s-\tau_s^h)}\|_{ B_{1,1}^{-2\beta}}d\lambda ds \|\varphi\|_{B_{\infty,\infty}^\beta}\notag\\
\lesssim & h\int_h^{\tau_t^h-h}  (1+(\tau_s^h)^{-\frac{{\beta}}\alpha}) (t-\tau_s^h)^{-(1+\frac 1\alpha
-\frac{2\beta}\alpha)}ds \|\varphi\|_{B_{\infty,\infty}^\beta},\label{BD_FOR_DELTA_31_A_CITER}
\end{align}
using {\eqref{BESOV_NORM_2} and \eqref{YOUNG} for the third inequality} and \eqref{SING_STABLE_HK}, noting that {$\alpha+1-2\beta>0$}, for the last inequality.

Pay attention that, if $\beta>\frac 12\iff 1-2\beta<0 $ then the previous time integral {is not singular} and:
$$|\Delta^{31}|\lesssim h(1+t^{-\frac{1-\beta}{\alpha}})\|\varphi\|_{B_{\infty,\infty}^\beta}\le {h(1+t^{-\frac{1}{2\alpha}})\|\varphi\|_{B_{\infty,\infty}^\beta}}.$$
If now $\beta\le \frac 12 $,
\begin{align*}
\int_h^{\tau_t^h-h}  (1+(\tau_s^h)^{-\frac{\beta}\alpha}) (t-\tau_s^h)^{-(1+\frac 1\alpha-\frac{2\beta}\alpha)}ds&\le t^{-(1+\frac 1\alpha-\frac{2\beta}\alpha)}\int_{h}^{\frac t2}(1+s^{-\frac{\beta}\alpha})ds+(1+t^{-\frac \beta\alpha})\int_{\frac t2}^{\tau_t^h-h}(t-\tau_s^h)^{-(1+\frac 1\alpha-\frac{2\beta}\alpha)}ds\\
&\lesssim t^{-\frac{1-\beta}{\alpha}} +(1+t^{-\frac \beta\alpha})\big(|\ln(h)|\I_{\beta=\frac 12}+ h^{-\frac{1-2\beta}{\alpha}}\I_{\beta<\frac 12}\big)\\
&\lesssim(1+t^{-\frac \beta\alpha})\big(|\ln(h)|\I_{\beta=\frac 12}+ h^{-\frac{1-2\beta}{\alpha}}\I_{\beta<\frac 12}\big),
\end{align*}
for $h$ small enough.
Hence,
\begin{align}
|\Delta^{31}|&\lesssim \big[h(1+t^{-{\frac{1}{2\alpha}}})\I_{\beta>\frac 12}+(1+t^{-\frac \beta\alpha})\big(h|\ln(h)|\I_{\beta=\frac 12}+ h^{\frac{\alpha-1+2\beta}{\alpha}}\I_{\beta<\frac 12}\big)
\big]\|\varphi\|_{B_{\infty,\infty}^\beta}\notag\\
&\lesssim h^{\frac{\alpha\wedge (\alpha-1+ 2\beta)}{\alpha}}(1+t^{-\frac{{\beta\wedge \frac 12}}{\alpha}})(1+|\ln(h)|\I_{\beta=\frac 12}){\|\varphi\|_{B_{\infty,\infty}^\beta}}.
\label{CTR_Delta_31}
\end{align}

Write then, 

\begin{align*}
\Delta^{32}&=\int_h^{\tau_t^h-h} \E_{0,x} \left[b(U_{\lfloor s/h\rfloor},X_{\tau_s^h}^h){\cdot}(\nabla  P_{t-\tau_s^h}^\alpha \varphi(X_{\tau_s^h}^h+b(U_{\lfloor s/h\rfloor},X_{\tau_s^h}^h)(s-\tau_s^h))-\nabla  P_{t-\tau_s^h}^\alpha \varphi(X_{\tau_s^h}^h))\right]ds\\
&=\int_h^{\tau_t^h-h} \frac 1h \int_{\tau_s^h}^{\tau_s^h+h} \int_{\R^d} \Gamma^h(0,x,{\tau^h_s},z)b(r,z){\cdot}\Big(\nabla  P_{t-\tau_s^h}^\alpha \varphi(z+b(r,z)(s-\tau_s^h))-\nabla  P_{t-\tau_s^h}^\alpha \varphi(z)\Big)dz dr ds\\
&=\int_h^{\tau_t^h-h} \frac 1h \int_{\tau_s^h}^{\tau_s^h+h} \int_{\R^d} \Gamma^h(0,x,{\tau^h_s},z)\int_0^1 d\lambda \Big\langle D^2 P_{t-\tau_s^h}^\alpha \varphi(z+\lambda b(r,z)(s-\tau_s^h))b(r,z),b(r,z)\Big\rangle (s-\tau_s^h) dz dr ds\\
&=\sum_{i,j=1}^d\int_h^{\tau_t^h-h} \frac 1h \int_{\tau_s^h}^{\tau_s^h+h} \int_0^1 d\lambda \int_{\R^d}\Gamma^h(0,x,{\tau^h_s},z)b_ib_j(r,z) D_{ij}^2 P_{t-\tau_s^h}^\alpha \varphi(z+\lambda b(r,z)(s-\tau_s^h)) (s-\tau_s^h) dz dr ds.
\end{align*}
Taking now a mollifier $b_m$ of $b$ as above for the term $\Delta^5 $, we derive:
\begin{align}
|\Delta^{32}|\lesssim& \sum_{i,j=1}^d\int_h^{\tau_t^h-h} \frac 1h \int_{\tau_s^h}^{\tau_s^h+h} \int_0^1\Big|\int_{\R^d}\Gamma^h(0,x,{\tau^h_s},z)b_ib_j(r,z) D_{ij}^2 P_{t-\tau_s^h}^\alpha \varphi(z+\lambda b_m(r,z)(s-\tau_s^h)) dz \Big|dr (s-\tau_s^h)  dsd\lambda\notag\\
&+\sum_{i,j=1}^d\int_h^{\tau_t^h-h} \frac 1h \int_{\tau_s^h}^{\tau_s^h+h} \int_0^1\Big|\int_{\R^d}\Gamma^h(0,x,{\tau^h_s},z)b_ib_j(r,z) \Big(D_{ij}^2 P_{t-\tau_s^h}^\alpha \varphi(z+\lambda b(r,z)(s-\tau_s^h)) \notag\\
&\quad\quad - D_{ij}^2 P_{t-\tau_s^h}^\alpha \varphi(z+\lambda b_m(r,z)(s-\tau_s^h)) \Big)  dz\Big| dr (s-\tau_s^h) dsd\lambda=:\Delta^{321}_m+\Delta^{322}_m.\label{DECOUP_Delta_32}
\end{align}
{We deal with $\Delta^{321}_m $ similarly to the term $\Delta^{521}_m $ above using \eqref{dual-ineq}, $b_ib_j\in L^\infty([0,T],B^\beta_{\infty,\infty})$, \eqref{BESOV_NORM_2}, \eqref{YOUNG} and \eqref{SING_STABLE_HK}. This gives,}
\begin{align}
|\Delta^{321}_m|\lesssim& h\sum_{i,j=1}^d\int_h^{\tau_t^h-h} \frac 1h \int_{\tau_s^h}^{\tau_s^h+h} \int_0^1\|\Gamma^h(0,x,{\tau^h_s},\cdot)b_ib_j(r,\cdot)\|_{B_{1,\infty}^{\beta}} \|D_{ij}^2 P_{t-\tau_s^h}^\alpha \varphi(\cdot+\lambda b_m(r,\cdot)(s-\tau_s^h))\|_{ B_{\infty,1}^{-\beta}} dr  dsd\lambda\notag\\
\lesssim&h\sum_{i,j=1}^d\int_h^{\tau_t^h-h} \frac 1h \int_{\tau_s^h}^{\tau_s^h+h} \int_0^1 (1+s^{-\frac \beta\alpha})\|D_{ij}^2 p_\alpha(t-\tau_s^h,\cdot)\|_{B_{1,1}^{-2\beta}}\| \varphi(\cdot+\lambda b_m(r,\cdot)(s-\tau_s^h))\|_{B_{\infty,\infty}^{\beta}} dr ds d\lambda\notag\\
\lesssim&h\|\varphi\|_{B_{\infty,\infty}^\beta}\sum_{i,j=1}^d\int_h^{\tau_t^h-h} \frac 1h \int_{\tau_s^h}^{\tau_s^h+h} (1+s^{-\frac \beta\alpha})(t-\tau_s^h)^{-\frac{2(1-\beta)}{\alpha}}(1+(m^{1-\beta}h)^\beta)  dr ds.\notag
\end{align}
The behavior of this term now depends on the indices. If $2(1-\beta)<\alpha \iff 1<\beta+\frac \alpha 2$, which is always satisfied for $\alpha\in (1,2] $ provided $\beta>\frac 12 $, the previous integral is non singular and:
$$ |\Delta^{321}_m|\lesssim h t^{\frac{\alpha-2+\beta}{\alpha}}(1+(m^{1-\beta}h)^\beta)\|\varphi\|_{B_{\infty,\infty}^\beta}\lesssim {h t^{-\frac{\beta\wedge \frac 12}{\alpha}}(1+(m^{1-\beta}h)^\beta)\|\varphi\|_{B_{\infty,\infty}^\beta}},$$
{where for the last inequality we observe that, if $\beta<\frac 12 $, since $\alpha-2+2\beta>0,\;t^{\frac{\alpha-2+\beta}{\alpha}}= t^{\frac{\alpha-2+2\beta}{\alpha}}t^{-\frac{\beta}{\alpha}} \lesssim t^{-\frac{\beta}{\alpha}}$, and {if $\beta\ge \frac 12 $, we have $\alpha-2+\beta>-\frac 12$.}}

If now {$2(1-\beta)\ge \alpha \iff \beta\le 1-\frac \alpha 2 $, which in particular gives that in this case ${\beta<\frac 12 }$}, we have
:
\begin{align*}
|\Delta^{321}_m|\lesssim &\Big(h^{2\frac{\alpha-1+\beta}{\alpha}}\I_{2(1-\beta)> \alpha} +h|\ln(h)|\I_{2(1-\beta)=\alpha}\Big)t^{-\frac{\beta}{\alpha}}(1+(m^{1-\beta}h)^\beta)\|\varphi\|_{B_{\infty,\infty}^\beta}.
\end{align*}
In any case, we derive:
\begin{align}\label{Delta_321_m}
|\Delta^{321}_m|&\lesssim 
(h^{\frac{\alpha \wedge (\alpha-1+2\beta)}{\alpha}}\I_{2(1-\beta)\neq\alpha}+h|\ln(h)|\I_{2(1-\beta)=\alpha})t^{-\frac{{\beta\wedge \frac 12}}{\alpha}} \|\varphi\|_{B_{\infty,\infty}^\beta}\notag\\
&\lesssim h^{\frac{\alpha \wedge (\alpha-1+2\beta)}{\alpha}}t^{-\frac{{\beta\wedge \frac 12}}{\alpha}} \|\varphi\|_{B_{\infty,\infty}^\beta},
\end{align}
provided $m^{1-\beta}h\lesssim 1$ as above, noting as well for the last inequality that if we have $1-\beta=\frac\alpha2 $, necessarily $\beta<\frac 12 $, since $\alpha>1 $. Turning to the other contribution, we get {using \eqref{ARONS_SCHEME} and a cancellation argument similar to the one used for $\Delta_m^{522} $:
\begin{align*}
|\Delta^{322}_m|\lesssim& \int_{h}^{\tau_t^h-h}\int_{\tau_s^h}^{\tau_s^h+h}\int_{\R^d}\bar p_\alpha(\tau_s^h,z-x) \int_0^1 \Big|\int_0^1 \int_{\R^d}  D^3p_\alpha (t-\tau_s^h,z+\lambda[(1-\mu)b_m+\mu  b](r,z)(s-\tau_s^h)-y)\\
&\times\lambda (b-b_m)(r,z)(s-\tau_s^h)\varphi(y)    dy d\mu\Big|{d\lambda}dz dr ds\\
\lesssim& \int_h^{\tau_t^h-h}\int_{\tau_s^h}^{\tau_s^h+h}\int_{\R^d}\bar p_\alpha(\tau_s^h,z-x)\int_0^1 \Big| \int_0^1 \int_{\R^d}  D^3p_\alpha (t-\tau_s^h,z+\lambda[(1-\mu)b_m+\mu b](r,z)(s-\tau_s^h)-y)\\
&\times [\varphi(y)-\varphi(z+\lambda[(1-\mu)b_m+\mu b](r,z)(s-\tau_s^h))]    dy d\mu\Big| \lambda |(b-b_m)(r,z)|(s-\tau_s^h) dz dr ds\\
&\lesssim m^{-\beta} \|\varphi\|_{B_{\infty,\infty}^\beta} \int_h^{\tau_t^h-h} h^2 (t-\tau_s^h)^{-\frac 3\alpha+\frac{\beta}{\alpha}}ds\le m^{-\beta}h^{\frac{3(\alpha-1)+\beta}{\alpha}}\|\varphi\|_{B_{\infty,\infty}^\beta}
\lesssim h^{\frac{3(\alpha-1)+\beta}{\alpha}+\frac{\beta}{1-\beta}} \|\varphi\|_{B_{\infty,\infty}^\beta}\\
&\lesssim h^{\frac{\alpha \wedge (\alpha-1+2\beta)}{\alpha}} \|\varphi\|_{B_{\infty,\infty}^\beta}, 
\end{align*}
{choosing again $m=h^{-\frac{1}{1-\beta}} $ for the last but one inequality}. Together with \eqref{Delta_321_m} and \eqref{DECOUP_Delta_32} this yields:
\begin{align*}
|\Delta^{32}|\lesssim h^{\frac{\alpha \wedge (\alpha-1+2\beta)}{\alpha}}(1+t^{-\frac{{\beta\wedge \frac 12}}{\alpha}}) \|\varphi\|_{B_{\infty,\infty}^\beta}.
\end{align*}
From the above equation and \eqref{CTR_Delta_31}, we eventually derive:
\begin{align}\label{CTR_DELTA_3}
|\Delta^3|\lesssim h^{\frac{\alpha\wedge (\alpha-1+ 2\beta)}{\alpha}}(1+t^{-\frac{{\beta\wedge \frac 12}
  }{\alpha}})(1+|\ln(h)|\I_{\beta=\frac 12})\|\varphi\|_{B_{\infty,\infty}^\beta}.
\end{align}
Let us now turn to the term $\Delta^4 $ which can somehow be handled as the previous contribution $\Delta^{31} $, {with the very same arguments, namely \eqref{dual-ineq}, \eqref{BESOV_NORM_2}, \eqref{SING_STABLE_HK}. Write:}
\begin{align}
|\Delta^4|=&\left|\int_h^{\tau_t^h-h} \E_{0,x} \left[b(U_{\lfloor s/h\rfloor},X_{\tau_s^h}^h)\cdot \left( \nabla P_{t-U_{\lfloor s/h\rfloor}}^\alpha\varphi(X_{\tau_s^h}^h)-\nabla  P_{t-s}^\alpha \varphi(X_{\tau_s^h}^h)\right)\right]d s\right|\notag\\
=&\left|\int_h^{\tau_t^h-h} \frac{1}{h}\int_{\tau_s^h}^{\tau_s^h+h} \int_{\R^d}\Gamma^h(0,x,\tau_s^h,z) b(r,z) \left( \nabla P_{t-r}^\alpha\varphi(z)-\nabla  P_{t-s}^\alpha \varphi(z)\right) dz dr ds\right|\notag\\
\lesssim& \int_h^{\tau_t^h-h} \frac{1}{h}\int_{\tau_s^h}^{\tau_s^h+h} \|\Gamma^h(0,x,\tau_s^h,\cdot) b(r,\cdot)\|_{B_{1,\infty}^\beta}\|(\nabla P_{t-r}^\alpha-\nabla  P_{t-s}^\alpha )\varphi\|_{B_{\infty,1}^{-\beta}} drds \notag\\
\lesssim& \int_h^{\tau_t^h-h} \frac{1}{h}\int_{\tau_s^h}^{\tau_s^h+h}{(1+(\tau_s^h)^{-\frac{\beta}\alpha})} \|\nabla p_\alpha(t-r,\cdot)-\nabla p_\alpha(t-s,\cdot)\|_{B_{1,1}^{-2\beta}}dr ds \|\varphi\|_{B_{\infty,\infty}^\beta}\notag\\
\lesssim& \int_h^{\tau_t^h-h} \frac{1}{h}\int_{\tau_s^h}^{\tau_s^h+h} (1+(\tau_s^h)^{-\frac{\beta}\alpha}) \int_{0}^1\|\partial_v\nabla p_\alpha (t-v,\cdot) |_{v=r\wedge s+\lambda(r\vee s-r\wedge s)}\|_{B_{1,1}^{-2\beta}}|r-s|d\lambda dr ds\|\varphi\|_{B_{\infty,\infty}^\beta}\notag\\
\lesssim& h \int_h^{\tau_t^h-h} \frac{1}{h}\int_{\tau_s^h}^{\tau_s^h+h} (1+(\tau_s^h)^{-\frac{\beta}\alpha}) (t-r\vee s)^{-\big(1+\frac 1\alpha-\frac {2\beta}\alpha\big)} dr ds\|\varphi\|_{B_{\infty,\infty}^\beta}\notag\\
\lesssim & 
h^{\frac{\alpha\wedge (\alpha-1+ 2\beta)}{\alpha}}(1+t^{-\frac{{\beta\wedge \frac 12}}{\alpha}})(1+|\ln(h)|\I_{\beta=\frac 12}){\|\varphi\|_{B_{\infty,\infty}^\beta}}, 
\label{CTR_DELTA_4}
\end{align}
noting that we can invoke \eqref{SING_STABLE_HK} since $\alpha+1-2\beta>0 $ and reproducing a similar discussion as for the term $\Delta^{31} $ for the last {inequality}.

Let us now turn on the time sensitivity:
\begin{align}
|\Delta^2|=&\left|\int_h^{\tau_t^h-h} \E_{0,x} \left[ b(s,X_s)\cdot \nabla P_{t-s}^\alpha \varphi(X_s)-b(s,X_{\tau_s^h})\cdot \nabla  P_{t-s}^\alpha \varphi (X_{\tau_s^h}) 
        \right] d s \right|\notag\\
        \le & \int_h^{\tau_t^h-h} \Big|\int_{\R^d} [\Gamma(0,x,s,z)-\Gamma(0,x,\tau_s^h,z)]  b(s,z)\cdot \nabla P_{t-s}^\alpha \varphi(z) dz\Big| ds\notag\\
        \le &\int_h^{\tau_t^h-h} \|\Gamma(0,x,s,\cdot)-\Gamma(0,x,\tau_s^h,\cdot)\|_{B_{1,1}^{-\beta}}\|b(s,\cdot) \nabla P_{t-s}^\alpha \varphi\|_{B_{\infty,\infty}^\beta} ds,\label{PREAL_BD_Delta_2}
\end{align}
using once again a duality argument {(see \eqref{dual-ineq})}. Note that it is for this term that the somehow \textit{natural} duality argument appears since it precisely isolates the drift and somehow gives the weaker norm to be investigated to derive a larger convergence rate w.r.t the former work \cite{fito:meno:25}.

Now, on the one hand it readily holds that
\begin{align}
\label{TEST_FUNC_Delta_2}
\|b(s,\cdot) \nabla P_{t-s}^\alpha \varphi\|_{B_{\infty,\infty}^\beta}&\le \|b(s,\cdot)\|_{B_{\infty,\infty}^\beta}\|\nabla P_{t-s}^\alpha \varphi\|_{L^\infty}+\|b(s,\cdot)\|_{L^\infty}\|\nabla P_{t-s}^\alpha \varphi\|_{B_{\infty,\infty}^\beta}\notag\\
&\lesssim \|b(s,\cdot)\|_{B_{\infty,\infty}^\beta} (t-s)^{-\frac{1-\beta}{\alpha}}\|\varphi\|_{B_{\infty,\infty}^\beta}+\|b(s,\cdot)\|_{L^\infty}\|\nabla p^\alpha(t-s,\cdot)\|_{B_{1,1}^{0}}\|\varphi\|_{B_{\infty,\infty}^\beta}\notag\\
&\lesssim \|b(s,\cdot)\|_{B_{\infty,\infty}^\beta} (t-s)^{-\frac{1}{\alpha}}\|\varphi\|_{B_{\infty,\infty}^\beta},
\end{align}
using the Young inequality \eqref{YOUNG} {and \eqref{GAIN_TEST_FUNC} for the last but one} inequality.

On the other hand, we have the following regularity result whose proof is postponed to Section \ref{SEC_PREUVE_TEC}:
\begin{PROP}[Time regularity in the forward variable]\label{PROP_FWD_TIME_REG}
Under the current assumptions {and, in particular, $\alpha-1>\beta$,} it holds that for $s\in [h,t]$:
\begin{align}
\label{NORM_NEG_HK}
\|\Gamma(0,x,s,\cdot)-\Gamma(0,x,\tau_s^h,\cdot)\|_{B_{1,1}^{-\beta}}\lesssim  h^{\frac{\alpha\wedge (\alpha-1+ 2\beta)}{\alpha}}(1+s^{-\frac{\alpha-(\beta\vee (1-\beta))}{\alpha}}+{s^{-\frac{1}{2\alpha}}\I_{\beta>\alpha-\frac 12}})(1+|\ln(h)|\I_{{\beta= \frac 12}}).
\end{align}
\end{PROP}
From \eqref{NORM_NEG_HK} and \eqref{TEST_FUNC_Delta_2} we readily get in \eqref{PREAL_BD_Delta_2}:
\begin{align}
|\Delta^2| \lesssim &h^{\frac{\alpha \wedge (\alpha-1+2\beta)}{\alpha}} (1+|\ln(h)|\I_{{\beta= \frac 12}})\int_{h}^{\tau_t^h-h} (1+s^{-\frac{\alpha-(\beta\vee (1-\beta))}{\alpha}}+{s^{-\frac{1}{2\alpha}}\I_{\beta>\alpha-\frac 12}})(t-s)^{-\frac 1\alpha}ds\|\varphi\|_{B_{\infty,\infty}^\beta}\notag\\
\lesssim &h^{\frac{\alpha \wedge (\alpha-1+2\beta)}{\alpha}} (1+|\ln(h)|\I_{{\beta= \frac 12}})(1+t^{-\frac{1-(\beta\vee (1-\beta))}{\alpha}}+{t^{\frac{\alpha-1-\frac 12}\alpha}\I_{\beta>\alpha-\frac 12}})\|\varphi\|_{B_{\infty,\infty}^\beta}\notag\\
\lesssim &h^{\frac{\alpha \wedge (\alpha-1+2\beta)}{\alpha}} (1+|\ln(h)|\I_{{\beta= \frac 12}})(1+t^{-\frac{\beta\wedge (1-\beta)}{\alpha}}+{t^{-\frac{1}{2\alpha}}\I_{\beta>\alpha-\frac 12}})\|\varphi\|_{B_{\infty,\infty}^\beta}\notag\\
\lesssim &h^{\frac{\alpha \wedge (\alpha-1+2\beta)}{\alpha}} (1+|\ln(h)|\I_{{\beta= \frac 12}})(1+t^{-\frac{{\beta\wedge \frac 12}}{\alpha}})\|\varphi\|_{B_{\infty,\infty}^\beta}.\label{CTR_DELTA_2}
\end{align}

It thus remains to handle the Gronwall type term $\Delta^1 $. Namely,
\begin{align}
\Delta^1&=\int_h^{\tau_t^h-h} \E_{0,x} \bigg[ b(s,X_{\tau_s^h})\cdot \nabla P_{t-s}^\alpha \varphi(X_{\tau_s^h})
        -b(s,X_{\tau_s^h}^h)\cdot \nabla  P_{t-s}^\alpha \varphi(X_{\tau_s^h}^h) \bigg] d s\notag\\
        &=\int_h^{\tau_t^h-h} \int_{\R^d} \Big[\Gamma(0,x,\tau_s^h,z)-\Gamma^h(0,x,\tau_s^h,\cdot)\Big] b(s,z)\cdot \nabla  P_{t-s}^\alpha \varphi(z)dz ds,\notag\\
|\Delta^1|&\lesssim \int_h^{\tau_t^h-h} \|\Gamma(0,x,\tau_s^h,\cdot)-\Gamma^h(0,x,\tau_s^h,\cdot)\|_{B_{1,1}^{-\beta}}     \|b(s,\cdot)\nabla  P_{t-s}^\alpha \varphi\|_{B_{\infty,\infty}^\beta} ds\notag\\
&\lesssim \|b\|_{L^\infty-B_{\infty,\infty}^\beta}\|\varphi\|_{B_{\infty,\infty}^\beta}\int_h^{\tau_t^h-h} \|\Gamma(0,x,\tau_s^h,\cdot)-\Gamma^h(0,x,\tau_s^h,\cdot)\|_{B_{1,1}^{-\beta}}   (t-s)^{-\frac{1}{\alpha}}   ds,\label{CTR_DELTA_1}
\end{align}
using again \eqref{TEST_FUNC_Delta_2} for the last inequality. Combining \eqref{CTR_DELTA_5}, \eqref{CTR_DELTA_6}, \eqref{CTR_DELTA_3}, \eqref{CTR_DELTA_4}, \eqref{CTR_DELTA_2}, we get that\begin{equation*}
   \sum_{i=2}^6|\Delta^i|\le C_{(2-6)}h^{\frac{\alpha\wedge (\alpha-1+2\beta)}{\alpha}}(1+|\ln(h)|\I_{{\beta=\frac 12}})t^{-\frac{{\beta\wedge \frac 12}}{\alpha}}\|\varphi\|_{B_{\infty,\infty}^\beta}.
 \end{equation*} We deduce that the function $f(t){:=}t^{\frac{{\beta\wedge \frac 12} }\alpha}\|\Gamma(0,x,t,\cdot)-\Gamma^h(0,x,t,\cdot)\|_{B_{1,1}^{-\beta}}$ which is bounded on $[0,T]$ according {to the Duhamel representations \eqref{duhamel-Diff} and \eqref{duhamel-schemeb}}, \eqref{ARONS_DIFF}, \eqref{ARONS_SCHEME} and \eqref{SING_STABLE_HK}% \footnote{\tr{B to S: j'ai l'impression d'obtenir $\sup_{t\in[0,T]}t^{-\frac{\beta}{\alpha}}\|(\Gamma-\Gamma^h)(0,x,t,\cdot)\|_{B_{1,1}^{-\beta}}<\infty$.}\textcolor{black}{De S. à B.: J'ai modifié \eqref{SING_STABLE_HK}, où l'exposant apparaît maintenant avec une partie positive. En fait il ne pouvait être fait appel à l'estimée précédente car on avait pas la condition de positivité justement. Mais en tout cas c'est bien borné. Cela suffit de justifier comme cela où il faudrait refaire un passage de justif?}}
 also satisfies
$$\forall t\in[0,T],\;f(t)\le C_{(2-6)}h^{\frac{\alpha\wedge (\alpha-1+2\beta)}{\alpha}}(1+|\ln(h)|\I_{{\beta=\frac 12}})+Ct^{\frac{{\beta\wedge \frac12}}{\alpha}}\int_{h}^{h\vee t}f(\tau^h_s)s^{-\frac{ {\beta\wedge \frac 12}}\alpha}(t-s)^{-\frac 1\alpha}ds.$$
Since  $\frac \beta\alpha<\frac 1\alpha<1$, we conclude with
Lemma \ref{lemgron}, {in the current case  $\alpha-1>\beta $}, which is the condition we used in order to have the control of the $ \beta$-Hölder modulus of the density of the Euler scheme in the forward variable.

\subsubsection{Modification of the proof when $\alpha-1\le \beta $.}
\label{MORE_GEN_FOR_PROOF}
We here need to specify how to handle the terms $ \Delta^3,\Delta^4,\Delta^5$ in \eqref{DECOUP_ERR} when $\alpha-1\le \beta $ {and \eqref{BESOV_NORM_2} is not available}.
We will discuss in detail $\Delta^3$, the other contributions can be handled following the same lines. Recall from \eqref{DECOUP_Delta_3} that $\Delta^3=\Delta^{31}\textcolor{red}{-}\Delta^{32} $, where:
\begin{align*}
\Delta^{31}=&\int_h^{\tau_t^h-h} \frac 1h \int_{\tau_s^h}^{\tau_s^h+h} \int_{\R^d} \Gamma^h(0,x,\tau_s^h,z)b(r,z) \cdot \left( \nabla P_{t-s}^\alpha \varphi(z)-\nabla  P_{t-\tau_s^h}^\alpha \varphi(z)\right) dz dr ds,\notag\\
\Delta^{32}=&\sum_{i,j=1}^d\int_h^{\tau_t^h-h} \frac 1h \int_{\tau_s^h}^{\tau_s^h+h} \int_0^1 d\lambda \int_{\R^d}\Gamma^h(0,x,s,z)b_ib_j(r,z) D_{ij}^2 P_{t-\tau_s^h}^\alpha \varphi(z+\lambda b(r,z)(s-\tau_s^h)) (s-\tau_s^h) dz dr ds.
\end{align*}
The difficulty is here that the previous duality results used in the above approach cannot be directly used since we do not control the $ \beta$-Hölder modulus of the density of the Euler scheme in the forward variable. However such a control holds for the {diffusion} and the main result from \cite{fito:meno:25} actually gives a pointwise control of the error. Namely, from Theorem 1 of the quoted reference it holds that
 there exists a constant $C:=C(\textcolor{black}{d},b,\alpha,T)<\infty$ s.t. for  all $h=T/n$ with $n\in\N^*$, and all $t\in(0,T]$, $x,y\in \R^d $,
		\begin{align}
			|\Gamma^h(0,x,t,y)-\Gamma(0,x,t,y)| &\leq C\big(1+t^{-\frac \beta\alpha}\big)  h^{\frac{\gamma}{\alpha}} \bar p_\alpha(t,y-x),\ \gamma=\alpha-1+\beta.\label{BD_THM_FITO_MENO_SPA}
        \end{align}
The idea is then to exploit this bound to \textit{replace} $\Gamma^h $ by $\Gamma$ in the above expression, up to an additional error term, which will be shown to be negligible w.r.t. to the final sought order. Write first:

\begin{align*}
\Delta^{31}=&\int_h^{\tau_t^h-h} \frac 1h \int_{\tau_s^h}^{\tau_s^h+h} \int_{\R^d} (\Gamma^h-\Gamma)(0,x,\tau_s^h,z)b(r,z) \cdot \left( \nabla P_{t-s}^\alpha \varphi(z)-\nabla  P_{t-\tau_s^h}^\alpha \varphi(z)\right) dz dr ds,\notag\\
&+\int_h^{\tau_t^h-h} \frac 1h \int_{\tau_s^h}^{\tau_s^h+h} \int_{\R^d} \Gamma(0,x,\tau_s^h,z)b(r,z) \cdot \left( \nabla P_{t-s}^\alpha \varphi(z)-\nabla  P_{t-\tau_s^h}^\alpha \varphi(z)\right) dz dr ds:=\bar \Delta^{311}+\bar \Delta^{312}.
\end{align*}
The term $\bar \Delta^{312} $ can be analyzed as the former $\Delta^{31} $ in the previous section exploiting \eqref{BESOV_NORM_1} for the diffusion, which is also valid under the current condition. This again yields:
\begin{align}\label{MAIN_OTHER_COND_D_312}
|\bar \Delta^{312}|\lesssim &h^{\frac{\alpha\wedge (\alpha-1+ 2\beta)}{\alpha}}(1+t^{-\frac{\beta\wedge (1-\beta)}{\alpha}})(1+|\ln(h)|\I_{\beta=\frac 12})%\notag\\
{\|\varphi\|_{B_{\infty,\infty}^\beta}}\lesssim %&
h^{\frac{\alpha\wedge (\alpha-1+ 2\beta)}{\alpha}}(1+t^{-\frac{{\beta\wedge \frac 12}}{\alpha}})(1+|\ln(h)|\I_{\beta=\frac 12}){\|\varphi\|_{B_{\infty,\infty}^\beta}}.
\end{align}

For the other term we obtain from \eqref{BD_THM_FITO_MENO_SPA}:
\begin{align*}
|\bar \Delta^{311}|\lesssim h^{\frac{\alpha-1+\beta}{\alpha}} \int_h^{\tau_t^h-h} \int_{\R^d} \bar p(\tau_s^h,z-x)(1+(\tau_s^h)^{-\frac \beta \alpha})|\nabla P_{t-s}^\alpha \varphi(z)-\nabla  P_{t-\tau_s^h}^\alpha \varphi(z)|dz ds.
\end{align*}
Observe now that, from \eqref{EMBEDDING} and \eqref{YOUNG}
\begin{align*}
|\nabla P_{t-s}^\alpha \varphi(z)-\nabla  P_{t-\tau_s^h}^\alpha \varphi(z)|&{\le \|\nabla P_{t-s}^\alpha \varphi-\nabla  P_{t-\tau_s^h}^\alpha \varphi\|_{B_{\infty,1}^0} } \le\|\nabla p_\alpha(t-s,\cdot)-\nabla p_\alpha(t-\tau_s^h,\cdot)\|_{B_{1,1}^{-\beta}}\|\varphi\|_{B_{\infty,\infty}^\beta}\\
&\lesssim \frac{(s-\tau_s^h)}{(t-s)^{1+\frac{1-\beta}{\alpha}}}\|\varphi\|_{B_{\infty,\infty}^\beta},
\end{align*}
{where the last inequality again follows from \eqref{SING_STABLE_HK}}.
Hence,
\begin{align*}
|\bar \Delta^{311}|\lesssim &h^{\frac{\alpha-1+\beta}{\alpha}} \int_h^{\tau_t^h-h} (1+(\tau_s^h)^{-\frac \beta \alpha}) \frac{(s-\tau_s^h)}{(t-s)^{1+\frac{1-\beta}{\alpha}}}ds\|\varphi\|_{B_{\infty,\infty}^\beta} \lesssim h^{2\frac{\alpha-1+\beta}{\alpha}}(1+t^{-\frac \beta\alpha})\|\varphi\|_{B_{\infty,\infty}^\beta}\\
\lesssim&h^{\frac{\alpha\wedge (\alpha-1+2\beta}{\alpha})}(1+t^{-\frac{{\beta\wedge \frac 12}}\alpha})\|\varphi\|_{B_{\infty,\infty}^\beta}.
\end{align*}
{Indeed, the control {follows from $\alpha-1>0$} if $\beta\le \frac 12 $ and for $\beta>\frac 12 $, 
\begin{align*}
h^{2\frac{\alpha-1+\beta}{\alpha}}(1+t^{-\frac \beta\alpha})\lesssim h h^{\frac{\alpha-2+2\beta}{\alpha}}(1+t^{-\frac \beta\alpha})\lesssim h t^{\frac{\alpha-2+\beta}{\alpha}}{\lesssim h t^{-\frac{1}{2\alpha}},}
\end{align*}
{since $\alpha-2+\beta>-\frac 12$}}. 

This contribution is therefore negligible compared to \eqref{MAIN_OTHER_COND_D_312}, whose upper bound then remains valid for $|\Delta^{31}|$.

Similarly, we decompose:
\begin{align*}
&\Delta^{32}\\
=&\sum_{i,j=1}^d\int_h^{\tau_t^h-h} \frac 1h \int_{\tau_s^h}^{\tau_s^h+h} \int_0^1 d\lambda \int_{\R^d}\Gamma^h(0,x,{\tau^h_s},z)b_ib_j(r,z) D_{ij}^2 P_{t-\tau_s^h}^\alpha \varphi(z+\lambda b(r,z)(s-\tau_s^h)) (s-\tau_s^h) dz dr ds\\
=&\sum_{i,j=1}^d\int_h^{\tau_t^h-h} \frac 1h \int_{\tau_s^h}^{\tau_s^h+h} \int_0^1 d\lambda \int_{\R^d}(\Gamma^h-\Gamma)(0,x,{\tau^h_s},z)b_ib_j(r,z) D_{ij}^2 P_{t-\tau_s^h}^\alpha \varphi(z+\lambda b(r,z)(s-\tau_s^h)) (s-\tau_s^h) dz dr ds\\
&+\sum_{i,j=1}^d\int_h^{\tau_t^h-h} \frac 1h \int_{\tau_s^h}^{\tau_s^h+h} \int_0^1 d\lambda \int_{\R^d}\Gamma(0,x,{\tau^h_s},z)b_ib_j(r,z) D_{ij}^2 P_{t-\tau_s^h}^\alpha \varphi(z+\lambda b(r,z)(s-\tau_s^h)) (s-\tau_s^h) dz dr ds\\
=:&\bar \Delta^{321}+\bar \Delta^{322}.
\end{align*}
The term $\bar \Delta^{322} $ can again be handled as the former $\Delta^{32} $ using \eqref{BESOV_NORM_1} for the diffusion. This yields: 
\begin{align}\label{BAR_Delta_322}
|\bar \Delta^{322}|\lesssim h^{\frac{\alpha \wedge (\alpha-1+2\beta)}{\alpha}}(1+t^{-\frac{{\beta\wedge \frac 12}}{\alpha}}) \|\varphi\|_{B_{\infty,\infty}^\beta}.
\end{align}
On the other hand, from \eqref{BD_THM_FITO_MENO_SPA} {and \eqref{GAIN_TEST_FUNC}}:
\begin{align*}
|\bar \Delta^{321}|&\lesssim h^{\frac{\alpha-1+\beta}{\alpha}}\int_h^{\tau_t^h-h}\int_{\R^d}\bar p_\alpha(\tau_s^h,z-x)(1+(\tau_s^h)^{-\frac\beta\alpha})(t-\tau_s^h)^{-\frac{2-\beta}{\alpha}}(s-\tau_s^h) dzds\|\varphi\|_{B_{\infty,\infty}^\beta}\\
&{\lesssim h^{\frac{2\alpha-1+\beta}{\alpha}}\int_h^{\tau_t^h-h}(1+(\tau_s^h)^{-\frac\beta\alpha})(t-\tau_s^h)^{-\frac{2-\beta}{\alpha}}ds\|\varphi\|_{B_{\infty,\infty}^\beta}.}
\end{align*}
If $2-\beta<\alpha $, then the above time singularity is integrable and:
\begin{align*}
|\bar \Delta^{321}|\lesssim h^{\frac{2\alpha-1+\beta}{\alpha}}t^{\frac{\alpha-2}{\alpha}}\|\varphi\|_{B_{\infty,\infty}^\beta}\le 
%\begin{cases}
h t^{\frac{{2\alpha-3+\beta}}{\alpha}} \|\varphi\|_{B_{\infty,\infty}^\beta}\le h t^{\frac{\alpha-1}\alpha}\|\varphi\|_{B_{\infty,\infty}^\beta},
%\end{cases}
\end{align*}
using the inequality $2-\beta<\alpha $ and the fact that $t$ is {smaller than $T$}.
If now $2-\beta\ge\alpha $,
\begin{align*}
|\bar \Delta^{321}|\lesssim& h^{\frac{2\alpha-1+\beta}{\alpha}}h^{1-\frac{2-\beta}{\alpha}}(1+|\ln(h)|\I_{2-\beta=\alpha})t^{-\frac{\beta}{\alpha}}\|\varphi\|_{B_{\infty,\infty}^\beta}\le h^{\frac{3(\alpha-1)+2\beta}{\alpha}}(1+|\ln(h)|\I_{2-\beta=\alpha})t^{-\frac{\beta}{\alpha}}\|\varphi\|_{B_{\infty,\infty}^\beta}\\
&%h^{\frac{\alpha-1+2\beta}{\alpha}}t^{-\frac{\beta}{\alpha}}\|\varphi\|_{B_{\infty,\infty}^\beta}
\lesssim h^{\frac{\alpha\wedge (\alpha-1+2\beta)}{\alpha}}t^{-\frac{{\beta\wedge \frac12}}{\alpha}}\|\varphi\|_{B_{\infty,\infty}^\beta},
%\begin{cases}
%\end{cases}
\end{align*}
{observing that this contribution is even smaller than the one which appeared for $\bar \Delta^{311}$}.

In all cases, the contribution is comparable to the one in \eqref{BAR_Delta_322} which gives that, this bound is precisely valid for $\Delta^{32} $ in the current setting. The contribution $\Delta^4$ could be analyzed in a similar way. 

Let us turn to $\Delta^5 $ {(last time steps contribution)} and more specifically the {term} $\Delta^{52} $ which involves the density of the scheme. Write:
\begin{align*}
\Delta^{52}:=&\frac{1}h\int_{\tau_t^h-h}^t\int_{\tau_t^h}^{\tau_t^h+h} \Big|\int_{\R^d } \Gamma^h (0,x,\tau_s^h,z)b(r,z)\cdot\nabla P_{ t-\tau_s^h}^\alpha\varphi (z+b(r,z)(s-\tau_s^h)) dz  \Big|dr d s\\
\le &\frac{1}h\int_{\tau_t^h-h}^t\int_{\tau_t^h}^{\tau_t^h+h} \Big|\int_{\R^d } (\Gamma^h-\Gamma) (0,x,\tau_s^h,z)b(r,z)\cdot\nabla P_{ t-\tau_s^h}^\alpha\varphi (z+b(r,z)(s-\tau_s^h)) dz  \Big|dr d s\\
&+\frac{1}h\int_{\tau_t^h-h}^t\int_{\tau_t^h}^{\tau_t^h+h} \Big|\int_{\R^d } \Gamma (0,x,\tau_s^h,z)b(r,z)\cdot\nabla P_{ t-\tau_s^h}^\alpha\varphi (z+b(r,z)(s-\tau_s^h)) dz  \Big|dr d s=:\bar \Delta^{521}+\bar \Delta^{522}.
\end{align*}
Again,  $\bar \Delta^{522} $ can again be handled as the former $\Delta^{52} $ using \eqref{BESOV_NORM_1} for the diffusion. This would lead to:
$$|\bar \Delta^{522}|\lesssim h^{\frac{\alpha\wedge (\alpha-1+2\beta) }{\alpha}}(1+|\ln (h)|\I_{{\beta= \frac 12}})(1+t^{-\frac{{\beta\wedge \frac 12}}\alpha})\|\varphi\|_{B_{\infty,\infty}^\beta}. $$
Now, \eqref{BD_THM_FITO_MENO_SPA} and \eqref{GAIN_TEST_FUNC} yield% \footnote{\textcolor{red}{De S. à B.: attention un terme avait été oublié précédemment. Cela ne change rien heureusement!}}
:
\begin{align*}
\bar \Delta^{521}\lesssim& h^{\frac{\alpha-1+\beta}\alpha}\int_{\tau_t^h-h}^t \int_{\R^d } {(1+(\tau_s^h)^{-\frac \beta\alpha})}\bar p_\alpha (\tau_s^h,z-x)(t-\tau_s^h)^{-\frac{1-\beta}\alpha} \|\varphi\|_{B_{\infty,\infty}^{\beta}} dz   d s\lesssim h^{2\frac{\alpha-1+\beta}\alpha}{t^{-\frac \beta\alpha}}\|\varphi\|_{B_{\infty,\infty}^{\beta}}\\
\lesssim& h^{\frac{\alpha\wedge (\alpha-1+2\beta) }{\alpha}}{t^{-\frac{\beta\wedge \frac 12}{\alpha}}}\|\varphi\|_{B_{\infty,\infty}^\beta},
\end{align*}
{repeating the previous discussion for the term $\bar \Delta^{311} $}. 
Once again, the convergence rate remains the same.

This concludes the presentation of the modification needed in the current setting {$ \alpha-1\le \beta$ to prove} the convergence rate of Theorem \ref{THM_HOLDER}. 

\subsection {Proof of the technical results for the Hölder case}\label{SEC_TEC}
We gather in this section the proof of various technical results needed for the analysis.
\subsubsection{Proof of the identity \eqref{SING_STABLE_HK}}
Let us start with the non thermic part from the definition in \eqref{HEAT_CAR}. Namely, we have to handle for $\phi\in C_0^\infty(\R^d,\R^d) $ s.t. $\phi(0)\neq 0 $ and setting $\psi=\mathcal F^{-1}(\phi)\in \mathcal S(\R^d) $, {for $u>0$}:
\begin{align}
% \|\mathcal F^{-1}\big(\phi\mathcal F(\partial_u^\theta \partial^{\mathbf a}p_{\alpha}({u},\cdot))\big)\|_{\textcolor{black}{L^\ell}}=&
                                                                                                                                       \|\mathcal F^{-1}(\phi)\star\partial_u^\theta \partial^{\mathbf a}p_{\alpha}({u},\cdot)\|_{L^\ell}&=\|\psi\star(-\Delta)^{\frac{ \alpha \theta}{2}} \partial^{\mathbf a}p_{\alpha}({u},\cdot)\|_{L^\ell}\notag
=\| \partial^{\mathbf a}(-\Delta)^{\frac {\alpha \theta}2} \psi\star p_{\alpha}({u},\cdot)\|_{L^\ell}\\&\le \| \partial^{\mathbf a}(-\Delta)^{\frac {\alpha \theta}2} \psi\|_{L^\ell }\|p_{\alpha}({u},\cdot)\|_{L^1} \le C_\psi.\label{CTR_NON_THERM}
\end{align}
Let us now turn to the thermic part. Namely, assuming $m<\infty$, for $\gamma<0 $ s.t. $\theta+\frac{|\mathbf a|}\alpha+\frac \gamma\alpha+\frac d\alpha(1-\frac 1\ell)\ge  0 $ and $u\in (0,T]$ we write:
\begin{align}
\mathcal T_{\ell,m}^\gamma(\partial_u^\theta \partial^{\mathbf a}p_{\alpha}(u,\cdot))=&
\left(\int_0^T\,\frac{dv}{v}v^{(n-\gamma/\textcolor{black}{\alpha})m}\|\partial^n_v p_\alpha(v,\cdot)*\partial_u^\theta \partial^{\mathbf a}p_{\alpha}(u,\cdot)\|^m_{L^\ell}\right)^{\frac 1m}\notag\\
\le &\left(\int_0^{u}\,\frac{dv}{v}v^{(n-\gamma/\textcolor{black}{\alpha})m}\|\partial^n_v  p_\alpha(v,\cdot)*\partial_u^\theta \partial^{\mathbf a}p_{\alpha}(u,\cdot)\|^m_{L^\ell}\right)^{\frac 1m}\notag\\
&+\left(\int_{u}^T\,\frac{dv}{v}v^{(n-\gamma/\textcolor{black}{\alpha})m}\|\partial^n_v  p_\alpha(v,\cdot)*\partial_u^\theta \partial^{\mathbf a}p_{\alpha}(u,\cdot)\|^m_{L^\ell}\right)^{\frac 1m}\notag\\
=:&\mathcal T_{\ell,m,1}^\gamma(\partial_u^\theta \partial^{\mathbf a}p_{\alpha}(u,\cdot))+\mathcal T_{\ell,m,2}^\gamma(\partial_u^\theta \partial^{\mathbf a}p_{\alpha}(u,\cdot)).\label{THERMIC_PART_HK}
\end{align}
For the lower cut, we directly exploit that the time singularity in the variable $v $ is integrable for a negative exponent and the usual bound {(see \eqref{estigradplu} in Proposition \ref{PROP_INT_LP_HK})}
\begin{align}\label{BD_HK_LEBESGUE}
\|\partial_u^\theta\partial^{\mathbf a}p_\alpha(u,\cdot)\|_{L^\ell}\le \frac{C}{u^{\theta+\frac{|\mathbf a|}{\alpha}+\frac d\alpha(1-\frac 1\ell)}},
\end{align}
to derive
\begin{align*}
\mathcal T_{\ell,m,1}^\gamma(\partial_u^\theta \partial^{\mathbf a}p_{\alpha}(u,\cdot))
\le& \left(\int_0^{u}\,\frac{dv}{v}v^{(n-\gamma/\textcolor{black}{\alpha})m} \|\partial_v^n{ p}_\alpha(v,\cdot)\|_{L^1}^m\right)^{\frac 1m}
\frac{C}{u^{\theta+\frac{|\mathbf a|}{\alpha}+\frac d\alpha(1-\frac 1\ell)}}\\
\le & \frac{C}{u^{\theta+\frac{|\mathbf a|}{\alpha}+\frac \gamma\alpha+\frac d\alpha(1-\frac 1\ell)}}.
\end{align*}
For the upper cut, and for $1\le m<\infty $,  we have assumed that $\theta+\frac \gamma \alpha+\frac{|\mathbf a|}\alpha+\frac{d}\alpha(1-\frac 1\ell) \ge 0$ in order to keep a negative exponent in the thermic variable {$v$} which is smaller than $-1$, so that the lower cut dominates. 
{From the semigroup property of $p_\alpha $ and \eqref{estigradplu}, we write}:
\begin{align*}
\mathcal T_{\ell,m,2}^\gamma(\partial_u^\theta \partial^{\mathbf a}p_{\alpha}(u,\cdot))
  \le& \left(\int_{u}^T\,\frac{dv}{v}v^{(n-\gamma/\textcolor{black}{\alpha})m} {C}(u+v)^{-m(n+\theta+\frac{|\mathbf a|}\alpha+\frac d\alpha(1-\frac 1\ell))}\right)^{\frac 1m}\\
  \le & {C\left(\int_{u}^T\,v^{-1-m(\theta+\frac \gamma\alpha+\frac{|\mathbf a|}\alpha+\frac d\alpha(1-\frac 1\ell))}dv\right)^{\frac 1m}}\\
\le & \frac{C}{u^{\theta+\frac \gamma\alpha+\frac{|\mathbf a|}{\alpha}+\frac d\alpha(1-\frac 1\ell)}}(1+\I_{\theta+\frac \gamma \alpha+\frac{|\mathbf a|}\alpha+\frac{d}\alpha(1-\frac 1\ell)=0} |\ln(u)|).
\end{align*}

Let us now suppose that $m=\infty$. Write then with obvious notations:
\begin{align}
\mathcal T_{\ell,\infty}^\gamma(\partial_u^\theta \partial^{\mathbf a}p_{\alpha}(u,\cdot))=&\sup_{v\in [0,T]}
v^{(n-\gamma/\textcolor{black}{\alpha})}\|\partial^n_v  p_\alpha(v,\cdot)\star\partial_u^\theta \partial^{\mathbf a}p_{\alpha}(u,\cdot)\|_{L^\ell}\notag\\
\le &\sup_{v\in [0,u]}v^{(n-\gamma/\textcolor{black}{\alpha})}\|\partial^n_v  p_\alpha(v,\cdot)\star\partial_u^\theta \partial^{\mathbf a}p_{\alpha}(u,\cdot)\|_{L^\ell}\notag\\
&+\sup_{v\in [u,T]}v^{(n-\gamma/\textcolor{black}{\alpha})}\|\partial^n_v  p_\alpha(v,\cdot)\star\partial_u^\theta \partial^{\mathbf a}p_{\alpha}(u,\cdot)\|_{L^\ell}\notag\\
=:&\mathcal T_{\ell,\infty,1}^\gamma(\partial_u^\theta \partial^{\mathbf a}p_{\alpha}(u,\cdot))+\mathcal T_{\ell,\infty,2}^\gamma(\partial_u^\theta \partial^{\mathbf a}p_{\alpha}(u,\cdot)).\label{THERMIC_PART_HK_infty}
\end{align}
From \eqref{BD_HK_LEBESGUE} one readily gets {(recall that $\gamma<0 $)}:
\begin{align*}
\mathcal T_{\ell,\infty,1}^\gamma(\partial_u^\theta \partial^{\mathbf a} p_{\alpha}(u,\cdot))
  &\le \sup_{v\in [0,u]} v^{n-\frac \gamma\alpha}\|\partial_v^n p_\alpha(v,\cdot)\|_{L^{\tr{1}}}u^{-(\theta+\frac{|\mathbf a|}{\alpha})+\frac d\alpha(1-\frac 1\ell)}\\
&\lesssim u^{-(\theta+\frac \gamma\alpha+\frac{|\mathbf a|}{\alpha})+\frac d\alpha(1-\frac 1\ell)}.
\end{align*}
Similarly,
\begin{align*}
\mathcal T_{\ell,\infty,2}^\gamma(\partial_u^\theta \partial^{\mathbf a}p_{\alpha}(u,\cdot))&\le \sup_{v\in [u,T]} v^{n-\frac \gamma\alpha}(u+v)^{-(n+\theta+\frac{|\mathbf a|}{\alpha})+\frac d\alpha(1-\frac 1\ell)}\\
&\lesssim \sup_{v\in [u,T]}v^{-[(\theta+\frac \gamma\alpha+\frac{|\mathbf a|}{\alpha})+\frac d\alpha(1-\frac 1\ell)]}\\
&\lesssim u^{-[(\theta+\frac \gamma\alpha+\frac{|\mathbf a|}{\alpha})+\frac d\alpha(1-\frac 1\ell)]}.
\end{align*}
{where we used  $(\theta+\frac \gamma\alpha+\frac{|\mathbf a|}{\alpha})+\frac d\alpha(1-\frac 1\ell)\ge 0 $ for the last inequality}.
From the above computations it is clear that, whenever $(\theta+\frac \gamma\alpha+\frac{|\mathbf a|}{\alpha})+\frac d\alpha(1-\frac 1\ell)< 0 $, the thermic part of the Besov norm can be bounded by a constant.  The control \eqref{SING_STABLE_HK} is proved. \quad $\square $

{The above proof emphasizes that, when the singularity induced by the differentiation and the integrability indexes are large enough, then the negative regularity will \textit{tame} those singularities, otherwise the corresponding norm is bounded by a constant.}
\subsubsection{Proof of \eqref{holder-space-gamma_SPEC}, \eqref{BESOV_NORM_1} and \eqref{BESOV_NORM_2}}
We start with \eqref{holder-space-gamma_SPEC} (H\"older modulus for the density of the scheme).
From the Duhamel representation \eqref{duhamel-schemeb}, for all $0{<t_k<t}\le T $, $x,y,w\in \R^d $, $|y-w|\le {(t-t_k)^{\frac 1\alpha}}$ we get:
\begin{align*}
			&|\Gamma^h(t_k,x,t,y)-\Gamma^h(t_k,x,t,w)|\\
			\le& |p_\alpha(t-t_k,y-x)-p_\alpha(t-t_k,w-x)|\\
			&+\int_{t_k}^{ t}\Big|\E_{t_k,x}\left[b(U_{\lfloor s/h\rfloor},X^h_{\tau_s^h})\cdot\Big(\nabla_y  p_\alpha(t-s,y-X^h_s)-\nabla_w p_\alpha(t-s,w-X_s^h)\Big)\right]\Big|ds\\
			\lesssim & \Big(\frac{|y-w|}{(t-t_k)^{\frac 1\alpha}}\Big)^{\beta}{\bar p_\alpha}(t-t_k,y-x) \\
			& +\int_{t_k}^t\frac{1}{h}\int_{\tau_s^h}^{\tau_s^h+h}\Big|\int_{\R^d} \Gamma^h(t_k,x,\tau_s^h,z)\\
			&\times b(r,z)\cdot \Big(\nabla p_\alpha(t-\tau_s^h,y-(z+b(r,z)(s-\tau_s^h)))-\nabla_wp_\alpha(t-\tau_s^h,w-(z+b(r,z)(s-\tau_s^h)))\Big)    dz\Big|dr ds,
\end{align*}
using {\eqref{drift-smoothing-iso-noise-diag}, the inequality $\frac{|y-w|}{(t-t_k)^{\frac 1\alpha}}\le \Big(\frac{|y-w|}{(t-t_k)^{\frac 1\alpha}}\Big)^{\beta}\le 1$} and \eqref{CONDITIONING_FOR_EULER_SCHEME}. From % Proposition \ref{PROP_DENS_HOLDER}, eq. 
\eqref{ARONS_SCHEME}, {\eqref{drift-smoothing-iso-noise} and the last inequality,} we now get:
\begin{align*}
|\Gamma^h(t_k,x,t,y)-\Gamma^h(t_k,x,t,w)|
			&\lesssim  \Big(\frac{|y-w|}{(t-t_k)^{\frac 1\alpha}}\Big)^{\beta}\bar p_\alpha(t-t_k,y-x) \\	
			&+	\int_{t_k}^t\int_{\R^d}\bar p_\alpha(\tau_s^h,z-x)\frac{|w-y|^\beta}{(t-\tau_s^h)^{\frac{1+\beta}\alpha}}\Big( \bar p_\alpha(t-\tau_s^h,y-z)+p_\alpha(t-\tau_s^h,w-z)\Big)dz,
\end{align*}
observing that since $t-\tau_s^h\ge s-\tau_s^h $, the contribution of the local drift transition $b(r,z)(s-\tau_s^h) $ is negligible. We eventually derive, since we have assumed that $\frac{1+\beta}{\alpha}<1\iff \beta<\alpha-1 $ that:
\begin{align*}
|\Gamma^h(t_k,x,t,y)-\Gamma^h(t_k,x,t,w)|
			& \lesssim \Big(\frac{|y-w|}{(t-t_k)^{\frac 1\alpha}}\Big)^{\beta}(\bar p_\alpha(t-t_k,y-x)+\bar p_\alpha(t-t_k,w-x))\\&\lesssim \Big(\frac{|y-w|}{(t-t_k)^{\frac 1\alpha}}\Big)^{\beta}\bar p_\alpha(t-t_k,w-x),
\end{align*}
using that the diagonal regime holds for the last inequality {since $|y-w|\le (t-t_k)^{\frac 1\alpha}$}. This is precisely the claim.

Let us now turn to \eqref{BESOV_NORM_2}. Write{, setting $\psi=\mathcal F^{-1}(\phi)$ like in the above derivation of \eqref{CTR_NON_THERM},}:
\begin{align*}
\|\Gamma^h(0,\tau_s^h,x,\cdot) \Psi \|_{B_{1,\infty}^\beta}
&=\|\psi\star(\Gamma^h(0,\tau_s^h,x,\cdot) \Psi)\|_{L^1}+\sup_{v\in [0,T]}v^{1-\frac\beta \alpha}\|\partial_v p_\alpha \star\big(\Gamma^h(0,\tau_s^h,x,\cdot) \Psi \big)\|_{L^1}\\
&\lesssim \|\Psi\|_{L^\infty}+\mathcal T_{1,\infty}^\beta\Big(\Gamma^h(0,\tau_s^h,x,\cdot) \Big).
\end{align*}
{Focus now on thermic part. We have:
\begin{align*}
{\mathcal T_{1,\infty}^\beta\Big(\Gamma^h(0,\tau_s^h,x,\cdot) \Big)}=&\sup_{v\in [0,T]}v^{1-\frac\beta \alpha}\int_{\R^d} \Big |\int_{\R^d} \partial_v p_\alpha(v,z-y)\big(\Gamma^h(0,\tau_s^h,x,y) \Psi(y)-\Gamma^h(0,\tau_s^h,x,z) \Psi(z)\big)dy\Big| dz\\
\lesssim& \sup_{v\in [0,T]}v^{1-\frac\beta \alpha}\int_{\R^d} \int_{\R^d} v^{-1}\bar p_\alpha(v,z-y)|z-y|^\beta\|\Psi\|_{B_{\infty,\infty}^\beta}\Big(\bar p_\alpha(\tau_s^h,y-x)\\
&\quad\quad+(\tau_s^h)^{-\frac\beta\alpha}(\bar p_\alpha(\tau_s^h,y-x)+\bar p_\alpha(\tau_s^h,z-x)) \Big)dy dz\\
\lesssim& \sup_{v\in [0,T]}v^{1-\frac\beta \alpha} v^{-1+\frac\beta \alpha}\|\Psi\|_{B_{\infty,\infty}^\beta}(1+(\tau_s^h)^{-\frac \beta\alpha})\lesssim \|\Psi\|_{B_{\infty,\infty}^\beta}(1+s^{-\frac \beta\alpha}),
\end{align*}
where we used a usual cancellation argument, {\eqref{drift-smoothing-iso-noise}}, \eqref{ARONS_SCHEME} and \eqref{holder-space-gamma_SPEC}. This gives the statement for the Euler scheme. For the diffusion, the previous computations can be reproduced, without the additional condition $\alpha-1>\beta $ since the $\beta $-modulus of continuity in the forward variable follows from \eqref{holder-space-gamma}. 
The claim \eqref{BESOV_NORM_1} is proved.
\subsubsection{Proof of Proposition \ref{PROP_FWD_TIME_REG} {(time regularity in the forward variable for the diffusion)}}
\label{SEC_PREUVE_TEC}
We again start, for $s\ge 2h$, from the Duhamel representation \eqref{duhamel-Diff}:
\begin{align}
&\Gamma(0,x,\tau_s^h,z)-\Gamma(0,x,s,z)\notag\\
=&p_\alpha(\tau_s^h,z-x)-p_\alpha(s,z-x)+\int_{0}^{\tau_s^h-h}\Gamma(0,x,r,w) b(r,w)\cdot \left[\nabla p_{\alpha}(s-r,z-w)-\nabla p_{\alpha}(\tau_s^h-r,z-w)\right] dw  dr\notag\\
&+\Big(\int_{\tau_s^h-h}^s \Gamma(0,x,r,w) b(r,w)\cdot \nabla p_{\alpha}(s-r,z-w)dw  dr\notag-\int_{\tau_s^h-h}^{\tau_s^h} \Gamma(0,x,r,w) b(r,w)\cdot \nabla p_{\alpha}(\tau_s^h-r,z-w)dw  dr\Big)\notag\\
=&(\Delta^{1}+\Delta^{2}+\Delta^{3})(0,x,s,\tau_s^h,z).\label{DECOUP_X_PROP}
\end{align}
Write first for the \textit{main} term of the expansion:
\begin{align*}
\Delta^{1}(0,x,s,\tau_s^h,z)=\int_0^1 \partial_{r_\lambda} p_\alpha(r_\lambda,z-x)|_{r_\lambda=\tau_s^h+\lambda (s-\tau_s^h)} d\lambda (s-\tau_s^h),
\end{align*}
from which one derives using \eqref{SING_STABLE_HK}:%\footnote{\textcolor{red}{Pay attention that may be everything can be cast in the setting of non homogeneous Besov spaces, because of short time and handling of signularities}}:
\begin{align}\label{CTR_DELTA_1_PROP_6}
\|\Delta^{1}(0,x,s,\tau_s^h,\cdot)\|_{%\dot 
B_{1,1}^{-\beta}}\lesssim & \int_{0}^1 \|\  \partial_{r_\lambda} p_\alpha(r_\lambda,\cdot-x)|_{r_\lambda=\tau_s^h+\lambda (s-\tau_s^h)}\ \|_{%\dot 
B_{1,1}^{-\beta}}d\lambda (s-\tau_s^h)\lesssim \frac{h}{(\tau_s^h)^{1-\frac \beta\alpha}}\notag\\
% \lesssim &\Big(\frac{h}{(\tau_s^h)^{1-\frac \beta\alpha}}\I_{\beta\ge \frac 12}+\frac{h^{\frac{\alpha-1+2\beta}{\alpha}}}{(\tau_s^h)^{1-\big(\frac{\beta}{\alpha}+\frac 1\alpha-\frac{2\beta}\alpha\big) }}\I_{\beta< \frac 12}\Big)\notag\\
\lesssim &\Big(\frac{h}{(\tau_s^h)^{1-\frac \beta\alpha}}\I_{\beta\ge \frac 12}+\frac{h^{\frac{\alpha-1+2\beta}{\alpha}}}{(\tau_s^h)^{\frac{\alpha-1+\beta}{\alpha}}}\I_{\beta< \frac 12}\Big)\notag\\
% \lesssim &h^{\frac{\alpha\wedge (\alpha-1+{2}\beta)}{\alpha}}\Big(\frac{\I_{\beta\ge \frac 12}}{(\tau_s^h)^{\frac{\alpha- \beta}\alpha}}+\frac{\I_{\beta< \frac 12}}{(\tau_s^h)^{\frac{\alpha-(1-\beta)}{\alpha}}}\Big)\notag\\
\lesssim &h^{\frac{\alpha\wedge (\alpha-1+{2}\beta)}{\alpha}}(\tau_s^h)^{-\frac{\alpha- (\beta \vee (1-\beta))}\alpha}.
\end{align}
Let us now consider ${\Delta^3}(0,x,s,\tau_s^h,z) $ for which we derive,
%\footnote{\textcolor{red}{On va essayer de jouer sur la troisieme patte du Besov pour tuer les singularites du noyau de la chaleur. Toutes? Toutes!}}:
\begin{align}
\|\Delta^{3}(0,x,s,\tau_s^h,\cdot)\|_{B_{1,1}^{-\beta}}\lesssim &\int_{\tau_s^h-h}^s  \|\Gamma(0,x,r,\cdot) b(r,\cdot)\|_{B_{1,\infty}^{{\beta\wedge \frac 12}}} \| \nabla p_{\alpha}(s-r,z-\cdot)\|_{B_{1,1}^{{-\beta-(\beta\wedge \frac 12)}}} dr\notag\\
&+\int_{\tau_s^h-h}^{\tau_s^h}  \|\Gamma(0,x,r,\cdot) b(r,\cdot)\|_{B_{1,\infty}^{{\beta\wedge \frac 12}}} \| \nabla p_{\alpha}(\tau_s^h-r,z-\cdot)\|_{B_{1,1}^{-{\beta-(\beta\wedge \frac 12)}}} dr\notag\\
\lesssim& \int_{\tau_s^h-h}^s (1+r^{-\frac{{\beta\wedge \frac 12}}\alpha}) (s-r)^{-{\frac{1-(2\beta\wedge 1)}\alpha}}(1+|\ln(s-r)|\I_{{\beta =\frac 12}})dr\notag\\
& +\int_{\tau_s^h-h}^{\tau_s^h} (1+r^{-\frac {{\beta\wedge \frac 12}}\alpha}) (\tau_s^h-r)^{-{\frac{1-(2\beta\wedge 1)}\alpha}}(1+|\ln(\tau_s^h-r)|\I_{{\beta= \frac 12}})dr\notag\\
\lesssim & (1+(\tau_s^h)^{-\frac {{\beta \wedge \frac 12}}\alpha})h^{ \frac{{\alpha\wedge(\alpha-1+2\beta)}}{\alpha}}(1+|\ln(h)|\I_{{\beta = \frac 12}}),
\label{CTR_DELTA_3_PROP_6}
\end{align}
using {first}  \eqref{YOUNG} and {then} {\eqref{SING_STABLE_HK}}  and \eqref{BESOV_NORM_1} for the last {but one} inequality. The idea is again to exploit at most the a priori smoothness of the drift coefficient and the associated heat kernel in order to decrease the time singularity of the proxy stable heat kernel. We get similarly to the computations performed {for the term $\Delta^{31}$ in \eqref{BD_FOR_DELTA_31_A_CITER} above}:
\begin{align*}
\|\Delta^{2}(0,x,s,\tau_s^h,\cdot)\|_{B_{1,1}^{-\beta}}\lesssim &\int_{0}^{\tau_s^h-h}  \|\Gamma(0,x,r,\cdot) b(r,\cdot)\|_{B_{1,\infty}^\beta} \| \nabla p_{\alpha}(s-r,z-\cdot)-\nabla p_{\alpha}(\tau_s^h-r,z-\cdot)\|_{B_{1,1}^{-2\beta}} dr\\
\lesssim& \int_{0}^{\tau_s^h-h} (1+r^{-\frac \beta\alpha}) (\tau_s^h-r)^{-(1+\frac{1}\alpha-\frac{2\beta}{\alpha})}(s-\tau_s^h) dr.
\end{align*}
For this term we perform the same previous discussion {as} for \eqref{CTR_Delta_31}. This leads to:
\begin{align}
\|\Delta^{2}(0,x,s,\tau_s^h,\cdot)\|_{B_{1,1}^{-\beta}}\lesssim &h^{\frac{\alpha\wedge (\alpha-1+ 2\beta)}{\alpha}}(1+s^{-\frac{{\beta\wedge \frac 12}}{\alpha}})(1+|\ln(h)|\I_{\beta\tr{=}\frac 12}).\label{CTR_DELTA_2_PROP_6}
\end{align}
From \eqref{CTR_DELTA_2_PROP_6}, \eqref{CTR_DELTA_3_PROP_6}, \eqref{CTR_DELTA_1_PROP_6}, we eventually derive:
\begin{align*}
  \|\Gamma(0,x,s,\cdot)-\Gamma(0,x,\tau_s^h,\cdot)\|_{B_{1,1}^{-\beta}}
  % lesssim &h^{\frac{\alpha\wedge (\alpha-1+ 2\beta)}{\alpha}}(1%+s^{-\frac{\alpha-(\beta\vee (1-\beta))}{\alpha}}
% +s^{-\frac{\beta\wedge (1-\beta)}{\alpha}}+{s^{-\frac{\beta\wedge \frac 12}\alpha}})(1+|\ln(h)|\I_{{\beta= \frac 12}})\\
%  %\label{FINAL}
\lesssim &h^{\frac{\alpha\wedge (\alpha-1+ 2\beta)}{\alpha}}(1+s^{-\frac{\alpha-(\beta\vee (1-\beta))}{\alpha}}+{s^{-\frac{\beta\wedge \frac 12}{\alpha}}})(1+|\ln(h)|\I_{{\beta= \frac 12}})\\\lesssim &h^{\frac{\alpha\wedge (\alpha-1+ 2\beta)}{\alpha}}(1+s^{-\frac{\alpha-(\beta\vee (1-\beta))}{\alpha}}+{s^{-\frac 1{2\alpha}}\I_{{\beta> \alpha-\frac 12}}})(1+|\ln(h)|\I_{{\beta= \frac 12}}),
\end{align*}
{
%where for the last inequality we use that since $\alpha>1 $, when $\beta>1/2 $, one can always take $\beta $ sufficiently close to $1/2 $  s.t. $\frac 12<\alpha-(\beta \vee (1-\beta)) $
since $\alpha-(\beta\vee (1-\beta))=\alpha-1+\beta>\beta$ when $\beta\le 1/2$ while, when $\beta>1/2$, $\alpha-(\beta\vee (1-\beta))<1/2\iff \beta>\alpha-1/2$.}
%\footnote{\textcolor{red}{De S. à B.: c'est ici que l'on utilise la restriction $\beta \in (0,1/2^+)$ pour ne pas changer l'estimée précédente, i.e. dire que les singularités temporelles dominantes restent inchangées. Sinon il faut garder les deux conditions}}.

\section{Besov drift with negative regularity index in $(-1/2,0) $ and stable non-Brownian noise}

We are in this section specifically interested in the case where $b\in L^\vartheta([0,T],B_{p,q}^\beta) $ where:
\begin{equation}
\label{serrin}
	 \beta \in \left( \frac{1-\alpha+\frac{d}{p}+\frac{\alpha}{\vartheta}}{2} ,0\right),
\end{equation}
{i.e. the regularity index is negative and the drift is therefore a distribution. Remark in particular that, for given integrability parameters $p,\vartheta $, this imposes a constraint on $\alpha $. Namely,
$$\frac{1-\alpha+\frac{d}{p}+\frac{\alpha}{\vartheta}}{2}<0 \iff \alpha \in \left( \frac{1+\frac{d}{p}}{1-\frac{1}{\vartheta}},2 \right). $$}
{We exclude here the Brownian case $\alpha=2 $ mainly because the  heat kernel estimates required for the analysis are not, yet, available. They have been established for $p=\vartheta=+\infty $ in \cite{Me:Pa:26} and the approach therein should extend to arbitrary integrability parameters}.

	To introduce the scheme associated with the formal previous SDE \eqref{SDE}, one first needs to recall that the precise meaning to be given to the SDE, following \cite{chau:meno:22} in the pure-jump setting, inspired by \cite{dela:diel:16} in the Brownian setting,  is:
	\begin{equation}\label{DYN_DIFF}
	X_t=x+\int_0^t {\mathfrak b}(s,X_s,ds)+Z_t,
	\end{equation}
	where for all $(s,z)\in [0,T]\times \R^d,h>0$,
	\begin{equation}
\label{DEF_DRIFT}
\mathfrak{b}(s,z,h) := \int_s^{s+h} \int_{\R^d} b(u,y)p_\alpha(u-s,z-y) dy du=\int_s^{s+h} P_{u-s}^\alpha b(u,z)   du,
	\end{equation}
$p_\alpha(v,\cdot)$ denoting the density of the $\alpha $-stable driving noise $(Z_v)_{v\geq0}$ at time $v$ and $P^\alpha $ the associated semi-group.  The integral in \eqref{DYN_DIFF} is a nonlinear Young integral obtained by passing to the limit  in a suitable procedure aimed at reconstructing the drift (see again \cite{chau:meno:22}). The resulting drift in \eqref{DYN_DIFF} is, \textit{per se}, a Dirichlet process (as it had already been indicated in the literature, see e.g. \cite{athr:butk:mytn:20} and references therein). % Importantly, the dynamics in \eqref{DYN_DIFF} also naturally provides a corresponding approximation scheme to be analyzed.
Note that, in order to give a precise meaning to the integral appearing in \eqref{DYN_DIFF}, {we need to strengthen the condition on $\beta $ to: 
\begin{equation*}
\label{COND_FOR_DIFF}
	%\alpha \in \left( \frac{1+\frac{d}{p}}{1-\frac{1}{\vartheta}},2 \right)\qquad 
	\beta \in \left( \frac{1-\alpha+\frac{2d}{p}+\frac{2\alpha}{\vartheta}}{2} ,0\right).
\end{equation*}
Anyhow the condition \eqref{serrin} above}  is actually enough to ensure {well-posedness of} a corresponding generalized martingale problem (see \cite{chau:meno:22}).
 Interestingly, {the} more stringent condition on $\beta $ does not appear {elsewhere} in the present work since we only consider the time marginals of the process. {It is not needed to define the} Euler scheme $X^h $, starting from $X_0^h=x$ and evolving on the time grid as
\begin{equation}\label{euler-scheme-besov_GRID}
		 X_{t_{i+1}}^h = X_{t_i}^h+  \mathfrak{b}(t_i,{X}_{t_i}^h,h)+Z_{t_{i+1}}-Z_{t_i}.
	\end{equation}
	We have precisely used the quantity $\mathfrak{b}(t_i,X_{t_i}^h,h)$ defined in \eqref{DEF_DRIFT} {with a time argument corresponding to the chosen time step} as an approximation of the nonlinear Young integral $\int_{t_i}^{t_{i+1}} {\mathfrak b}(s,X^h_s,ds)$, 
%plugged the expression \eqref{DEF_DRIFT}, 
which served to define the limit dynamics \eqref{DYN_DIFF} for the SDE.\\

	Set now for $(s,z)\in {\{[0,T]\backslash \{kh:k\in \llbracket 0,n\rrbracket\}\}}\times \R^d$,
	\begin{equation}\label{DEF_BH_SCHEME}
{\mathfrak b}_h(s,z):=P_{s-\tau_s^h}^\alpha b(s,z). 
	\end{equation}
Observe from that definition and using \eqref{DEF_DRIFT} that, on any time step, the drift also writes as
\begin{equation}
\label{EXPR_DRIFT}
 \mathfrak{b}(t_i,{X}_{t_i}^h,h)=\int_{t_i}^{t_{i+1}} \mathfrak b_h(u,X_{t_i}^h)du=\E[\mathfrak{b}_h(U_{i},{X}_{t_i}^h)|{X}_{t_i}^h]h,
 \end{equation}
where  the	$(U_k)_{k\in {\llbracket 0,n-1\rrbracket}} $ are independent random variables, %defined on some probability space $(\tilde \Omega, \tilde {\mathcal A}, \tilde P ) $ 
independent as well from the driving noise, such that % $U_k\overset{({\rm law})}={\mathcal U}([{t_k,t_{k+1}}]) $, i.e. 
$U_k $ is uniform on the time interval $[t_k,t_{k+1}]$.\\

Concretely, the above time and spatial expectations  \eqref{EXPR_DRIFT} and \eqref{DEF_BH_SCHEME}, need to be approximated in order to  implement this discretization. Such computations are anyhow case-dependent. Let us recall that a usual trick to spare one of these approximations  consists in considering some randomization in  time. This would lead to consider $\tilde {\mathfrak{b}}(t_i,{X}_{t_i}^h,h):=\mathfrak{b}_h(U_{i},{X}_{{t_i}}^h) {h}$. This \textit{trick} was  used successfully for Lebesgue drifts (see \cite{BJ20,jour:meno:24,fito:jour:meno:25} and {Section \ref{SEC_LP}}) and also allowed in the spatial Hölder setting to achieve the somehow expected convergence rates without any requirements on the time regularity (see \cite{fito:meno:25} as well as the previous Section \ref{SEC_HOLDER}). Anyhow, in the current singular setting it seems difficult to benefit from such an effect in the sense that without any additional time integration we do not have controls on the approximate drift norm. 
%This can be  seen e.g. in \eqref{CTR_PONCTUEL_BH} below or in the proof of the sensitivity analysis involving the local transitions (see proof of control \eqref{besov-estimate-gammah-stable-PERTURB_DRIFT}).\\

The representation \eqref{EXPR_DRIFT} naturally suggests to extend the dynamics of the scheme in continuous time as follows 
	\begin{equation}\label{euler-scheme-besov}
		 X_{t}^h = X_{\tau_t^h}^h+  \mathfrak{b}(\tau_t^h,{X}_{\tau_t^h}^h,t-\tau_t^h)+Z_{t}-Z_{\tau_t^h}=X_{\tau_t^h}^h+ \int_{\tau_t^h}^t \mathfrak{b}_h(s,{X}_{\tau_t^h}^h)ds+Z_t-Z_{\tau_t^h},
	\end{equation}
which gives an  extension in integral form which is  similar to the dynamics of Euler schemes involving non-singular drifts, i.e. it is an Itô type process and the approximate drift appears through a usual time integral. 

Under Condition \eqref{serrin}, both the diffusion and the scheme enjoy in positive time densities w.r.t. the Lebesgue measure, that will be as above denoted by $\Gamma $ and $\Gamma^h $ respectively. We gather below some of their useful properties for the analysis of the error we are interested in.

\begin{PROP}[Heat kernel estimates for the densities]\label{prop-HK}\label{THE_PROP} Set $\gamma=\alpha-1+2\beta-\frac{d}{p}-\frac{\alpha}{\vartheta} $ and assume \eqref{serrin}, which implies $\gamma>0$. Then the following assertions hold.
	\begin{itemize}
		\item Heat kernel and \textcolor{black}{forward sensitivity bounds} for the density of the Euler scheme: \textcolor{black}{there exists $C{=C(d,b,\alpha)}$ and for all $\rho \in (-\beta,\gamma-\beta)$ there exists $C_\rho{=C_\rho(d,b,\alpha)}$} such that for all $(x,y,y')\in (\R^d)^3$, $t>0$, % s.t. $|y-y'|\le t^{\frac 1\alpha} $,
		\begin{align}
		\Gamma^h (0,x,t,y) &\leq C \bar p_{\alpha}(t,y-x),\label{aronson-gammah}\\
		|\Gamma^h (0,x,t,y')-\Gamma^h (0,x,t,y)|&\le {\textcolor{black}{C_\rho}\frac{|y-y'|^\rho}{t^{\frac \rho\alpha}} (\bar p_\alpha(t,y-x)+\bar p_\alpha(t,y'-x))}.\label{holder-forward-gammah}
		\end{align}
%		Consequently, in terms of Besov spaces 
%		\begin{equation}\label{ineq-density-scheme}
%			% \Gamma^h (0,x,t,y) \leq C p_\alpha(t,y-x).
%						\left\Vert \frac{\Gamma^h (0,x,t,\cdot)}{\bar{p}_{\alpha}(t,\cdot-x)} \right\Vert_{B_{\infty,\infty}^\rho} \leq C(1+t^{-\frac \rho\alpha}).
%		\end{equation}
		\item Heat kernel and forward sensitivity bounds for the density of the SDE: \textcolor{black}{there exists $C{=C(d,b,\alpha)}$ and for all $\rho \in (-\beta,\gamma-\beta)$ there exists $C_\rho{=C(d,b,\alpha)}$} such that for all $(x,y,y')\in (\R^d)^3$, $t\in (0,T]$,  s.t. ,
				\begin{align}
					\Gamma (0,x,t,y) &\leq C \bar p_\alpha(t,y-x),\label{aronson-gamma}\\
				|\Gamma (0,x,t,y')-\Gamma (0,x,t,y)|&\le \textcolor{black}{C_\rho}\frac{|y-y'|^\rho}{t^{\frac \rho\alpha}}\left(\bar p_\alpha(t,y-x)+\bar p_\alpha(t,y'-x)\right).\label{holder-forward-gamma}
			\end{align}
%Consequently, in terms of Besov spaces,
%		\begin{equation}\label{ineq-density-diff}
%			\left\Vert \frac{\Gamma (0,x,t,\cdot)}{\bar{p}_{\alpha}(t,\cdot-x)} \right\Vert_{B_{\infty,\infty}^\rho} \leq C(1+t^{-\frac \rho\alpha}).
%		\end{equation}
%Moreover, it holds that for all $\eps>0,t'\in (t,T]$ such that  $|t-t'|\le t/2$, 
%		\begin{equation}\label{holder-time-gamma}
%			\left\Vert \frac{\Gamma (0,x,t,\cdot)-\Gamma (0,x,t',\cdot)}{\bar{p}_{\alpha}(t',\cdot-x)} \right\Vert_{B_{\infty,\infty}^\rho} \leq C \frac{(t'-t)^\frac{\gamma-\eps}{\alpha}}{t^\frac{\gamma-\eps+\rho}{\alpha}} .
%		\end{equation}
		
In the same spirit, it holds that \textcolor{black}{for all $\eps>0$ meant to be small}%\footnote{\textcolor{black}{From S. : not true as is, it must hold that $\rho+\gamma-\varepsilon<-\beta+\gamma $}}
, $\rho\in(-\beta,-\beta+\varepsilon/2)$,  $t'\in (t,T]$ such that  $|t-t'|\le t/2$, 
		\begin{equation}\label{holder-time-gamma_PP_QP_NO_NORM}
			\left\Vert \Gamma (0,x,t,\cdot)-\Gamma (0,x,t',\cdot) \right\Vert_{B_{p',q'}^\rho} \leq \textcolor{black}{C_\varepsilon} \frac{(t'-t)^\frac{\gamma-\eps}{\alpha}}{t^\frac{\gamma-\eps+\rho+\frac dp}{\alpha}} .
		\end{equation}

%		\begin{equation}\label{holder-time-gamma}
%			\forall 0\le s<t<t'\le T,\    |t-t'|\le (t-s) ,\   |\Gamma(s,x,t,y)-\Gamma (s,x,t',y)| \le C \frac{(t'-t)^\frac{\gamma}{\alpha}}{(t-s)^\frac{\gamma}{\alpha}} \bar p_\alpha (t'-s,y-x),
%		\end{equation}
	\end{itemize}	
\end{PROP}
The bounds \eqref{aronson-gamma}, \eqref{holder-forward-gamma} follow from \cite{Fit23}. %, \eqref{holder-time-gamma} is proved in the Appendix of \cite{fito:isso:meno:25}
%%%% De S a B. la version normalisee ne sert pas.
 Equation \eqref{holder-time-gamma_PP_QP_NO_NORM} is proved in Section \ref{THE_FINAL_EST}.

The Duhamel representation \eqref{duhamel-Diff} still holds for the density of the diffusion with the quantity: 
{
$$\E_{s,x}\left[b(r,X_r)\cdot\nabla_y  p_\alpha(t-r,y-X_r)\right]$$ defined as the product of the distribution $b(r,\cdot)$ by the smooth $\nabla_y  p_\alpha(t-r,y-\cdot)$ acting on the density $\Gamma(s,x,r,\cdot)$.} This was established in \cite{Fit23}, see as well \cite{fito:isso:meno:25}. We write:
{
\begin{align}
\label{DUHAMEL_DIFF_BESOV}
			\Gamma(s,x,t,y)
			&= p_\alpha(t-s,y-x)-\int_{s}^{ t} \langle \Gamma(s,x,r,\cdot),b(r,\cdot)\cdot\nabla_y  p_\alpha(t-r,y-\cdot) \rangle_{B_{p',q'}^{-\beta},B_{p,q}^\beta} d r\notag\\
			&= p_\alpha(t-s,y-x)-\int_{s}^{ t}\int_{\R^d}\Gamma(s,x,r,z) b(r,z)\cdot\nabla_y  p_\alpha(t-r,y-z) dz d r.
\end{align}
}

We resume the error analysis considering a test function  $\varphi \in B_{p,q}^{\beta} $, $\beta\in (\frac{1-\alpha-\frac dp-\frac{\alpha}{\vartheta} }{2},0) $ for a drift $b$ which belongs to the same function space in its spatial variable (recall indeed that $b\in L^\vartheta([0,T],B_{p,q}^\beta) $). Again, this is one of the aim of this paper to properly emphasize that the analysis of the error for the densities on a function space which can  be put in duality with the one the drift belongs to allows to improve the convergence rate or consider a very singular test function w.r.t. the series of papers \cite{jour:meno:24}, \cite{fito:jour:meno:25}, \cite{fito:isso:meno:25} in which a more precise pointwise error on the densities was investigated.
%, yielding to a lower convergence rate. 
Precisely, we have the following theorem.

\begin{THM}[Weak error in the Besov case]\label{THM_BESOV}
Assume {$b\in L^\vartheta([0,T],B_{p,q}^\beta)$ under \eqref{serrin}}. {Consider the weak  error for a (distributional) test function $\varphi\in B_{p,q}^\beta $:
	\begin{equation}
		\label{def-weak-error_BESOV}\mathcal{E}(\varphi,t,x,h):= \E_{0,x}\left[\varphi(X_t^h)-\varphi(X_t)\right]=\langle (\Gamma^h-\Gamma)(0,x,t,\cdot),\varphi \rangle_{B_{p',q'}^{-\beta},B_{p,q}^\beta},
	\end{equation}
	where $\langle \cdot,\cdot\rangle_{B_{p',q'}^{-\beta},B_{p,q}^\beta} $ stands for the duality pairing.}

Then, for $\varepsilon>0 $, there exists a constant $C_\varepsilon \ge 1$ s.t. for all {$\varphi\in B_{p,q}^\beta $,} $x\in \R^d, t\in [2h,T]$,
\begin{align}
|\mathcal{E}(\varphi,t,x,h)|\le %Ch^{\frac{\alpha\wedge (\alpha-1+ 2\beta)}{\alpha}}(1+t^{-\frac{\beta\wedge (1-\beta)}{\alpha}})(1+|\ln(h)|\I_{\beta=\frac 12}).
C_\varepsilon t^{-\frac{\frac dp-\beta}{\alpha}}h^{\frac{\gamma-\varepsilon}{\alpha}}\|\varphi\|_{B_{p,q}^\beta}\mbox{ where }\gamma=\alpha-1+2\beta-\frac{d}{p}-\frac{\alpha}{\vartheta}.
\label{BD_W_ERR_BESOV}
\end{align}
%with $\varepsilon>0 $ and $C:=C(\varepsilon)$ if $\beta\ge \frac 12 $ and $C=C(1-2\beta) $ if $\beta<\frac 12 $.
\end{THM}
%\begin{rem}[About the test function and the convergence rate]\textcolor{red}{Ca c'est ce que l'on souhaiterait par rapport au travail de Zimo. Pour l'instant je crains de ne pas savoir faire. On est en l'état coincé par la régularité en temps forward. Je pense que si on la cherche sur le schéma (pour lequel le drift est régulier) on peut utiliser une cancel$ $}
%We carefully mention that we here focused on the worst possible case we could handle for the test function, taking $\varphi \in B_{p,q}^\beta $. In particular, should we consider a \textit{better} test function, the analysis below would show that we would impose less singularity on the heat kernels involved. For a function  $\varphi\in B_{1,1}^{0} $ we should gain the corresponding intrinsinc regularity index $-\beta-\frac d{p} $   multiplied by $1/\alpha $ (characteristic time scale), yielding a convergence rate of order $h^{\frac {\alpha-1+\beta}{\alpha}} $ without time singularity in $t$.
%\end{rem}
\begin{rem}[About previous results and singular test functions]
One could wonder if the results of Theorem \ref{THM_BESOV} intersect those established in \cite{fito:isso:meno:25} for the densities. In that work, the same convergence was obtained for  $\|\frac{(\Gamma-\Gamma^h)(0,x,t,\cdot)}{\bar p_\alpha(t,\cdot-x)}\|_{B_{\infty,\infty}^\rho},\ \rho>-\beta$.
% was investigated and  (for which the same rate was obtained).
Denoting by $\mathcal P(\R^d)$ the  set of probability measures on $\R^d $, we recall, as e.g. a consequence of Lemma 5 in \cite{chau:jabi:meno:22-1} 
in  that:
\begin{equation}\label{INCL}
\mathcal P(\R^d) \subset \cup_{\ell\ge 1} B_{\ell,\infty}^{-\frac d{\ell'}}.
\end{equation}
To derive the upper bound on the densities from Theorem \ref{THM_BESOV}, one should be able to take the Dirac mass as a suitable test function
and choose indices satisfying as well the condition \eqref{serrin}.

From \eqref{INCL} this suggests to take $p=\ell,\beta=-\frac{d}{\ell'} $ for which condition \eqref{serrin}  would read as $\beta=-\frac{d}{\ell'} >\frac{1-\alpha+\frac d\ell}2\iff -\frac{d}{\ell'}>1-\alpha+d$ which cannot be satisfied, for any $\alpha\in (1,2] $ and any dimension $d\ge 1$. Hence, the control in \cite{fito:isso:meno:25} cannot be derived as a consequence of the current result. It requires a more thorough analysis of the heat kernels at hand.

Conversely, the current result cannot be either seen as a direct consequence of \cite{fito:isso:meno:25}. %where the studied quantity is .
\end{rem}

%%%%%%%%% Attention chaud: il y a la dérivée du flot...
%\begin{rem}[About the test function: an example]
%Let us mention that, from an applicative viewpoint, the previous result could be used for instance in order to estimate the discretization error for a sensitivity analysis. Assume for instance that $X$ corresponds to the dynamics of a financial asset with singular drift in $L^\infty([0,T],B_{\infty,\infty}^{\beta}),\ \beta\in (\frac{1-\alpha}2,0) $ and that one is interested in estimating the $\delta$ of an option associated with a pay-off $\Phi\in C^{\beta+1}=B_{\infty,\infty}^{\beta+1}$.
%\end{rem}

To proceed, we recall that for $k\in \llbracket 0,n-1\rrbracket ,\ t\in (t_k,T] $, the density of $X_t^h$ admits, conditionally to $X_{t_k}^h=x$, a transition density $\Gamma^h(t_k,x,t,\cdot) $, which enjoys a Duhamel type representation {(see Proposition 2 in \cite{fito:isso:meno:25})}: for all $y\in \R^d $,
		\begin{align}
			\Gamma^h(t_k,x,t,y)
			= p_\alpha(t-t_k,y-x)-\int_{t_k}^{ t}\E_{t_k,x}\left[\mathfrak b_h(r,X^h_{\tau_r^h})\cdot\nabla_y  p_\alpha(t-r,y-X^h_r)\right]d r,\label{duhamel-scheme_BESOV}
		\end{align} 
which somehow naturally extends the one in \eqref{duhamel-scheme}. For the analysis, we will thoroughly use the following lemma about the drift $\mathfrak b_h $ proved in \cite{fito:isso:meno:25} (see also Proposition 2 of \cite{chau:jabi:meno:22-1} for \eqref{APPROX_DRIFT_ERR} below) {apart for the statements \eqref{QUI_SERT_TOUT_LE_TEMPS_EN_BESOV_DRIFTED_SCHEME} and \eqref{QUI_SERT_TOUT_LE_TEMPS_EN_BESOV_DRIFTED_SCHEME_TIME_SENSI} which are respectively proved in Sections \ref{prqst} and \ref{prqstbis}.}
\begin{lem}[Useful bounds for $\mathfrak{b}_h$] \label{lemma-regularity-mollified-b} There exists $C\geq 1$ s.t. for all $h>0$ and all $(s,z)\in [0,T] \times \R^d$, $s\neq \tau_s^h $,
	\begin{itemize}
		\item Pointwise control (uniform in $z$)
		\begin{equation}\label{CTR_PONCTUEL_BH}
			|\mathfrak b_h(s,z)|\le C (s-\tau_s^h)^{-\frac d{\alpha p}+\frac\beta\alpha}\|b(s,\cdot)\|_{B_{p,q}^\beta}.
		\end{equation}
		\item Time-integrated pointwise control (uniform in $z$)
		\begin{equation}\label{CTR_PONCTUEL_BH_INT}
			\left|\int_{\tau_s^h}^s\mathfrak b_h(u,z)d u \right|\le C (s-\tau_s^h)^{\frac{\gamma}{\alpha}+\frac{1-\beta}{\alpha}}\|b\|_{L^\vartheta-B_{p,q}^\beta}.
			%h^{1-\frac{1}{\vartheta}-\frac{d}{\alpha p}+\frac{\beta}{\alpha}}=h^{\frac{\gamma-\beta}\alpha+\frac{1}\alpha},
		\end{equation} 
		\item Spatial H\"older modulus of the integrated drift\\
		For all $(z,z')\in (\R^d)^2 $, $\zeta\in [-\beta,\gamma -\beta) $,
		\begin{equation}\label{CTR_PONCTUEL_BH_INT_HOLDER} 
			\left|\int_{\tau_s^h}^s\Big(\mathfrak b_h(u,z)-\mathfrak b_h(u,z')\Big)d u \right|\le  C|z-z'|^\zeta (s-\tau_s^h)^{\frac{\gamma}{\alpha}+\frac{1-\beta-\zeta}{\alpha}}\|b\|_{L^\vartheta-B_{p,q}^\beta}.
			%h^{1-\frac{1}{\vartheta}-\frac{d}{\alpha p}+\frac{\beta}{\alpha}}=h^{\frac{\gamma-\beta}\alpha+\frac{1}\alpha},
		\end{equation}
		\item Besov norm of the mollified drift
		\begin{equation}
			\label{CTR_BESOV_BH}
			\|\mathfrak b_h(s,\cdot)\|_{B_{p,q}^\beta}\le \|b(s,\cdot)\|_{B_{p,q}^\beta}.
		\end{equation}
		\item Approximation in Besov norm. {For $\tilde \gamma\in [0,\gamma] $},
		\begin{equation}\label{APPROX_DRIFT_ERR}
		\|b(s,\cdot)-\mathfrak b_h(s,\cdot)\|_{\B_{p,q}^{\beta-{\tilde \gamma}}}\lesssim h^{\frac{{\tilde \gamma}}{\alpha}}\|b(s,\cdot)\|_{\B_{p,q}^{\beta}}.
		\end{equation}
		\item Besov norm of the conditioned scheme. 
		For all $\varphi\in B_{p,q}^\beta $, $\rho\in (-\beta,-\beta+\gamma) $, {$0\le s<t\le T$},
		\begin{align}
				\label{QUI_SERT_TOUT_LE_TEMPS_EN_BESOV_DRIFTED_SCHEME}
				\|\Gamma^h(0,x,\tau_s^h,\cdot)\nabla P^\alpha_{t-\tau_s^h}\varphi(\cdot+\int_{\tau_s^h}^s \mathfrak b_h(r,\cdot)dr )\|_{B_{p',q'}^{-\beta}}\lesssim& \Big[t^{\frac{\beta+\rho}{\alpha}}\Big(s^{-(\frac \rho\alpha+\frac{d}{p\alpha})}(t-s)^{-\frac 1\alpha+\frac \beta\alpha-\frac d{p\alpha}}\notag\\
				&+s^{-\frac d{\alpha p}}(t-s)^{-\frac 1\alpha+\frac{\beta}{\alpha}-\frac \rho\alpha-\frac d{p\alpha}}\Big)\Big]\|\varphi\|_{B_{p,q}^\beta}.
		\end{align}
		\item Time sensitivity for the scheme.
				For {$\varepsilon \in (0,\gamma)$}, all $\varphi\in B_{p,q}^\beta $, {$\rho\in(-\beta,-\beta+\gamma)$}, $s\in [h,\tau_t^h-h] $,
		\begin{align}
				\label{QUI_SERT_TOUT_LE_TEMPS_EN_BESOV_DRIFTED_SCHEME_TIME_SENSI}
				&\left\|\Gamma^h(0,x,\tau_s^h,\cdot)\Big[\nabla P^\alpha_{t-\tau_s^h}\varphi(\cdot+\int_{\tau_s^h}^s \mathfrak b_h(r,\cdot)dr )-\nabla P^\alpha_{t-s}\varphi(\cdot )\Big]\right\|_{B_{p',q'}^{-\beta}}\notag\\
                  \lesssim& h^\frac{\gamma-\varepsilon}{\alpha}\Big[t^{\frac{\beta+\rho}{\alpha}}\Big(s^{-(\frac \rho\alpha+\frac{d}{p\alpha})}(t-s)^{{-\frac\gamma\alpha+\frac\varepsilon\alpha}-\frac 1\alpha+\frac \beta\alpha-\frac d{p\alpha}}%\notag\\
				%&
				+s^{-\frac d{\alpha p}}(t-s)^{{-\frac\gamma\alpha+\frac\varepsilon\alpha}-\frac 1\alpha+\frac{\beta}{\alpha}-\frac \rho\alpha-\frac d{p\alpha}}\Big)\Big]\|\varphi\|_{B_{p,q}^\beta}.
		\end{align}

	\end{itemize}
\end{lem}
{
\begin{rem}[About the time range constraints for the controls involving the drift of the scheme]\label{REM_FOR_RANGE}
We point out that the control \eqref{QUI_SERT_TOUT_LE_TEMPS_EN_BESOV_DRIFTED_SCHEME} (which will be needed for the analysis of the contribution associated with the last time steps) is
indeed valid for $0\le s<t\le T $, whereas the additional constraint $s\in [h,\tau_t^h-h] $ appearing for \eqref{QUI_SERT_TOUT_LE_TEMPS_EN_BESOV_DRIFTED_SCHEME_TIME_SENSI} is precisely needed in order that $t-s $ and $t-\tau_s^h $ have \textit{somehow} the same order. They are on the considered time interval be both greater than the time step $h $.
\end{rem}
}

  We will as well need some controls on the smoothing effects of the stable type kernel on singular \textit{test} functions. Namely,
  \begin{PROP}[Schauder type estimates for the stable heat kernel]\label{HK_W_BESOV}
  Let $\varphi\in B_{p,q}^\beta $. Then, under \eqref{serrin}, for all $0\le s<t\le T$, $\rho\in (-\beta,-\beta+\gamma) $,
  \begin{align}
\|\nabla P_{t-s}^\alpha \varphi\|_{L^\infty}&\lesssim (t-s)^{-\frac 1\alpha+\frac \beta \alpha-\frac d{p\alpha}}\|\varphi\|_{B_{p,q}^\beta},\label{BD_INFTY_SC_BESOV}\\
\|{\partial_t^\theta\partial^{\mathbf a}} P_{t-s}^\alpha \varphi\|_{B_{\infty,\infty}^{\rho}}&\lesssim (t-s)^{-{\theta-\frac{|{\mathbf a|}}\alpha}+\frac \beta \alpha-\frac \rho\alpha-\frac d{p\alpha}}\|\varphi\|_{B_{p,q}^\beta},\ {|\mathbf a|\le 2,\ \theta\in \{0,1\}}\label{BD_RHO_SC_BESOV}.
  \end{align}
  Also, for the shift associated with the drift of the scheme:
  \begin{align}
\|{\partial^{\mathbf a}} P_{t-\tau_s^h}^\alpha \varphi(\cdot+\int_{\tau_s^h}^s \mathfrak b_h(r,\cdot)dr )\|_{L^\infty}&\lesssim (t-\tau_s^h)^{-\frac{{|{\mathbf a}|}}\alpha+\frac{\beta}\alpha -\frac d{p\alpha}}\|\varphi\|_{B_{p,q}^\beta},\label{BD_INFTY_SC_BESOV_SCHEME}\\
\|{\partial^{\mathbf a}}  P_{t-\tau_s^h}^\alpha \varphi(\cdot+\int_{\tau_s^h}^s \mathfrak b_h(r,\cdot)dr )\|_{B_{\infty,\infty}^{\rho}}&\lesssim (t-s)^{-\frac{{|{\mathbf a|}}}\alpha+\frac \beta \alpha-\frac \rho\alpha-\frac d{p\alpha}}\|\varphi\|_{B_{p,q}^\beta},\ {|\mathbf a|\le 2}\label{BD_RHO_SC_BESOV_SCHEME}.
  \end{align}
  Eventually, for the time {and spatial} sensitivities, for $s\in [h,\tau_t^h-h]$:
  \begin{align}
\|\nabla P_{t-s}^\alpha \varphi-\nabla P_{t-\tau_s^h}^\alpha \varphi\|_{L^\infty}&\lesssim (t-s)^{-\frac 1\alpha+\frac \beta \alpha-\frac d{p\alpha}}\|\varphi\|_{B_{p,q}^\beta}\Big(\frac{s-\tau_s^h}{t-s}\Big)^\theta,\ {\theta\in [0,1],}\label{BD_INFTY_SC_BESOV_TIME_SENSI}\\
\|\nabla P_{t-{\tau_s^h}}^\alpha \varphi{(\cdot+z)}-\nabla P_{t-\tau_s^h}^\alpha\varphi{(\cdot)} \|_{L^\infty}&\lesssim (t-s)^{-\frac 1\alpha+\frac \beta \alpha-\frac d{p\alpha}}\|\varphi\|_{B_{p,q}^\beta}\frac{|z|}{(t-s)^{\frac 1\alpha}},\label{BD_INFTY_SC_BESOV_SPACE_SENSI}\\
\|\nabla P_{t-s}^\alpha \varphi-\nabla P_{t-\tau_s^h}^\alpha\varphi\|_{B_{\infty,\infty}^{\rho}}&\lesssim (t-s)^{-\frac 1\alpha+\frac \beta \alpha-\frac \rho\alpha-\frac d{p\alpha}}\|\varphi\|_{B_{p,q}^\beta}\Big(\frac{s-\tau_s^h}{t-s}\Big)^\theta,\ \theta\in [0,1]\label{BD_RHO_SC_BESOV_TIME_SENSI}.
  \end{align}
  
%  Eventually, for the terms involving the transition density as well we have (\textcolor{red}{Fait doublon avec \eqref{QUI_SERT_TOUT_LE_TEMPS_EN_BESOV}}, il faut étendre (5.15) qui ne tient pa compte de la régularité supplémentaire, on a écrit \eqref{QUI_SERT_TOUT_LE_TEMPS_EN_BESOV_FOR_DELTA_3} qui devrait faire fonctionner les choses):
%  \begin{align}
%\label{NEEDED_FOR_DELTA_3_BESOV}
%  &\|\Gamma(0,x,\tau_s^h,\cdot)\nabla  P^\alpha_{t-s}\varphi\|_{B_{p',q'}^{-\beta+\gamma-\varepsilon}}\notag\\
%  \lesssim& \frac{1}{(t-s)^{\frac 1\alpha}} \Big(\frac{1}{s^{\frac{d}{\alpha p}}}+\frac{1}{(t-s)^{\frac{d}{\alpha p}}}\Big)\Big(\frac{1}{s^{\frac{-\beta+\gamma-\varepsilon}{\alpha}}}\frac{1}{(t-s)^{\frac{-\beta+\frac dp}{\alpha}}}+\frac{1}{(t-s)^{\frac{-2\beta+\gamma-\varepsilon+\frac dp}{\alpha}}} \Big) ds\|\varphi\|_{B_{p,q}^{\beta}}.
%  \end{align}

  \end{PROP}{Proposition \ref{HK_W_BESOV} is proved in Section \ref{prHK_W_BESOV}.}
\begin{rem}[About the approximate drift in the norms]
Observe that the controls in \eqref{BD_INFTY_SC_BESOV_SCHEME}, \eqref{BD_RHO_SC_BESOV_SCHEME} emphasize that the approximate drift does not affect the bounds. Indeed, the r.h.s. is similar to what would have appeared for the corresponding norm without the drift.
\end{rem}
}
We give additional controls involving the densities:
\begin{PROP}[Regularity estimates involving the density and the heat semigroup.]\label{PROP_SENSI_WITH_SHIFTED_INDEX}
{\phantom{BoUHH}}\\
{Let $\varphi\in B_{p,q}^\beta $. Then, under \eqref{serrin}, for all $0\le s<t\le T$, $\rho\in (-\beta,-\beta+\gamma) $,}%\footnote{Déplacer ces deux dernières équations ailleurs. Plus adapté dans une section technique.}, 
\begin{align}\label{QUI_SERT_TOUT_LE_TEMPS_EN_BESOV}
\|\Gamma(0,x,s,\cdot) \nabla P_{t-s}^{\alpha}\varphi\|_{B_{p',q'}^{-\beta}}\lesssim&\Big[t^{\frac{\beta+\rho}{\alpha}}\Big(s^{-(\frac \rho\alpha+\frac{d}{p\alpha})}(t-s)^{-\frac 1\alpha+\frac \beta\alpha-\frac d{p\alpha}}+s^{-\frac d{\alpha p}}(t-s)^{-\frac 1\alpha+\frac{\beta}{\alpha}-\frac \rho\alpha-\frac d{p\alpha}}\Big)\Big]\|\varphi\|_{B_{p,q}^\beta},%\\
%\lesssim &\Big(\frac{1}{s^{-\frac{\beta}\alpha+\frac{d}{p\alpha}}}\frac{1}{(t-s)^{\frac 1\alpha-\frac \beta\alpha}}+ \frac{1}{(t-s)^{\frac 1\alpha-\frac{2\beta}\alpha}} \Big).\notag
\end{align}
and {for $\varepsilon \in (0,\gamma)$},
\begin{align}\label{QUI_SERT_TOUT_LE_TEMPS_EN_BESOV_FOR_DELTA_3}
\|\Gamma(0,x,s,\cdot) \nabla P_{t-s}^{\alpha}\varphi\|_{B_{p',q'}^{-\beta+\gamma-\varepsilon}}\lesssim&\Big[t^{\frac{\varepsilon}{2\alpha}}\Big(s^{-(\frac {-\beta+\gamma-\frac \varepsilon 2}\alpha+\frac{d}{p\alpha})}(t-s)^{-\frac 1\alpha+\frac \beta\alpha-\frac d{p\alpha}}+s^{-\frac d{\alpha p}}(t-s)^{-\frac 1\alpha+\frac{2\beta}{\alpha}-\frac{\gamma-\frac \varepsilon2}\alpha-\frac d{p\alpha}}\Big)\Big]\|\varphi\|_{B_{p,q}^\beta}.
\end{align}
\end{PROP}
{Proposition \ref{PROP_SENSI_WITH_SHIFTED_INDEX} is proved in Section \ref{prpswsi}.}
\begin{rem}[About the singularity exponents in the previous proposition]
%Observe that, since for \eqref{QUI_SERT_TOUT_LE_TEMPS_EN_BESOV} we want to bound the $B_{p',q'}^{-\beta} $ norm of the product $\Gamma(0,x,s,\cdot) \nabla P_{t-s}^{\alpha}\varphi $ (positive regularity) and we aim at  making the $B_{p,q}^\beta $ norm of $\varphi $ appear, this leads to consider from \eqref{YOUNG} a norm of regularity  $-2\beta $ for  the gradient of the heat kernel $ $
%
In \eqref{QUI_SERT_TOUT_LE_TEMPS_EN_BESOV}, the parameter $\rho $ is somehow technical (needed for product rules in Besov spaces) and does not give any contribution once the singularities in the r.h.s. are integrated in time (provided the exponents give integrable contributions).  Similarly the contribution of the distributional regularity is of order $2\beta/\alpha $ once the r.h.s. is integrated in time. This roughly comes from Young inequalities \eqref{YOUNG}
and the fact that we aim at  making the $B_{p,q}^\beta $ norm of $\varphi $ appear. The other time singularities appearing are similar to those of the {Lebesgue} case (integrability of the stable heat kernel and intrinsic time singularity for the gradient).\\
Similar arguments give the second bound, which is somehow similar to the first one, where the regularity index as been replaced by   $-\beta+\gamma-\varepsilon $ taking $ \rho=-\beta-\gamma-\varepsilon/2$ which is indeed bigger.
%\textcolor{black}{De S.: je trouve la deuxième borne étrange...}
\end{rem}

\subsection{Proof of Theorem \ref{THM_BESOV}}
The  error associated with the scheme \eqref{euler-scheme-besov} and the test function $\varphi $ writes for $t\ge 2h$:
    \begin{align}
\mathcal{E}(\varphi,t,x,h)=&\langle (\Gamma^h-\Gamma)(0,x,\cdot,\cdot),\varphi\rangle_{B_{p',q'}^{-\beta}, B_{p,q}^{\beta}}\notag\\	 
=   &\int_{}\big(\Gamma^h(0,x,t,y)-\Gamma(0,x,t,y)\big) \varphi(y) d y\notag\\
=      &    \int_h^{\tau_t^h-h} \E_{0,x} \bigg[ {\mathfrak b}_h(s,X_{\tau_s^h})\cdot \nabla P_{t-s}^\alpha \varphi(X_{\tau_s^h})
        -{\mathfrak b}_h(s,X_{\tau_s^h}^h)\cdot \nabla  P_{t-s}^\alpha \varphi(X_{\tau_s^h}^h) \bigg] d s\notag\\%Gronwall
&   + \int_h^{\tau_t^h-h} \E_{0,x} \left[ b(s,X_s)\cdot \nabla P_{t-s}^\alpha \varphi(X_s)-b(s,X_{\tau_s^h})\cdot \nabla  P_{t-s}^\alpha \varphi (X_{\tau_s^h})
        \right] d s\notag\\% Sensi temps dens
                &  + \int_h^{\tau_t^h-h} \E_{0,x} \left[ b(s,X_{\tau_s^h})\cdot \nabla P^\alpha_{t-s}\varphi(X_{\tau_s^h})-%P_{s-\tau_s^h}^\alpha{b}
	\mathfrak b_h(s,X_{\tau_s^h})\cdot \nabla  P^\alpha_{t-s}\varphi(X_{\tau_s^h})\right] d s \notag\\ %% Approxi du drift
                %+ \int_h^{\tau_t^h-h} \E_{0,x} \left[b(U_{\lfloor s/h\rfloor},X_{\tau_s^h}^h)\cdot \left( \nabla P_{t-s}^\alpha \varphi(X_{\tau_s^h}^h)-\nabla  P_{t-s}^\alpha \varphi(X_s^h)\right)\right]d s\notag\\%Sensi temps Euler
		&  +\int_h^{\tau_t^h-h} \E_{0,x} \left[\mathfrak b_h(s,X_{\tau_s^h}^h)\cdot \left( \nabla P_{t-s}^\alpha\varphi(X_{\tau_s^h}^h)-\nabla  P_{t-s}^\alpha \varphi(X_s^h)\right)\right]d s\notag\\%Sensi temps HK
		&+\int_0^h \E_{0,x} \left[ b(s,X_s)\cdot \nabla P_{t-s}^\alpha \varphi(X_s)-{\mathfrak b}_h(s,{x})\cdot\nabla P_{t-s}^\alpha \varphi(X_s^h) \right] d s\notag\\% First time step
  		&+\int_{\tau_t^h-h}^t \E_{0,x} \left[ b(s,X_s)\cdot \nabla P_{t-s}^\alpha \varphi(X_s)-{\mathfrak b}_h(s,X_{\tau_s^h}^h)\cdot\nabla P_{ t-s}^\alpha\varphi (X_s^h)  \right] d s\notag\\% Last time step
        =:&\Delta^1 + \Delta^2  + \Delta^6 + \Delta^3+\Delta^4 +\Delta^5.\label{DECOUP_ERR_BESOV}
    \end{align}
    Note that the terms are not numbered in increasing order to match as much as possible the numbering in \eqref{decomperrt} for Lebesgue drift {coefficients}. Still, the contribution $\Delta^3_t(y)$ of the time randomization  in \eqref{decomperrt} and the first time step contribution $\Delta^3$ are of different nature. Also note that  we assume that $t\ge 2h$ to slightly simplify the argument but could deal with the case $t\in [0,2h)$ like we did in the proof of Theorem \ref{thmsinglebdrift} when the drift belongs to Lebesgue spaces.

In the current section, when we refer to a term $(\Delta^i)_{i\in \{1,\cdots,6\}} $, it will be the one defined in \eqref{DECOUP_ERR_BESOV}.

Let us start with the term $\Delta^4 $. Write:
\begin{align}
|\Delta^4|{\lesssim}&\int_0^h  \Big|\int \Gamma(0,x,s,z)b(s,z) \nabla P_{t-s}^\alpha \varphi(z)dz\Big| ds+\int_0^h\|{\mathfrak b}_h{(s,\cdot)}\|_{L^\infty} \|\nabla P_{t-s}^\alpha \varphi(t-s,\cdot)\|_{L^\infty}  d s\notag\\
\lesssim& \int_0^h \|b(s,\cdot)\|_{B_{p,q}^\beta}\|\Gamma(0,x,s,\cdot)\nabla P_{t-s}\varphi\|_{B_{p',q'}^{-\beta}}ds+\int_0^h (s-\tau_s^h)^{-\frac d{\alpha p}+\frac\beta\alpha}\|b(s,\cdot)\|_{B_{p,q}^\beta}\|\nabla p_\alpha(t-s,\cdot)\|_{B_{p',q'}^{-\beta}} \|\varphi\|_{B_{p,q}^{\beta}}ds\notag\\
\lesssim& \Big(\int_0^h \|b(s,\cdot)\|_{B_{p,q}^\beta}\Big[t^{\frac{\beta+\rho}{\alpha}}\Big(s^{-(\frac \rho\alpha+\frac{d}{p\alpha})}(t-s)^{-\frac 1\alpha+\frac \beta\alpha-\frac d{p\alpha}}+s^{-\frac d{\alpha p}}(t-s)^{-\frac 1\alpha+\frac{\beta}{\alpha}-\frac \rho\alpha-\frac d{p\alpha}}\Big)\Big]ds\notag\\
&+\int_0^h (s-\tau_s^h)^{-\frac d{\alpha p}+\frac\beta\alpha}\|b(s,\cdot)\|_{B_{p,q}^\beta}(t-s)^{-\frac 1\alpha+\frac \beta\alpha-\frac{d}{p\alpha}}ds\Big) \|\varphi\|_{B_{p,q}^{\beta}}\notag,
\end{align}
using %\eqref{serrin}
{\eqref{CTR_PONCTUEL_BH}, \eqref{YOUNG}} and \eqref{SING_STABLE_HK} as well as \eqref{QUI_SERT_TOUT_LE_TEMPS_EN_BESOV} {with $\rho\in(-\beta,-\beta+\gamma)$} for the last inequality.
Then,
\begin{align}
|\Delta^4|\lesssim%& \Big(\int_0^h \|b(s,\cdot)\|_{B_{p,q}^\beta}\Big[t^{\frac{\beta+\rho}{\alpha}}\Big(s^{-(\frac \rho\alpha+\frac{d}{p\alpha})}(t-s)^{-\frac 1\alpha+\frac \beta\alpha-\frac d{p\alpha}}+s^{-\frac d{\alpha p}}(t-s)^{-\frac 1\alpha+\frac{\beta}{\alpha}-\frac \rho\alpha-\frac d{p\alpha}}\Big)\Big]ds\notag\\
%&+\int_0^h (s-\tau_s^h)^{-\frac d{\alpha p}+\frac\beta\alpha}\|b(s,\cdot)\|_{B_{p,q}^\beta}(t-s)^{-\frac 1\alpha+\frac \beta\alpha-\frac{d}{p\alpha}}ds\Big) \|\varphi\|_{B_{p,q}^{\beta}}\notag\\
%\lesssim
&\|\varphi\|_{B_{p,q}^{\beta}}\|b\|_{L^\vartheta-B_{p,q}^\beta}\Big(t^{\frac{\beta+\rho}{\alpha}}\Big[\Big(\int_0^h s^{-(\frac \rho\alpha+\frac{d}{p\alpha})\vartheta'} ds\Big)^{\frac{1}{\vartheta'}} t^{-\frac 1\alpha+\frac{\beta}{\alpha}-\frac d{p\alpha}}+\Big(\int_0^h s^{-\frac{d}{p\alpha}\vartheta'} ds\Big)^{\frac 1{\vartheta'}} t^{-\frac 1\alpha+\frac{\beta}{\alpha}-\frac \rho\alpha-\frac d{p\alpha}}\Big]\notag\\
&+\Big(\int_0^h (s-\tau_s^h)^{(-\frac d{\alpha p}+\frac\beta\alpha)\vartheta'}ds\Big)^{\frac{1}{\vartheta'}} t^{-\frac 1\alpha+\frac \beta\alpha-\frac{d}{p\alpha}}\Big)\notag\\
\lesssim&\|\varphi\|_{B_{p,q}^{\beta}}\|b\|_{L^\vartheta-B_{p,q}^\beta}\Big(t^{\frac{2\beta+\rho}{\alpha}-\left(\frac{1+\frac dp}{\alpha}\right)}h^{1-\frac{1}{\alpha}(\rho+\frac dp+\frac{\alpha}{\vartheta})}+t^{\frac{2\beta}{\alpha}-\left(\frac{1+\frac dp}{\alpha}\right)}h^{1-\frac{1}{\alpha}(\frac dp+\frac{\alpha}{\vartheta})}+t^{\frac{{\beta}}{\alpha}-\left(\frac{1+\frac dp}{\alpha}\right)}h^{1-\frac{1}{\alpha}(-\beta+\frac dp+\frac{\alpha}{\vartheta})}\Big)\notag\\
\lesssim&\|\varphi\|_{B_{p,q}^{\beta}}\|b\|_{L^\vartheta-B_{p,q}^\beta}t^{\frac{2\beta+\rho}{\alpha}-\left(\frac{1+\frac dp}{\alpha}\right)}h^{1-\frac{1}{\alpha}(\rho+\frac dp+\frac{\alpha}{\vartheta})}\lesssim \|\varphi\|_{B_{p,q}^{\beta}}\|b\|_{L^\vartheta-B_{p,q}^\beta}t^{\frac{{2\beta+\rho}}{\alpha}-\left(\frac{1+\frac dp}{\alpha}\right)}h^{\frac \gamma \alpha}h^{\frac 1\alpha-\frac{1}{\alpha}(\rho+2\beta)}\notag\\
\lesssim &\|\varphi\|_{B_{p,q}^{\beta}}\|b\|_{L^\vartheta-B_{p,q}^\beta}t^{-\frac d{\alpha p}}h^{\frac \gamma \alpha}\lesssim \|\varphi\|_{B_{p,q}^{\beta}}\|b\|_{L^\vartheta-B_{p,q}^\beta}t^{\frac{\varepsilon-\frac d{ p}}{\alpha}}h^{\frac {\gamma -\varepsilon}\alpha},\label{CTR_DELTA_5_BESOV}
\end{align}
{recalling $\gamma=\alpha-1+2\beta-\frac dp-\frac{\alpha}\vartheta $ and choosing $ \rho\in(-\beta,-\beta+\gamma)$ s.t. $2\beta+\rho<0 $ for the last but one inequality. Note that this is e.g. possible for $\rho=-\beta+\varepsilon $ for $\varepsilon>0 $ small enough.}
For $\Delta^5 $ write:
\begin{align}
|\Delta^5|=&\int_{\tau_t^h-h}^t  \Big|\int \Gamma(0,x,s,z)b(s,z) \nabla P_{t-s}^\alpha \varphi(z)dz\Big| ds\notag\\
&+\int_{\tau_t^h-h}^t\Big|\int \Gamma^h(0,x,\tau_s^h,z)\mathfrak b_h(s,z) \nabla P_{t-{\tau_s^h}}^\alpha \varphi(z+\int_{\tau_s^h}^s\mathfrak b_h(r,z)dr)dz\Big|  d s\notag,
\end{align}
{where for the last inequality we remark that, similarly to \eqref{CONDITIONING_FOR_EULER_SCHEME} in the Hölder setting,
\begin{align}
\E_{0,x} \left[ {\mathfrak b}_h(s,X_{\tau_s^h}^h)\cdot\nabla P_{ t-s}^\alpha\varphi (X_s^h)  \right]
&=\E_{0,x} \left[{\mathfrak b}_h(s,X_{\tau_s^h}^h)\cdot \E[\nabla P_{ t-s}^\alpha\varphi (X_s^h) |{X_{\tau_s^h}^h}] \right]\notag\\
&=\E_{0,x} \left[{\mathfrak b}_h(s,X_{\tau_s^h}^h)\cdot \nabla P_{t-{\tau_s^h}}^\alpha\varphi\Big(X_{\tau_s^h}^h+\int_{\tau_s^h}^s\mathfrak b_h(r,X_{\tau_s^h}^h)dr\Big)\right].\label{condeul}
\end{align}
The last equality can be viewed as a consequence of Itô's formula (harmonicity of the stable heat kernel). 
%Namely, similarly to \eqref{CONDITIONING_FOR_EULER_SCHEME} in the Hölder settting, it holds that
%\begin{align}\label{CONDITIONING_FOR_EULER_SCHEME_BESOV}
%\E[\nabla P_{t-s}^\alpha\varphi(X_s^h)|X_{{\tau_s}^h}^h]=\nabla P_{t-{\tau_s^h}}^\alpha\varphi(X_{\tau_s^h}^h+\int_{\tau_s^h}^sb(r,X_{\tau_s^h}^h)dr),
%\end{align}
}
The difference with $\Delta^4 $ is now that {we do not take} the supremum norm of the approximating drift which would have led to a too coarse control because of the appearance of the Sobolev exponent$-d/p \alpha $ exponent twice, i.e. in both the $s$ and $t-s $ variables (see control \eqref{QUI_SERT_TOUT_LE_TEMPS_EN_BESOV}) {but use duality instead. From \eqref{dual-ineq}, we get:} 
\begin{align}
|\Delta^5|\lesssim& \int_{\tau_t^h-h}^t\|b(s,\cdot)\|_{B_{p,q}^\beta}\|\Gamma(0,x,s,\cdot)\nabla P_{t-s}\varphi\|_{B_{p',q'}^{-\beta}}ds\notag\\
&+\int_{\tau_t^h-h}^t \|\mathfrak b_h(s,\cdot)\|_{B_{p,q}^\beta}\|\Gamma^h(0,x,{\tau^h_s},\cdot)\nabla P^\alpha_{t-\tau_s^h}\varphi(\cdot+\int_{\tau_s^h}^s \mathfrak b_h(r,\cdot)dr )\|_{B_{p',q'}^{-\beta}}ds.\notag
\end{align}
Using now \eqref{QUI_SERT_TOUT_LE_TEMPS_EN_BESOV}, \eqref{QUI_SERT_TOUT_LE_TEMPS_EN_BESOV_DRIFTED_SCHEME} and \eqref{CTR_BESOV_BH}, one derives for $\rho\in (-\beta,-\beta+\gamma) $
\begin{align}
&|\Delta^5|\notag\\
\lesssim& \Big(\int_{\tau_t^h-h}^t (\|b(s,\cdot)\|_{B_{p,q}^\beta}+\|\mathfrak b_h(s,\cdot)\|_{B_{p,q}^\beta})\Big[t^{\frac{\beta+\rho}{\alpha}}\Big(s^{-(\frac \rho\alpha+\frac{d}{p\alpha})}(t-s)^{-\frac 1\alpha+\frac \beta\alpha-\frac d{p\alpha}}+s^{-\frac d{\alpha p}}(t-s)^{-\frac 1\alpha+\frac{\beta}{\alpha}-\frac \rho\alpha-\frac d{p\alpha}}\Big)\Big]ds%\notag\\
%&+\int_{\tau_t^h-h}^t %(s-\tau_s^h)^{-\frac d{\alpha p}+\frac\beta\alpha}
%(t-s)^{-\frac 1\alpha+\frac \beta\alpha-\frac{d}{p\alpha}}ds
\Big)\notag \\
&\times\|\varphi\|_{B_{p,q}^{\beta}}\notag\\
\lesssim& \|\varphi\|_{B_{p,q}^{\beta}}\|b\|_{L^\vartheta-B_{p,q}^\beta}\Big( t^{\frac{\beta-\frac{d}{p}}{\alpha}}\Big(\int_{\tau_t^h-h}^t (t-s)^{\vartheta'(-\frac 1\alpha+\frac \beta\alpha-\frac d{p\alpha})}ds\Big)^{\frac 1{\vartheta'}}+t^{\frac{\beta+\rho-\frac{d}{p}}{\alpha}}\Big(\int_{\tau_t^h-h}^t (t-s)^{\vartheta'(-\frac 1\alpha+\frac \beta\alpha-\frac \rho\alpha-\frac d{p\alpha})}ds\Big)^{\frac 1{\vartheta'}}\notag\\
\lesssim &\|\varphi\|_{B_{p,q}^{\beta}}\|b\|_{L^\vartheta-B_{p,q}^\beta} t^{\frac{\beta+\rho-\frac dp}{\alpha}} h^{\frac{\gamma-\beta-\rho}{\alpha}}\lesssim  \|\varphi\|_{B_{p,q}^{\beta}}\|b\|_{L^\vartheta-B_{p,q}^\beta} t^{\frac{\varepsilon-\frac dp}{\alpha}} h^{\frac{\gamma-\varepsilon}{\alpha}} ,\label{CTR_DELTA_6_BESOV}
\end{align}
taking eventually $\rho=-\beta+\varepsilon,\ \varepsilon>0 $.

For the term $\Delta^2 $, we are led to consider the forward sensitivity in time for the density of the diffusion. Namely, using \eqref{dual-ineq} for the second inequality, then {the product rule} \eqref{PR} for the third, {we obtain}
\begin{align}
|\Delta^2|=&\Big|\int_h^{\tau_t^h-h} \E_{0,x} \left[ b(s,X_s)\cdot \nabla P_{t-s}^\alpha \varphi(X_s)-b(s,X_{\tau_s^h})\cdot \nabla  P_{t-s}^\alpha \varphi (X_{\tau_s^h})\right]\Big|\notag\\
&\lesssim \int_{h}^{\tau_t^h-h}\Big|\int_{\R^d} \big(\Gamma(0,x,s,y)-\Gamma(0,x,\tau_s^h,y)\big) b(s,y)\nabla P_{t-s}^\alpha \varphi(y) dy \Big| ds\notag\\
           &\lesssim \int_{h}^{\tau_t^h-h} \|\Gamma(0,x,s,\cdot)-\Gamma(0,x,\tau_s^h,\cdot)\|_{B_{p',q'}^{-\beta}}\|b(s,\cdot)\nabla P_{t-s}^\alpha \varphi\|_{B_{p,q}^{{\beta}}} ds\notag\\&\lesssim \int_{h}^{\tau_t^h-h} \|\Gamma(0,x,s,\cdot)-\Gamma(0,x,\tau_s^h,\cdot)\|_{B_{p',q'}^{-\beta}}\|b(s,\cdot)\|_{B_{p,q}^\beta}\|\nabla P_{t-s}^\alpha \varphi\|_{B_{\infty,\infty}^{-\beta+\frac{3\varepsilon} 4}} ds\notag.
\end{align}
{Since, from Proposition 2.2 in \cite{Sawano-18},
$
B^{\rho}_{p',q'}\hookrightarrow B^{-\beta}_{p',q'},\ \rho>-\beta$, we can then apply \eqref{holder-time-gamma_PP_QP_NO_NORM}  taking $\rho=-\beta+\varepsilon/4 $, which together with  \eqref{BD_RHO_SC_BESOV}, gives}\begin{align*}
|\Delta^2|&\lesssim \int_h^{\tau_t^h-h} \frac{h^{\frac{\gamma-\varepsilon}\alpha}}{(\tau_s^h)^{\frac{\gamma-\frac{3\varepsilon}4-\beta+\frac dp}{\alpha}}} \|b(s,\cdot)\|_{B_{p,q}^\beta} (t-s)^{-\frac{1}{\alpha}+\frac{2\beta}{\alpha}-\frac{3\varepsilon}{4 \alpha}-\frac{d}{p\alpha}} ds \|\varphi\|_{B_{p,q}^{\beta}}\notag\\
&\lesssim h^{\frac{\gamma-\varepsilon}\alpha}\|\varphi\|_{B_{p,q}^\beta} \|b\|_{L^\vartheta-B_{p,q}^\beta}\Big(\int_0^t s^{-(\frac{\gamma-\frac{3\varepsilon}4-\beta+\frac dp}{\alpha})\vartheta'}
(t-s)^{(-\frac{1}{\alpha}+\frac{2\beta}{\alpha}-\frac{3\varepsilon}{4\alpha}-\frac{d}{p\alpha})\vartheta'}ds\Big)^{\frac{1}{\vartheta'}}.\notag
 \end{align*}

  Recall now that 
\begin{align*}
(\frac{\gamma-\frac{3\varepsilon} 4-\beta+\frac dp}{\alpha})\vartheta'<1&\iff \alpha-\frac{\alpha}{\vartheta}>\gamma-\frac{3\varepsilon} 4-\beta+\frac dp
\iff 0>-1+ \beta-\frac{3\varepsilon} 4, 
\end{align*}
which is indeed always satisfied. {Also, $(-\frac{1}{\alpha}+\frac{2\beta}{\alpha}-\frac{3\varepsilon}{4\alpha}-\frac{d}{p\alpha})\vartheta'>-1\iff -1+2\beta- \frac{3\varepsilon} 4-\frac{d}{p}>-\alpha+\frac{\alpha}{\vartheta} \iff \gamma>\frac{3\varepsilon}4$ which can be satisfied for $\varepsilon $ small enough}. Integrating the time singularities, we then derive:
\begin{align}
|\Delta^2|&\lesssim h^{\frac{\gamma-\varepsilon}\alpha}\|\varphi\|_{B_{p,q}^\beta} \|b\|_{L^\vartheta-B_{p,q}^\beta} t^{-\frac{\gamma-\frac{3\varepsilon} 4-\beta+\frac dp}{\alpha} +\frac{\gamma-\frac{3\varepsilon} 4}{\alpha}}=h^{\frac{\gamma-\varepsilon}\alpha}\|\varphi\|_{B_{p,q}^\beta} \|b\|_{L^\vartheta-B_{p,q}^\beta} t^{\frac{\beta-\frac dp}{\alpha}}. \label{CTR_DELTA_2_BESOV}
\end{align}
Let us turn to $\Delta^6 $ and write:
\begin{align*}
|\Delta^6|=&\Big|\int_h^{\tau_t^h-h} \E_{0,x} \left[ b(s,X_{\tau_s^h})\cdot \nabla P^\alpha_{t-s}\varphi(X_{\tau_s^h})-
	\mathfrak b_h(s,X_{\tau_s^h})\cdot \nabla  P^\alpha_{t-s}\varphi(X_{\tau_s^h})\right] d s\Big|\notag\\
=&\int_h^{\tau_t^h-h} \Big|\int_{\R^d} \big(b(s,y)-\mathfrak b_h(s,y)\big)\Big(\Gamma(0,x,\tau_s^h,y)\nabla  P^\alpha_{t-s}\varphi(y) \Big)	dy\Big|ds\notag\\
\lesssim & \int_h^{\tau_t^h-h} \|(b-\mathfrak b_h)(s,\cdot)\|_{B_{p,q}^{\beta-\gamma+\varepsilon}}\|\Gamma(0,x,\tau_s^h,\cdot)\nabla  P^\alpha_{t-s}\varphi\|_{B_{p',q'}^{-\beta+\gamma-\varepsilon}} ds\notag\\
\lesssim& h^{\frac{\gamma-\varepsilon}\alpha}\|b\|_{L^\vartheta-B_{p,q}^\beta}\Big(\int_h^{\tau_t^h-h}
\Big[t^{\frac{\varepsilon}{2\alpha}}\Big(s^{-(\frac {-\beta+\gamma-\frac \varepsilon 2}\alpha+\frac{d}{p\alpha})}(t-s)^{-\frac 1\alpha+\frac \beta\alpha-\frac d{p\alpha}}+s^{-\frac d{\alpha p}}(t-s)^{-\frac 1\alpha+\frac{2\beta}{\alpha}-\frac{\gamma-\frac \varepsilon2}\alpha-\frac d{p\alpha}}\Big)\Big]^{\vartheta'} ds\Big)^{\frac 1{\vartheta'}}\notag\\
&\times \|\varphi\|_{B_{p,q}^{\beta}},
\end{align*}
using \eqref{APPROX_DRIFT_ERR} {with $\tilde \gamma=\gamma-\varepsilon $}  and \eqref{QUI_SERT_TOUT_LE_TEMPS_EN_BESOV_FOR_DELTA_3} for the last inequality. {The singularities are integrable according to the above discussion for $\Delta^2$.} Eventually, recalling {again} that $\gamma=\alpha-1+2\beta-\frac dp-\frac{\alpha}{\vartheta} $ we get:
\begin{align}\label{CTR_DELTA_3_BESOV}
|\Delta^6|\lesssim h^{\frac{\gamma-\varepsilon}\alpha}\|b\|_{L^\vartheta-B_{p,q}^\beta}\|\varphi\|_{B_{p,q}^{\beta}} {t^{\frac \varepsilon \alpha-\frac d{p\alpha}}}.
\end{align}
It remains to handle $\Delta^3 $. Namely, {by \eqref{condeul},}
\begin{align}
|\Delta^3|=&\Big|\int_h^{\tau_t^h-h} \E_{0,x} \left[\mathfrak b_h(s,X_{\tau_s^h}^h)\cdot \left( \nabla P_{t-s}^\alpha\varphi(X_{\tau_s^h}^h)-\nabla  P_{t-s}^\alpha \varphi(X_s^h)\right)\right]d s\Big|\notag\\
\lesssim& \int_h^{\tau_t^h-h}\Big|\int_{\R^d} \mathfrak b_h(s,y)\Big(\Gamma^h(0,x,\tau_s^h,y)\Big[ \nabla P_{t-s}^\alpha\varphi(y)-\nabla P_{t-\tau_s^h}^\alpha \varphi(y+\int_{\tau_s^h}^s\mathfrak b_h(r,y) dr )\Big]\Big)dy \Big|ds \notag\\
\lesssim& \int_h^{\tau_t^h-h} \|\mathfrak b_h(s,\cdot)\|_{B_{p,q}^\beta} \|\Gamma^h(0,x,\tau_s^h,\cdot)\Big[ \nabla P_{t-s}^\alpha\varphi(\cdot)-\nabla P_{t-\tau_s^h}^\alpha \varphi(\cdot+\int_{\tau_s^h}^s\mathfrak b_h(r,\cdot) dr )\Big]\|_{B_{p',q'}^{-\beta}}ds\notag\\
\lesssim & h^\frac{\gamma-\varepsilon}{\alpha}t^{\frac{\beta+\rho}{\alpha}} \| b\|_{L^\vartheta-B_{p,q}^\beta}\Big(\int_h^{\tau_t^h-h}\Big[\Big(s^{-(\frac \rho\alpha+\frac{d}{p\alpha})}(t-s)^{{-\frac{\gamma}{\alpha}+\frac \varepsilon\alpha}-\frac 1\alpha+\frac \beta\alpha-\frac d{p\alpha}}
				+s^{-\frac d{\alpha p}}(t-s)^{{-\frac{\gamma}{\alpha}+\frac \varepsilon\alpha}-\frac 1\alpha+\frac{\beta}{\alpha}-\frac \rho\alpha-\frac d{p\alpha}}\Big)\Big]^{\vartheta'}ds\Big)^{\frac{1}{\vartheta'}}\notag\\
&\times				\|\varphi\|_{B_{p,q}^\beta}\notag\\
\lesssim & h^\frac{\gamma-\varepsilon}{\alpha}t^{\frac{{\varepsilon-\frac{d}{p}}}{\alpha}} \| b\|_{L^\vartheta-B_{p,q}^\beta} \|\varphi\|_{B_{p,q}^\beta},				\label{CTR_DELTA_4_BESOV}
\end{align}
{using \eqref{QUI_SERT_TOUT_LE_TEMPS_EN_BESOV_DRIFTED_SCHEME_TIME_SENSI} for the last but one inequality} and taking $\rho=-\beta+{\varepsilon/2}$ {to ensure that $(-\frac{\gamma}{\alpha}+\frac \varepsilon\alpha-\frac 1\alpha+\frac{\beta}{\alpha}-\frac \rho\alpha-\frac d{p\alpha})\vartheta'>-1$} for the last inequality.

Let us finally consider the Gronwall type term $\Delta^1 $. Using \eqref{dual-ineq} for the second inequality, then \eqref{PR} for the third, \eqref{CTR_BESOV_BH} and \eqref{BD_RHO_SC_BESOV} for the fourth and H\"older's inequality for the last, we obtain
\begin{align*}
|\Delta^1|=& \left|\int_h^{\tau_t^h-h} \E_{0,x} \bigg[ {\mathfrak b}_h(s,X_{\tau_s^h})\cdot \nabla P_{t-s}^\alpha \varphi(X_{\tau_s^h})
        -{\mathfrak b}_h(s,X_{\tau_s^h}^h)\cdot \nabla  P_{t-s}^\alpha \varphi(X_{\tau_s^h}^h) \bigg] d s\right|\notag\\
        \lesssim&\int_{h}^{\tau_t^h-h}\Big|\int_{\R^d }(\Gamma-\Gamma^h)(0,x,\tau_s^h,y)\mathfrak b_h(s,y)\nabla P_{t-s}^\alpha \varphi(y) dy\Big| ds\notag\\
\lesssim &  \int_h^{\tau_t^h-h} \|\mathfrak b_h(s,\cdot)\nabla P_{t-s}^\alpha \varphi\|_{B_{p,q}^\beta}  \|(\Gamma-\Gamma^h)(0,x,\tau_s^h,\cdot)\|_{B_{p',q'}^{-\beta}}   ds\notag\\
\lesssim &  \int_h^{\tau_t^h-h} \|\mathfrak b_h(s,\cdot)\|_{B_{p,q}^\beta} \|\nabla P_{t-s}^\alpha \varphi\|_{B_{\infty,\infty}^{-\beta+\varepsilon}} \|(\Gamma-\Gamma^h)(0,x,\tau_s^h,\cdot)\|_{B_{p',q'}^{-\beta}}   ds\notag\\
\lesssim &  \int_h^{\tau_t^h-h} \|b(s,\cdot)\|_{B_{p,q}^\beta}  \|\varphi\|_{B_{p,q}^{\beta}}(t-s)^{-\frac{1}{\alpha}+\frac{2\beta}{\alpha}-\frac \varepsilon \alpha-\frac{d}{p\alpha}} \|(\Gamma-\Gamma^h)(0,x,\tau_s^h,\cdot)\|_{B_{p',q'}^{-\beta}}   ds,\notag
\\\lesssim &  \|\varphi\|_{B_{p,q}^\beta}\|b\|_{L^\vartheta-B_{p,q}^\beta}\Big(\int_{0}^{t} \|(\Gamma-\Gamma^h)(0,x,\tau_s^h,\cdot)\|^{\vartheta'}_{B_{p',q'}^{-\beta}}(t-s)^{(-\frac{1}{\alpha}+\frac{2\beta}{\alpha}-\frac \varepsilon \alpha-\frac{d}{p\alpha})\vartheta'}{ds}\Big)^{\frac{1}{\vartheta'}}.\end{align*}
Combining \eqref{CTR_DELTA_5_BESOV}, \eqref{CTR_DELTA_6_BESOV}, \eqref{CTR_DELTA_2_BESOV}, \eqref{CTR_DELTA_3_BESOV} and \eqref{CTR_DELTA_4_BESOV}, we obtain \begin{equation*}
   \sum_{i=2}^6|\Delta^i|\le C_{(2-6)}^{\frac 1{\vartheta'}} h^{\frac{\gamma-\varepsilon}\alpha}\|\varphi\|_{B_{p,q}^\beta} \|b\|_{L^\vartheta-B_{p,q}^\beta} t^{\frac{\beta-\frac dp}{\alpha}}.
 \end{equation*}
 We deduce that the function $f(t)=\left(t^{\frac{\frac dp-\beta}{\alpha}}\|\Gamma(0,x,t,\cdot)-\Gamma^h(0,x,t,\cdot)\|_{B_{p',q'}^{-\beta}}\right)^{\vartheta'}$ satisfies $$\forall t\in[0,T],\;f(t)\le C_{(2-6)}h^{\frac{\gamma-\varepsilon}\alpha\times \vartheta'}+Ct^{\frac{\frac dp-\beta}{\alpha}\times \vartheta'}\int_{h}^{h\vee t}f(\tau^h_s)s^{-\frac{\frac dp-\beta}{\alpha}\times \vartheta'}(t-s)^{(-\frac{1}{\alpha}+\frac{2\beta}{\alpha}-\frac \varepsilon \alpha-\frac{d}{p\alpha})\vartheta'}ds.$$
 We are next going to check that $f$ is bounded on $(0,T]$.
Since $(\frac{1}{\alpha}-\frac{2\beta}{\alpha}-\frac \varepsilon \alpha+\frac{d}{p\alpha})\vartheta'<1$ when $\varepsilon>0$ is small and  $\frac{\frac dp-\beta}{\alpha}\times\vartheta'<1$, we conclude the proof of Theorem \ref{THM_BESOV} with
Lemma \ref{lemgron}.
 
We justify here that the quantity $\left(t^{\frac{\frac dp-\beta}{\alpha}}\|\Gamma(0,x,t,\cdot)-\Gamma^h(0,x,t,\cdot)\|_{B_{p',q'}^{-\beta}}\right)$ is indeed bounded on $(0,T] $. The finiteness readily follows from the bound
\begin{align}\label{PREAL_BD_FOR_GRNW_BESOV}
\|\Gamma(0,x,t,\cdot)\|_{B_{p',q'}^{-\beta}}+\|\Gamma^h(0,x,t,\cdot)\|_{B_{p',q'}^{-\beta}} \lesssim t^{\frac{-\frac dp+\beta}{\alpha}},
\end{align}
that we are going to prove from the definition of the Besov norm in \eqref{HEAT_CAR} and the controls \eqref{aronson-gammah}, {\eqref{holder-forward-gammah}}, \eqref{aronson-gamma} and {\eqref{holder-forward-gamma}}
 for the scheme and the diffusion respectively. Since the invoked controls are exactly the same for the scheme and the diffusion, we will only provide the proof for the diffusion.
 
 Recall $\|\Gamma(0,x,t,\cdot)\|_{B_{p',q'}^{-\beta}}=\|{\mathcal F}^{-1} (\phi) \star \Gamma(0,x,t,\cdot) \|_{L^{p'}}+\mathcal T_{p',q'}^{-\beta}(\Gamma(0,x,t,\cdot))\lesssim C+\mathcal T_{p',q'}^{-\beta}(\Gamma(0,x,t,\cdot))$,
 using the $L^{p'}-L^1 $ convolution inequality to control the non-thermic part. Let us now deal with the case $q\neq 1 $, the case $q=1$ being easier:
\begin{align}
\Big[\mathcal T_{p',q'}^{-\beta}\Big(\Gamma(0,x,t,\cdot)\Big)\Big]^{q'}=
&\int_0^{t} \frac{dv}{v}v^{(1+\frac{\beta}{\alpha}) q'}\|\partial_v p_\alpha(v,\cdot)\star \Gamma(0,x,t,\cdot)\|_{L^{p'}}^{q'}\notag\\
&+\int_{t}^T \frac{dv}{v}v^{(1+\frac{\beta}{\alpha}) q'}\|\partial_v p_\alpha(v,\cdot)\star \Gamma(0,x,t,\cdot)\|_{L^{p'}}^{q'}\notag\\
=:&\big(\mathcal T_{1,p',q'}^{-\beta}+\mathcal T_{2,p',q'}^{-\beta}\big)(\Gamma(0,x,t,\cdot)),\label{CUT_FOR_BESOV_NORM_HK_FOR_GRN}
\end{align}
where we cut the integral associated with the \textit{thermic} variable according to the current time scale of the investigated term. For the upper-cut, we write from \eqref{aronson-gamma} and \eqref{estigradplu}:
\begin{align}\label{CTR_CUT_HIGH_HK_FOR_GRN}
\mathcal T_{2,p',q'}^{-\beta}(\Gamma(0,x,t,\cdot))&\le \int_{t}^{T}\frac{dv}{v} v^{(1+\frac{\beta}{\alpha})q'} \|\partial_{v} p_\alpha(v,\cdot)\|_{L^1}^{q'}\|\Gamma(0,x,t,\cdot) \|_{L^{p'}}^{q'}
\lesssim t^{(\frac{\beta}{\alpha}-\frac d{ p\alpha})q'}.
\end{align}
For the lower cut, we only deal with the case $p\neq 1$ since the case $p=1$ is easier:
\begin{align*}
\mathcal T_{1,p',q'}^{-\beta}(\Gamma(0,x,t,\cdot))&\le \int_0^{t}\frac{dv}{v} v^{(1+\frac{\beta}{\alpha})q'} \Big(\int_{\R^d} dz\Big|\int_{\R^d} \partial_{v}p_\alpha(v,z-y)\Big[\Gamma(0,x,t,y) 
-\Gamma(0,x,t,z)\Big] dy \Big|^{p'}\Big)^{\frac{q'}{p'}}\\
&\lesssim \int_0^{t}\frac{dv}{v} v^{\frac{\beta}{\alpha}q'} \Big(\int_{\R^d} dz\Big|\int_{\R^d} |\bar p_\alpha(v,z-y)|\frac{|z-y|^\rho}{t^{\frac \rho\alpha}}\Big[\bar p_\alpha(t,y-x)+ 
\bar p_\alpha(t,z-x)\Big] dy \Big|^{p'}\Big)^{\frac{q'}{p'}}
\\&{\lesssim t^{-\frac{\rho}{\alpha} q'}\int_0^{t}\frac{dv}{v} v^{\frac{\beta}{\alpha}q'} \Big(\int_{\R^d} dz\Big|\int_{\R^d} |\bar p_\alpha(v,z-y)||z-y|^\rho\bar p_\alpha(t,y-x) dy \Big|^{p'}}\\&\phantom{\lesssim t^{-\frac{\rho}{\alpha} q'}\int_0^{t}\frac{dv}{v} v^{\frac{\beta}{\alpha}q'} \Big(\int_{\R^d} dz}
  {+|\bar p_\alpha(t,z-x)|^{p'}\Big|\int_{\R^d} |\bar p_\alpha(v,z-y)||z-y|^\rho dy \Big|^{p'}\Big)^{\frac{q'}{p'}}}\end{align*}
where we used again a cancellation argument for the inner integral in the first inequality, \tr{\eqref{GD_BOUNDS} and} \eqref{holder-forward-gamma} for the second one with $\rho\in (-\beta,\gamma-\beta )$ and a convexity inequality for the third. {Using \eqref{drift-smoothing-iso-noise} and \eqref{estigradplu} for the second inequality, we derive:}
\begin{align}
\mathcal T_{1,p',q'}^{-\beta}(\Gamma(0,x,t,\cdot))&\lesssim t^{-\frac{\rho}{\alpha} q'}\int_0^{t}\frac{dv}{v} v^{\frac{\beta}{\alpha}q'}\Big(\|(\bar p_\alpha (v,\cdot)|\cdot|^{\rho})\star \bar p_{\alpha}(t,\cdot-x) \|_{L^{p'}} +\|\bar p_\alpha(t,\cdot-x)\|_{L^{p'}}\|\bar p_\alpha (v,\cdot)|\cdot|^{\rho}\|_{L^1}\Big)^{q'}\notag\\
&\lesssim t^{-(\frac{\rho}{\alpha}+\frac d{\alpha p} )q'}\int_0^{t}\frac{dv}{v} v^{\frac{(\beta+\rho)}{\alpha}q'}\lesssim t^{(\frac{\beta}{\alpha}-\frac d{\alpha p} )q'}.\label{esti:T1gam}
 \end{align}
 This last inequality and \eqref{CTR_CUT_HIGH_HK_FOR_GRN}, together with \eqref{CUT_FOR_BESOV_NORM_HK_FOR_GRN} and the former control for the non-thermic part of the norm give \eqref{holder-forward-gammah} for the diffusion. Repeating the previous arguments invoking \eqref{aronson-gammah} and \eqref{holder-forward-gammah} instead of \eqref{aronson-gamma}, \eqref{holder-forward-gamma} provides the result for the scheme.

\subsection{Technical results for a Besov drift} \label{SEC_TEC_BESOV}

\subsubsection{Proof of Proposition \ref{HK_W_BESOV}}\label{prHK_W_BESOV}
Let us start with \eqref{BD_INFTY_SC_BESOV}. From \eqref{EMBEDDING}, \eqref{YOUNG} and {\eqref{SING_STABLE_HK}}, one gets:
\begin{align*}
\|\nabla P_{t-s}^\alpha \varphi\|_{L^\infty}&\lesssim \|\nabla P_{t-s}^\alpha \varphi\|_{B_{\infty,1}^0}\lesssim \|\nabla p_\alpha(t-s,\cdot)\|_{B_{p',q'}^{-\beta}}\|\varphi\|_{B_{p,q}^\beta}\\
&\lesssim (t-s)^{-\frac 1\alpha+\frac{\beta}\alpha -\frac d{p\alpha}}\|\varphi\|_{B_{p,q}^\beta},
\end{align*}
{which} gives \eqref{BD_INFTY_SC_BESOV}. For the  claim \eqref{BD_RHO_SC_BESOV}, fix $\rho\in (-\beta,-\beta+\gamma) $. Write then  from \eqref{YOUNG} and \eqref{SING_STABLE_HK}:
\begin{align*}
\|{\partial_t^\theta\partial^{\mathbf a}} P_{t-s}^\alpha \varphi\|_{B_{\infty,\infty}^\rho}&\lesssim \|{\partial_t^\theta\partial^{\mathbf a}} p_\alpha(t-s,\cdot)\|_{B_{p',q'}^{-\beta+\rho}}\|\varphi\|_{B_{p,q}^\beta}
\lesssim (t-s)^{-{\theta+\frac{|\mathbf a |}\alpha}+\frac{\beta}\alpha-\frac \rho\alpha -\frac d{p\alpha}}\|\varphi\|_{B_{p,q}^\beta}.
\end{align*}
Let us now turn to the corresponding estimates involving the drift starting with \eqref{BD_INFTY_SC_BESOV_SCHEME}. 
Write from \eqref{EMBEDDING} and \eqref{YOUNG}:
\begin{align}
\|{\partial^{\mathbf a}}P_{t-\tau_s^h}^\alpha \varphi\big(\cdot+\int_{\tau_s^h}^s \mathfrak b_h(r,\cdot)dr\big)\|_{L^\infty}&\lesssim \|{\partial^{\mathbf a}} P_{t-\tau_s^h}^\alpha \varphi\big(\cdot+\int_{\tau_s^h}^s \mathfrak b_h(r,\cdot)dr\big)\|_{B_{\infty,1}^0}\notag\\
&\lesssim \|{\partial^{\mathbf a}} p_\alpha(t-\tau_s^h,\cdot+\int_{\tau_s^h}^s \mathfrak b_h(r,\cdot)dr)\|_{B_{p',q'}^{-\beta}}\|\varphi\|_{B_{p,q}^\beta}.\label{527_INIT}
\end{align}
{\color{black}{Let us now control the term $\|g_{\alpha,\mathfrak b_h}^{h,s,t}\|_{B_{p',q'}^{-\beta}}$ where $g_{\alpha,\mathfrak b_h}^{h,s,t}:={\partial^{\mathbf a}} p_\alpha(t-\tau_s^h,\cdot+\int_{\tau_s^h}^s \mathfrak b_h(r,\cdot)dr)$. Write:

\begin{align}
\|g_{\alpha,\mathfrak b_h}^{h,s,t}\|_{B_{p',q'}^{-\beta}}=&\|{\cal F}^{-1}(\phi) \star
g_{\alpha,\mathfrak b_h}^{h,s,t}
\|_{L^{p'}}+
\mathcal T_{p',q'}^{{-\beta}}(g_{\alpha,\mathfrak b_h}^{h,s,t})
\lesssim
\|{\cal F}^{-1}(\phi)\|_{L^{p'}}\|g_{\alpha,\mathfrak b_h}^{h,s,t}
\|_{L^{1}}+\mathcal T_{p',q'}^{{-\beta}}(g_{\alpha,\mathfrak b_h}^{h,s,t}).\label{TO_BE_RECALLED_LATER}
\end{align}
From \eqref{GD_BOUNDS} and \eqref{CTR_PONCTUEL_BH_INT}, for all $z\in \R^d $,
\begin{align}
|{\partial^{\mathbf a}} p_\alpha(t-\tau_s^h,z+\int_{\tau_s^h}^s \mathfrak b_h(r,\cdot)dr)|\le &\frac{C}{(t-\tau_s^h)^{\frac{d+|\mathbf a|}\alpha}}\frac{1}{\Big(2+\frac{|z+\int_{\tau_s^h}^s \mathfrak b_h(r,\cdot)dr|}{(t-\tau_s^h)^{\frac 1\alpha}}\Big)^{d+\alpha+|{\mathbf a}|}}\notag\\\le &\frac{C}{(t-\tau_s^h)^{\frac{d+|\mathbf a|}\alpha}}\frac{1}{\Big(2+\frac{|z+\int_{\tau_s^h}^s \mathfrak b_h(r,\cdot)dr|}{(t-\tau_s^h)^{\frac 1\alpha}}\Big)^{d+\alpha}}\notag\\
\le &\frac{C}{(t-\tau_s^h)^{\frac{d+|\mathbf a|}\alpha}}\frac{1}{\Big(2-\frac{C(s-\tau_s^h)^{\frac{\gamma}{\alpha} +\frac{1-\beta}\alpha}\|b\|_{L^\vartheta-B_{p,q}^\beta}}{(t-\tau_s^h)^{\frac 1\alpha}} +\frac{|z|}{(t-\tau_s^h)^{\frac 1\alpha}}\Big)^{d+\alpha}}\notag\\
\le &\frac{C}{(t-\tau_s^h)^{\frac{d+|\mathbf a|}\alpha}}\frac{1}{\Big(2-C(s-\tau_s^h)^{\frac{\gamma-\beta}{\alpha}}\|b\|_{L^\vartheta-B_{p,q}^\beta} +\frac{|z|}{(t-\tau_s^h)^{\frac 1\alpha}}\Big)^{d+\alpha}}\notag\\
\le &\frac{C}{(t-\tau_s^h)^{\frac{d+|\mathbf a|}\alpha}}\frac{1}{\Big(1+\frac{|z|}{(t-\tau_s^h)^{\frac 1\alpha}}\Big)^{d+\alpha}}\lesssim \frac{\bar p_\alpha(t-\tau_s^h,z)}{(t-\tau_s^h)^{\frac{|\mathbf a|}\alpha}}.
\label{DRIFT_CAN_BE_NEGLECTED_IN_BESOV}
\end{align}
Hence,
\begin{align}
\label{CTR_LPP_527}
\|g_{\alpha,\mathfrak b_h}^{h,s,t}\|_{L^1}\lesssim (t-\tau_s^h)^{-\frac{|\mathbf a|}\alpha}.
\end{align}
On the other hand, only dealing with the case $q\neq 1$ since the case $q=1$ is easier, 
we get
:
\begin{align}
\Big[\mathcal T_{p',q'}^{-\beta}\Big(g_{\alpha,\mathfrak b_h}^{h,s,t}
\Big)\Big]^{q'}=&\int_0^T \frac{dv}{v}v^{(1+\frac{\beta}{\alpha}) q'}\|\partial_v p_\alpha(v,\cdot)\star\partial^{\mathbf a} p_\alpha({t-\tau_s^h},\cdot+\int_{\tau_s^h}^s \mathfrak b_h(r,\cdot)dr)\|_{L^{p'}}^{q'}\notag\\
=&\int_0^{t-\tau_s^h} \frac{dv}{v}v^{(1+\frac{\beta}{\alpha}) q'}\|\partial_v p_\alpha(v,\cdot)\star \partial^{\mathbf a} p_\alpha({t-\tau_s^h},\cdot+\int_{\tau_s^h}^s \mathfrak b_h(r,\cdot)dr)\|_{L^{p'}}^{q'}\notag\\
&+\int_{t-\tau_s^h}^T \frac{dv}{v}v^{(1+\frac{\beta}{\alpha}) q'}\|\partial_v p_\alpha(v,\cdot)\star \partial^{\mathbf a} p_\alpha({t-\tau_s^h},\cdot+\int_{\tau_s^h}^s \mathfrak b_h(r,\cdot)dr)\|_{L^{p'}}^{q'}\notag\\
=:&\big(\mathcal T_{1,p',q'}^{-\beta}+\mathcal T_{2,p',q'}^{-\beta}\big)(g_{\alpha,\mathfrak b_h}^{h,s,t}),\label{CUT_FOR_BESOV_NORM_0}
\end{align}
where we cut the integral associated with the \textit{thermic} variable according to the current time scale of the investigated term. For the upper-cut, we write:
\begin{align}\label{CTR_CUT_HIGH_0}
\mathcal T_{2,p',q'}^{-\beta}(g_{\alpha,\mathfrak b_h}^{h,s,t})&\le \int_{t-\tau_s^h}^{T}\frac{dv}{v} v^{(1+\frac{\beta}{\alpha})q'} \|\partial_{v} p_\alpha(v,\cdot)\|_{L^1}^{q'}\|\partial^{\mathbf a} p_\alpha({t-\tau_s^h},\cdot+\int_{\tau_s^h}^s \mathfrak b_h(r,\cdot)dr) \|_{L^{p'}}^{q'}\notag\\
&\lesssim \Big( (t-\tau_s^h)^{\frac{\beta}{\alpha}-(\frac{|\mathbf a|}\alpha+\frac d{ p\alpha}) } \Big)^{q'},
\end{align}
where we used as well \eqref{DRIFT_CAN_BE_NEGLECTED_IN_BESOV} and \eqref{estigradplu} for the last inequality. For the lower cut, we only deal with the case $p\ne 1$ since the case $p=1$ is easier:
\begin{align*}
\mathcal T_{1,p',q'}^{-\beta}(g_{\alpha,\mathfrak b_h}^{h,s,t})&\le \int_0^{t-\tau_s^h}\frac{dv}{v} v^{(1+\frac{\beta}{\alpha})q'} \Big(\int_{\R^d} dz\Big|\int_{\R^d} \partial_{v}p_\alpha(v,z-y)\Big[\partial^{\mathbf a} p_\alpha({t-\tau_s^h},y+\int_{\tau_s^h}^s \mathfrak b_h(r,y)dr) \notag\\
&-\partial^{\mathbf a} p_\alpha({t-\tau_s^h},z+\int_{\tau_s^h}^s \mathfrak b_h(r,z)dr)\Big] dy \Big|^{p'}\Big)^{\frac{q'}{p'}},
\end{align*}
where we used again a cancellation argument for the inner integral.
Observe now that:
\begin{trivlist}
\item[-] If $|y-z|>(t-\tau_s^h)^{\frac 1\alpha} $ (off diagonal regime for the heat kernel at hand), then from \eqref{DRIFT_CAN_BE_NEGLECTED_IN_BESOV}, for $\rho>-\beta $,
\begin{align*}
&\Big|\partial^{\mathbf a} p_\alpha({t-\tau_s^h},y+\int_{\tau_s^h}^s \mathfrak b_h(r,y)dr) -\partial^{\mathbf a} p_\alpha({t-\tau_s^h},z+\int_{\tau_s^h}^s \mathfrak b_h(r,z)dr)\Big|\\
\le & \Big|\partial^{\mathbf a} p_{t-\tau_s^h}^\alpha(y+\int_{\tau_s^h}^s \mathfrak b_h(r,y)dr)\Big|+\Big|\partial^{\mathbf a} p_{t-\tau_s^hs}^\alpha(z+\int_{\tau_s^h}^s \mathfrak b_h(r,z)dr)\Big|\lesssim \frac{|z-y|^\rho}{(t-\tau_s^h)^{\frac{|\mathbf a |}\alpha+\frac{\rho}{\alpha}}}\Big(\bar p_\alpha(t-\tau_s^h,z)+\bar p_\alpha(t-\tau_s^h,y) \Big).
\end{align*}
\item[-] If now $|y-z|\le (t-\tau_s^h)^{\frac 1\alpha} $ (diagonal regime for the heat kernel at hand), we get from \eqref{GD_BOUNDS} and similarly to the derivation of \eqref{DRIFT_CAN_BE_NEGLECTED_IN_BESOV} {with \eqref{CTR_PONCTUEL_BH_INT} replaced by \eqref{CTR_PONCTUEL_BH_INT_HOLDER} that for $\rho\in (-\beta,-\beta+\gamma)$,}
\begin{align*}
&\Big|\partial^{\mathbf a} p_\alpha({t-\tau_s^h},y+\int_{\tau_s^h}^s \mathfrak b_h(r,y)dr) -\partial^{\mathbf a} p_\alpha({t-\tau_s^h},z+\int_{\tau_s^h}^s \mathfrak b_h(r,z)dr)\Big|\\
\le &\int_0^1 d\lambda \Big|\nabla \partial^{\mathbf a} p_\alpha({t-\tau_s^h},z+\int_{\tau_s^h}^s \mathfrak b_h(r,z)dr+\lambda \Big(y-z+\int_{\tau_s^h}^s (\mathfrak b_h(r,y)-\mathfrak b_h(r,z))dr\Big)\Big)\Big|\\
&\times   \Big(|y-z|+\Big|\int_{\tau_s^h}^s (\mathfrak b_h(r,y)-\mathfrak b_h(r,z))dr\Big|\Big)\\
\lesssim & \frac{\bar p_\alpha(t-\tau_s^h,z)}{(t-\tau_s^h)^{\frac{ |\mathbf a |+1}\alpha}}\Big(|y-z|+C|y-z|^\rho (s-\tau_s^h)^{\frac{\gamma}{\alpha}+\frac{1-\beta-\rho}{\alpha}}\|b\|_{L^\vartheta-B_{p,q}^\beta}\Big)\\
\lesssim & \frac{\bar p_\alpha(t-\tau_s^h,z)}{(t-\tau_s^h)^{\frac{|\mathbf a|}\alpha+\frac \rho\alpha }}|y-z|^\rho
\end{align*}
{recalling for the last inequality  that, since $\rho<\gamma-\beta<1$,}
\begin{align}\label{USEFUL_FOR_DRIFT_EFFECT}
\frac{(s-\tau_s^h)^{\frac{\gamma}{\alpha}+\frac{1-\beta-\rho}{\alpha}}}{(t-\tau_s^h)^{\frac{|\mathbf a|+1}\alpha}}=\frac{(s-\tau_s^h)^{\frac{\gamma-\beta}{\alpha}+\frac{1-\rho}{\alpha}}}{(t-\tau_s^h)^{\frac {|\mathbf a|+\rho}\alpha+\frac{1-\rho}{\alpha}}}
\le \frac{(s-\tau_s^h)^{\frac{\gamma-\beta}{\alpha}}}{(t-\tau_s^h)^{\frac {|\mathbf a |+\rho}\alpha} }.
\end{align}
\end{trivlist}
{Reasoning like in the above derivation of \eqref{esti:T1gam} and using \eqref{GD_BOUNDS}, we deduce that for $\rho\in (-\beta,{-\beta+\gamma}) $:}
\begin{align}\label{CTR_CUT_LOW_0}
\mathcal T_{1,p',q'}^{-\beta}(g_{\alpha,\mathfrak b_h}^{h,s,t})&\le \int_0^{t-\tau_s^h}\frac{dv}{v} v^{\frac{\beta}{\alpha}q'} \Big(\int_{\R^d} dz\Big[\int_{\R^d} \bar p_\alpha(v,z-y)\frac{|z-y|^\rho}{(t-\tau_s^h)^{\frac{|\mathbf a|+\rho}{\alpha}}}\Big[\bar p_\alpha({t-\tau_s^h},y)+\bar p_\alpha({t-\tau_s^h},z)\Big] dy \Big]^{p'}\Big)^{\frac{q'}{p'}}\notag\\
&\le \int_0^{t-\tau_s^h}\frac{dv}{v} v^{\frac{\beta+\rho}{\alpha}q'} (t-\tau_s^h)^{-(\frac{|\mathbf a|+\rho+\frac dp}\alpha)q'}\notag\\
&\le \Big((t-\tau_s^h)^{\frac{\beta+\rho}{\alpha}-\frac{|\mathbf a|+\rho+\frac dp}{\alpha}}\Big)^{q'}\lesssim \Big((t-\tau_s^h)^{\frac{\beta}{\alpha}-(\frac{|\mathbf a |}\alpha+\frac d{p\alpha})}\Big)^{q'}.
\end{align}
From \eqref{TO_BE_RECALLED_LATER}, \eqref{CTR_LPP_527}, \eqref{CUT_FOR_BESOV_NORM_0}, \eqref{CTR_CUT_HIGH_0} and \eqref{CTR_CUT_LOW_0}, we have established that:
\begin{align}
\|g_{\alpha,\mathfrak b_h}^{h,s,t}\|_{B_{p',q'}^{-\beta}}=\|\partial^{\mathbf a } p_\alpha(t-\tau_s^h,\cdot+\int_{\tau_s^h}^s \mathfrak b_h(r,\cdot)dr)\|_{B_{p',q'}^{-\beta}}\lesssim (t-\tau_s^h)^{\frac{\beta}{\alpha}-(\frac{|\mathbf a|}\alpha+\frac d{p\alpha})}. \label{NORME_BESOV_POUR_GERER_DRIFT_SHIFTE}
\end{align}
We thus derive from \eqref{527_INIT}: 
\begin{align*}
\|\partial^{\mathbf a } P^\alpha_{t-\tau_s^h}\varphi\big(\cdot+\int_{\tau_s^h}^s \mathfrak b_h(r,\cdot)dr\big)\|_{L^\infty}&\lesssim (t-\tau_s^h)^{-\frac {|\mathbf a|}\alpha+\frac{\beta}\alpha -\frac d{p\alpha}}\|\varphi\|_{B_{p,q}^\beta},
\end{align*}
which is precisely \eqref{BD_INFTY_SC_BESOV_SCHEME}.

For \eqref{BD_RHO_SC_BESOV_SCHEME}, we remark that the previous arguments yielding to \eqref{NORME_BESOV_POUR_GERER_DRIFT_SHIFTE}  can be reproduced replacing $-\beta $ by $-\beta+\rho $ and $\rho $ by $\tilde \rho\in (-\beta+\rho,1) $: for $\rho\in (-\beta,-\beta+\gamma) $ there is always a feasible $\tilde \rho $ since $-2\beta+\gamma<1$. Hence, \eqref{NORME_BESOV_POUR_GERER_DRIFT_SHIFTE} remains valid replacing $-\beta $ by $-\beta+\rho $.}  
\color{black}
This yields:
\begin{align*}
\|{\partial^{\mathbf a}} P_{t-\tau_s^h}^\alpha \varphi(\cdot+\int_{\tau_s^h}^s \mathfrak b_h(r,\cdot)dr )\|_{B_{\infty,\infty}^{\rho}}&\lesssim \|{\partial^{\mathbf a}} p_\alpha(t-\tau_s^h,\cdot+\int_{\tau_s^h}^s \mathfrak b_h(r,\cdot)dr)\|_{B_{p',q'}^{-\beta+\rho}}\|\varphi\|_{B_{p,q}^\beta}\notag\\
&\lesssim(t-\tau_s^h)^{-\frac{{|\mathbf a |}}\alpha+\frac \beta \alpha-\frac \rho\alpha-\frac d{p\alpha}}\|\varphi\|_{B_{p,q}^\beta}.%\label{}.
\end{align*}

Let us now turn to the proof of the estimates involving the time sensitivities, {for which we assume that $s\in [h,\tau_t^h-h] $}. Let us start with \eqref{BD_INFTY_SC_BESOV_TIME_SENSI}. Write for all $z\in \R^d $:
\begin{align*}
\Big|\nabla P_{t-s}^\alpha \varphi(z)-\nabla P_{t-\tau_s^h}^\alpha\varphi(z)\Big|=\Big|\int_0^1 d\lambda  \partial_{r_\lambda}\nabla P_{r_\lambda}^\alpha\varphi(z) |_{r_\lambda=t-(\tau_s^h+\lambda(s-\tau_s^h)) }\Big| (s-\tau_s^h)
\end{align*}
From \eqref{EMBEDDING}, \eqref{YOUNG} and \eqref{SING_STABLE_HK},  one derives:
\begin{align*}
\Big\|\nabla P_{t-s}^\alpha \varphi-\nabla P_{t-\tau_s^h}^\alpha\varphi\Big\|_{L^\infty}&\lesssim \Big\|\nabla P_{t-s}^\alpha \varphi-\nabla P_{t-\tau_s^h}^\alpha\varphi\Big\|_{B_{\infty,1}^0}\lesssim (s-\tau_s^h)\int_0^1 d\lambda \Big(  \|\partial_{r_\lambda}\nabla P_{r_\lambda}^\alpha\varphi\|_{B_{\infty,1}^0} \Big|_{r_\lambda=t-(\tau_s^h+\lambda(s-\tau_s^h)) }\Big) \\
&\lesssim  (s-\tau_s^h)\int_0^1 d\lambda   \Big(\|\partial_{r_\lambda}\nabla P_{r_\lambda}^\alpha\|_{B_{p',q'}^{-\beta}} \Big|_{r_\lambda=t-(\tau_s^h+\lambda(s-\tau_s^h)) }\Big) \|\varphi\|_{B_{p,q}^{\beta}}\\
&\lesssim  (s-\tau_s^h)\|\varphi\|_{B_{p,q}^{\beta}}\int_0^1 d\lambda r_\lambda^{-(1+\frac 1\alpha-\frac \beta\alpha +\frac d{p\alpha})}\le \frac{(s-\tau_s^h)}{t-s}\|\varphi\|_{B_{p,q}^{\beta}}(t-s)^{-\frac 1\alpha+\frac \beta\alpha -\frac d{p\alpha}}\\
&\lesssim \left(\frac{(s-\tau_s^h)}{t-s}\right)^\theta\|\varphi\|_{B_{p,q}^{\beta}}(t-s)^{-\frac 1\alpha+\frac \beta\alpha -\frac d{p\alpha}},
\end{align*}
for any $\theta\in [0,1] $,   recalling for the last inquality that, since $s\in [h,\tau_t^h-h],\ t-s\ge h\ge s-\tau_s^h $.

Equation \eqref{BD_RHO_SC_BESOV_TIME_SENSI} is derived similarly,
\begin{align*}
\Big\|\nabla P_{t-s}^\alpha \varphi-\nabla P_{t-\tau_s^h}^\alpha\varphi\Big\|_{B_{\infty,\infty}^\rho}&\lesssim (s-\tau_s^h)\int_0^1 d\lambda \Big(  \|\partial_{r_\lambda}\nabla P_{r_\lambda}^\alpha\varphi\|_{B_{\infty,\infty}^\rho} \Big|_{r_\lambda=t-(\tau_s^h+\lambda(s-\tau_s^h)) }\Big) \\
&\lesssim  (s-\tau_s^h)\int_0^1 d\lambda  \Big( \|\partial_{r_\lambda}\nabla P_{r_\lambda}^\alpha\|_{B_{p',q'}^{-\beta+\rho}} \Big|_{r_\lambda=t-(\tau_s^h+\lambda(s-\tau_s^h)) }\Big) \|\varphi\|_{B_{p,q}^{\beta}}\\
&\lesssim  (s-\tau_s^h)\|\varphi\|_{B_{p,q}^{\beta}}\int_0^1 d\lambda r_\lambda^{-(1+\frac 1\alpha-\frac \beta\alpha +\frac \rho\alpha+\frac d{p\alpha})}\le \frac{(s-\tau_s^h)}{t-s}\|\varphi\|_{B_{p,q}^{\beta}}(t-s)^{-\frac 1\alpha+\frac \beta\alpha -\frac \rho\alpha-\frac d{p\alpha}}\\
&\lesssim \left(\frac{(s-\tau_s^h)}{t-s}\right)^\theta\|\varphi\|_{B_{p,q}^{\beta}}(t-s)^{-\frac 1\alpha+\frac \beta\alpha-\frac{\rho}{\alpha} -\frac d{p\alpha}}.
\end{align*}
For \eqref{BD_INFTY_SC_BESOV_SPACE_SENSI}, we first write a spatial expansion. For all $y\in \R^d$:
{$$\nabla P_{t-s}^\alpha \varphi(y+z)-\nabla P_{t-\tau_s^h}^\alpha\varphi(y)=\int_0^1 d\lambda \nabla^2 P_{t-s}^\alpha \varphi(y+\lambda z) z. $$
Thus, from the same previous arguments (\eqref{EMBEDDING}, \eqref{YOUNG}, \eqref{SING_STABLE_HK}),
\begin{align*}
\|\nabla P_{t-\tau_s^h}^\alpha\varphi(\cdot+z) -\nabla P_{t-\tau_s^h}^\alpha \varphi(\cdot)\|_{L^\infty}&\lesssim \int_0^1 d\lambda \|\nabla^2  P_{t-\tau_s^h}^\alpha\varphi(\cdot+\lambda z) \|_{B_{\infty,1}^0}|z|\\
&\lesssim \int_0^1 d\lambda\|\nabla^2  p_\alpha(t-\tau_s^h,\cdot+\lambda z) \|_{B_{p',q'}^{-\beta}}\|\varphi\|_{B_{p,q}^{\beta}}|z|\\
&\lesssim(t-s)^{-\frac 1\alpha+\frac \beta \alpha-\frac d{p\alpha}}\|\varphi\|_{B_{p,q}^\beta}\frac{|z|}{(t-s)^{\frac 1\alpha}}.
 \end{align*}}
  The proof of Proposition \ref{HK_W_BESOV} is complete.

\subsubsection{{Proof of Proposition \ref{PROP_SENSI_WITH_SHIFTED_INDEX}}}\label{prpswsi}
{We are here concerned with the proofs of \eqref{QUI_SERT_TOUT_LE_TEMPS_EN_BESOV} and \eqref{QUI_SERT_TOUT_LE_TEMPS_EN_BESOV_FOR_DELTA_3}.}
We again proceed through the thermic characterization.  For notational simplicity we set for this paragraph:
$$\mathfrak q_x^{s,t} (\cdot):=\Gamma(0,x,s,\cdot) \nabla P_{t-s}^{\alpha}\varphi(\cdot).$$
For the non-thermic part one has from \eqref{HEAT_CAR}:
\begin{align}
\label{NON_THERMIC_PART}
\|{\mathcal F^{-1}} (\phi) \star\mathfrak q_x^{s,t} 
\|_{L^{p'}}\le& \|\ |{\mathcal F^{-1}(\phi)}|\star\Gamma(0,x,s,\cdot) \ \|_{L^{p'}}\|\nabla P_{t-s}^{\alpha}\varphi\|_{L^\infty}\notag\\
\lesssim & (t-s)^{-\frac{1-\beta+\frac dp}{\alpha}}\|\varphi\|_{B_{p,q}^{\beta}},
\end{align}
using an $L^{p'}-L^1$ convolution inequality (recall that ${\mathcal F^{-1} (\phi)} \in \mathcal S $) and \eqref{BD_INFTY_SC_BESOV} for the last inequality.

Let us now focus on the thermic part which gives the most singular contributions. Only dealing with the case $q\neq 1$ since the case $q=1$ is easier,  we have for $\delta\in \{0,1\} $ {using that $-\beta+\gamma<1$}:
\begin{align}
\Big[\mathcal T_{p',q'}^{-\beta+\delta(\gamma-\varepsilon)}\Big(\mathfrak q_x^{s,t}
\Big)\Big]^{q'}=&\int_0^t \frac{dv}{v}v^{(1+\frac{\beta-\delta(\gamma-\varepsilon)}{\alpha}) q'}\|\partial_v p_\alpha(v,\cdot)\star(\Gamma(0,x,s,\cdot) \nabla P_{t-s}^{\alpha}\varphi)\|_{L^{p'}}^{q'}\notag\\
&+\int_t^T \frac{dv}{v}v^{(1+\frac{\beta-\delta(\gamma-\varepsilon)}{\alpha}) q'}\|\partial_v p_\alpha(v,\cdot)\star(\Gamma(0,x,s,\cdot) \nabla P_{t-s}^{\alpha}\varphi)\|_{L^{p'}}^{q'}\notag\\
=:&\big(\mathcal T_{1,p',q'}^{-\beta+\delta(\gamma-\varepsilon)}+\mathcal T_{2,p',q'}^{-\beta+\delta(\gamma-\varepsilon)}\big)(\mathfrak q_x^{s,t}),\label{CUT_FOR_BESOV_NORM}
\end{align}
using as usual the lower and upper cut w.r.t. the corresponding time scale. Let us start with the upper cut for which:
\begin{align}\label{CTR_CUT_HIGH}
\mathcal T_{2,p',q'}^{-\beta+\delta(\gamma-\varepsilon)}(\mathfrak q_x^{s,t})&\le \int_{t}^{T}\frac{dv}{v} v^{(1+\frac{\beta-\delta(\gamma-\varepsilon)}{\alpha})q'} \|\partial_{v} p_\alpha(v,\cdot)\|_{L^1}^{q'}\|\Gamma(0,x,s,\cdot)\|_{L^{p'}}^{q'}\|\nabla P_{t-s}^\alpha \varphi\|_{L^\infty}^{q'}\notag\\
&\lesssim \Big(t^{\frac{\beta-\delta(\gamma-\varepsilon)} \alpha } s^{-\frac d{p \alpha}} (t-s)^{-\frac 1\alpha+\frac{\beta}{\alpha}-\frac d{ p\alpha} }\|\varphi\|_{B_{p,q}^\beta} \Big)^{q'},
\end{align}
{where we used \eqref{aronson-gamma}, \eqref{estigradplu} and \eqref{BD_INFTY_SC_BESOV} for the last inequality}.

Let us now deal with the lower cut. We only deal with the case $p\ne 1$ since the case $p=1$ is easier. {Combining \eqref{BD_INFTY_SC_BESOV_SPACE_SENSI} and \eqref{BD_INFTY_SC_BESOV}, we obtain that for $\rho\in[0,1]$,
$$\|\nabla P_{t-{\tau_s^h}}^\alpha \varphi{(\cdot+w)}-\nabla P_{t-\tau_s^h}^\alpha\varphi{(\cdot)} \|_{L^\infty}\lesssim (t-s)^{-\frac 1\alpha+\frac \beta \alpha-\frac \rho\alpha -\frac d{p\alpha}}|w|^\rho\|\varphi\|_{B_{p,q}^\beta}.$$ From usual cancellation arguments and also using \eqref{GD_BOUNDS}, \eqref{holder-forward-gamma} and \eqref{BD_INFTY_SC_BESOV}, we deduce that:
\begin{align}
&\mathcal T_{1,p',q'}^{-\beta+\delta(\gamma-\varepsilon)}(\mathfrak q_x^{s,t})\notag\\
  \le& \int_0^t \frac{dv}{v}v^{(1+\frac{\beta-\delta(\gamma-\varepsilon)}{\alpha}) q'}\Big(\int_{\R^d} dz \Big|\int_{\R^d}\partial_v p_\alpha(v,z-y)[(\Gamma(0,x,s,y)-\Gamma(0,x,s,z)) \nabla P_{t-s}^\alpha \varphi(y)\\
  &\phantom{int_0^t \frac{dv}{v}v^{(1+\frac{\beta-\delta(\gamma-\varepsilon)}{\alpha}) q'}\Big(\int_{\R^d} dz \Big|\int_{\R^d}\partial_v p_\alpha(v,z-y)[}+\Gamma(0,x,s,z)(\nabla P_{t-s}^\alpha \varphi(y)-\nabla P_{t-s}^\alpha \varphi(z))]dy \Big|^{p'}\Big)^{\frac{q'}{p'}}\notag\\
\le& \int_0^t \frac{dv}{v}v^{\frac{\beta-\delta(\gamma-\varepsilon)}{\alpha}q'}\Big(\int_{\R^d} dz \Big|\int_{\R^d}\bar p_\alpha(v,z-y) |z-y|^{\rho}\Big(s^{-\frac\rho\alpha}(\bar p_\alpha(s,y-x)+\bar p_\alpha(s,z-x))(t-s)^{-\frac 1\alpha+\frac \beta\alpha-\frac d{p\alpha}}\notag\\
&+\bar p_\alpha(s,z-x)(t-s)^{-\frac 1\alpha+\frac{\beta}{\alpha}-\frac \rho\alpha-\frac d{p\alpha}}\Big) dy\Big|^{p'}\Big)^{\frac{q'}{p'}}\|\varphi\|_{B_{p,q}^{\beta}}^{q'}, \rho\in (-\beta,-\beta+\gamma).\notag
\end{align}}
We eventually derive:
\begin{align*}
\mathcal T_{1,p',q'}^{-\beta+\delta(\gamma-\varepsilon)}(\mathfrak q_x^{s,t})\lesssim& \int_0^t \frac{dv}{v}v^{(\frac{\beta-\delta(\gamma-\varepsilon)}{\alpha}+\frac\rho\alpha) q'}\Big( s^{-(\frac \rho\alpha+\frac{d}{p\alpha})}(t-s)^{-\frac 1\alpha+\frac \beta\alpha-\frac d{p\alpha}}+s^{-\frac d{\alpha p}}(t-s)^{-\frac 1\alpha+\frac{\beta}{\alpha}-\frac \rho\alpha-\frac d{p\alpha}}\Big)^{q'}\|\varphi\|_{B_{p,q}^{\beta}}^{q'}\\
\lesssim& \Big[t^{\frac{\beta-\delta(\gamma-\varepsilon)+\rho}{\alpha}}\Big(s^{-(\frac \rho\alpha+\frac{d}{p\alpha})}(t-s)^{-\frac 1\alpha+\frac \beta\alpha-\frac d{p\alpha}}+s^{-\frac d{\alpha p}}(t-s)^{-\frac 1\alpha+\frac{\beta}{\alpha}-\frac \rho\alpha-\frac d{p\alpha}}\Big)\Big]^{q'}\|\varphi\|_{B_{p,q}^{\beta}}^{q'}.
\end{align*}
Plugging the above control and \eqref{CTR_CUT_HIGH} into \eqref{CUT_FOR_BESOV_NORM} and recalling \eqref{NON_THERMIC_PART}, we derive the statement \eqref{QUI_SERT_TOUT_LE_TEMPS_EN_BESOV} taking {$\delta=0 $} and any $\rho\in (-\beta,-\beta+\gamma) $ and \eqref{QUI_SERT_TOUT_LE_TEMPS_EN_BESOV_FOR_DELTA_3} taking {$\delta=1 $}, $\rho=-\beta+\gamma -\frac{\varepsilon}2$.   
\subsubsection{Proof of \eqref{QUI_SERT_TOUT_LE_TEMPS_EN_BESOV_DRIFTED_SCHEME}}\label{prqst}	
For this paragraph, we set for notational simplicity:
$$ \mathfrak q_x^{h,s,t}(\cdot):=\Gamma^h(0,x,\tau_s^h,\cdot)\nabla P^\alpha_{t-\tau_s^h}\varphi(\cdot+\int_{\tau_s^h}^s \mathfrak b_h(r,\cdot)dr ).$$
Let us again first focus on the non-thermic part. Write:
\begin{align}
\|{\mathcal F}^{-1}(\phi)\star\mathfrak q_x^{h,s,t} 
\|_{L^{p'}}\lesssim & \|\ |{\mathcal F^{-1}(\phi)}|\star\Gamma^h(0,x,\tau_s^h,\cdot) \ \|_{L^{p'}}\|\nabla P_{t-\tau_s^h}^{\alpha}\varphi(\cdot+\int_{\tau_s^h}^s \mathfrak b_h(r,\cdot)dr )\|_{L^\infty}\notag\\
\lesssim & (t-s)^{-\frac{1-\beta+\frac dp}{\alpha}}\|\varphi\|_{B_{p,q}^{\beta}},
\end{align}
using \eqref{BD_INFTY_SC_BESOV_SCHEME} for the last inequality. Let us now consider the same 
terms {as} those that appeared in \eqref{CUT_FOR_BESOV_NORM}, associated with the lower and upper cut in time for the thermic variable, but replacing therein $\mathfrak q_x^{s,t} $ by $\mathfrak q_x^{h,s,t}$. Namely, for the upper cut, which is the most direct to handle, we have:
\begin{align}\label{CTR_CUT_HIGH_SCHEME}
\mathcal T_{2,p',q'}^{-\beta}(\mathfrak q_x^{h,s,t})&\le \int_{t}^{T}\frac{dv}{v} v^{(1+\frac{\beta}{\alpha})q'} \|\partial_{v} p_\alpha(v,\cdot)\|_{L^1}^{q'}\|\Gamma^h(0,x,\tau_s^h,\cdot)\|_{L^{p'}}^{q'}\|\nabla P_{t-\tau_s^h}^\alpha \varphi(\cdot+\int_{\tau_s^h}^s \mathfrak b_h(r,\cdot) dr)\|_{L^\infty}^{q'}\notag\\
&\lesssim \Big(t^{\frac \beta \alpha } s^{-\frac d{p \alpha}} (t-s)^{-\frac 1\alpha+\frac{\beta}{\alpha}-\frac d{ p\alpha} }\|\varphi\|_{B_{p,q}^\beta} \Big)^{q'},\end{align}
where we again used {\eqref{aronson-gammah}, \eqref{estigradplu}} and \eqref{BD_INFTY_SC_BESOV_SCHEME} for the last inequality. Write now:
\begin{align*}
\mathcal T_{1,p',q'}^{-\beta}(\mathfrak q_x^{h,s,t})\le& \int_0^t \frac{dv}{v}v^{(1+\frac{\beta}{\alpha}) q'}\Big(\int_{\R^d} dz \left|\int_{\R^d}\partial_v p_\alpha(v,z-y)\Big[\Gamma^h\Big(0,x,\tau_s^h,y) \nabla P_{t-\tau_s^h}^\alpha \varphi(y+\int_{\tau_s^h}^s \mathfrak b_h(r,y)dr\Big)\right.\\
&\left.-\Gamma^h(0,x,\tau_s^h,z) \nabla P_{t-\tau_s^h}^\alpha \varphi\Big(z+\int_{\tau_s^h}^s \mathfrak b_h(r,z)dr\Big)\Big]dy \right|^{p'}\Big)^{\frac{q'}{p'}}\notag\\
\le& \int_0^t \frac{dv}{v}v^{\frac{\beta}{\alpha}q'}\Big(\int_{\R^d} dz \Big|\int_{\R^d}\bar p_\alpha(v,z-y) |z-y|^{\rho}\Big(s^{-\frac\rho\alpha}(\bar p_\alpha(s,y-x)+\bar p_\alpha(s,z-x))(t-s)^{-\frac 1\alpha+\frac \beta\alpha-\frac d{p\alpha}}\notag\\
&+\bar p_\alpha(s,z-x)(t-s)^{-\frac 1\alpha+\frac{\beta}{\alpha}-\frac \rho\alpha-\frac d{p\alpha}}
\Big) dy\Big|^{p'}\Big)^{\frac{q'}{p'}}\|\varphi\|_{B_{p,q}^{\beta}}^{q'}, \rho\in (-\beta,-\beta+\gamma),\notag\\
\lesssim&\Big[t^{\frac{\beta+\rho}{\alpha}}\Big(s^{-(\frac \rho\alpha+\frac{d}{p\alpha})}(t-s)^{-\frac 1\alpha+\frac \beta\alpha-\frac d{p\alpha}}+s^{-\frac d{\alpha p}}(t-s)^{-\frac 1\alpha+\frac{\beta}{\alpha}-\frac \rho\alpha-\frac d{p\alpha}}\Big)\Big]^{q'}\|\varphi\|_{B_{p,q}^{\beta}}^{q'}
\end{align*} 
using {\eqref{holder-forward-gammah}, \eqref{BD_INFTY_SC_BESOV_SCHEME}  and  \eqref{aronson-gammah}, \eqref{BD_RHO_SC_BESOV_SCHEME} (recalling that $B_{\infty,\infty}^\rho $ can be identified with the usual Hölder space $C_b^\rho $)} for the last but one inequality {and proceeding {like} in \eqref{CTR_CUT_LOW_0} for the last one}. We thus have \eqref{QUI_SERT_TOUT_LE_TEMPS_EN_BESOV_DRIFTED_SCHEME}. This means that the approximate drift does not perturb the regularity estimate, which turns out to be the same {as} for the diffusion.

\subsubsection{Proof of \eqref{QUI_SERT_TOUT_LE_TEMPS_EN_BESOV_DRIFTED_SCHEME_TIME_SENSI}}\label{prqstbis}
For this paragraph, we set for notational simplicity:
$$ \mathfrak q_{x,{d}}^{h,s,t}(\cdot):=\Gamma^h(0,x,\tau_s^h,\cdot)\Big[\nabla P^\alpha_{t-\tau_s^h}\varphi(\cdot+\int_{\tau_s^h}^s \mathfrak b_h(r,\cdot)dr )-\nabla P^\alpha_{t-s}\varphi(\cdot )
\Big],$$
{where the subscript $d$ is here to recall we are handling a \textit{difference} of semi-groups}.
Let us again first focus on the non-thermic part. Write:
\begin{align}
\|\mathcal F^{-1}(\phi)\star\mathfrak q_{x,{d}}^{h,s,t} 
\|_{L^{p'}}\lesssim & \|\ |{\mathcal F^{-1}(\phi)}|\star\Gamma^h(0,x,\tau_s^h,\cdot) \ \|_{L^{p'}}\|\nabla P_{t-\tau_s^h}^{\alpha}\varphi(\cdot+\int_{\tau_s^h}^s \mathfrak b_h(r,\cdot)dr )-\nabla P_{t-s}^{\alpha}\varphi(\cdot)\|_{L^\infty}\notag\\
\lesssim&\|\nabla P_{t-\tau_s^h}^{\alpha}\varphi(\cdot+\int_{\tau_s^h}^s \mathfrak b_h(r,\cdot)dr )-\nabla P_{t-\tau_s^h}^{\alpha}\varphi(\cdot)\|_{L^\infty}+\|\nabla P_{t-\tau_s^h}^{\alpha}\varphi(\cdot)-\nabla P_{t-s}^{\alpha}\varphi(\cdot)\|_{L^\infty}\notag\\
\lesssim & (t-s)^{-\frac{1-\beta+\frac dp}{\alpha}}\|\varphi\|_{B_{p,q}^{\beta}}\Big(\frac{\|\int_{\tau_s^h}^s \mathfrak b_h(r,\cdot)dr \|_{L^\infty}}{(t-s)^{\frac1\alpha}}+\frac{h^{\frac{\gamma-\varepsilon}{\alpha}}}{(t-s)^{\frac{\gamma-\varepsilon}\alpha}}\Big)\notag\\
\lesssim&  (t-s)^{-\frac{1-\beta+\frac dp}{\alpha}}\|\varphi\|_{B_{p,q}^{\beta}}\Big[\frac{(s-\tau_s^h)^{\frac{\gamma}{\alpha}+\frac{1-\beta}{\alpha}}\|b\|_{L^\vartheta-B_{p,q}^\beta}}{(t-s)^{\frac 1\alpha}}+ \frac{h^{\frac{\gamma-\varepsilon}{\alpha}}}{(t-s)^{\frac{\gamma-\varepsilon}\alpha}}\Big]\notag\\
 \lesssim &(t-s)^{-\frac{1-\beta+\frac dp}{\alpha}}\|\varphi\|_{B_{p,q}^{\beta}}\frac{h^{\frac{\gamma-\varepsilon}{\alpha}}}{(t-s)^{\frac{\gamma-\varepsilon}\alpha}}\label{NON_THERMIC_SENSI_TIME},
\end{align}
using \eqref{BD_INFTY_SC_BESOV_SPACE_SENSI} and \eqref{BD_INFTY_SC_BESOV_TIME_SENSI} (with $\theta=(\gamma-\varepsilon)/\alpha $) for the third inequality, \eqref{CTR_PONCTUEL_BH_INT} for the last but one {and $s-\tau^h_s\le t-s$ for the last one}. 

Keeping the same previous notations, we now write for the upper cut:
\begin{align}\label{CTR_CUT_HIGH_SCHEME_SENSI_TIME}
\mathcal T_{2,p',q'}^{-\beta}(\mathfrak q_{x,{d}}^{h,s,t})&\le \int_{t}^{T}\frac{dv}{v} v^{(1+\frac{\beta}{\alpha})q'} \|\partial_{v} p_\alpha(v,\cdot)\|_{L^1}^{q'}\|\Gamma^h(0,x,\tau_s^h,\cdot)\|_{L^{p'}}^{q'}\|\nabla P_{t-\tau_s^h}^\alpha \varphi(\cdot+\int_{\tau_s^h}^s \mathfrak b_h(r,\cdot) dr)-\nabla P_{t-s}^\alpha \varphi\|_{L^\infty}^{q'}\notag\\
&\lesssim \Big(t^{\frac \beta \alpha } s^{-\frac d{p \alpha}} (t-s)^{-\frac 1\alpha+\frac{\beta}{\alpha}-\frac d{ p\alpha} }\|\varphi\|_{B_{p,q}^\beta} \frac{h^{\frac{\gamma-\varepsilon}{\alpha}}}{(t-s)^{\frac{\gamma-\varepsilon}\alpha}}\Big)^{q'},
\end{align}
proceeding as for \eqref{NON_THERMIC_SENSI_TIME} for the last inequality. We now turn to the lower cut:
\begin{align}
\mathcal T_{1,p',q'}^{-\beta}(\mathfrak q_{x,{d}}^{h,s,t})
\le& \int_0^t \frac{dv}{v}v^{(1+\frac{\beta}{\alpha}) q'}\Big(\int_{\R^d} dz \left|\int_{\R^d}\partial_v p_\alpha(v,z-y)\Big[(\Gamma^h(0,x,\tau_s^h,y)-\Gamma^h(0,x,\tau_s^h,z)) \right.\notag\\
&\phantom{\int_0^t }\left.\times \big[\nabla P_{t-\tau_s^h}^\alpha \varphi(y+\int_{\tau_s^h}^s \mathfrak b_h(r,y)dr)-\nabla P_{t-s}^\alpha \varphi(y)\big]+\Gamma^h(0,x,\tau_s^h,z) \Delta(s,y,z)\Big]dy \right|^{p'}\Big)^{\frac{q'}{p'}}\label{THE_SENSI_TO_CONCLUDE}
  \end{align} where 
\begin{align*}
\Delta(s,y,z):=&\big[\nabla P_{t-\tau_s^h}^\alpha \varphi\Big(y+\int_{\tau_s^h}^s \mathfrak b_h(r,y)dr\Big)-\nabla P_{t-s}^\alpha \varphi(y)]-\big[\nabla P_{t-\tau_s^h}^\alpha \varphi\Big(z+\int_{\tau_s^h}^s \mathfrak b_h(r,z)dr\Big)-\nabla P_{t-s}^\alpha \varphi(z)]\\
=&\big[\nabla P_{t-\tau_s^h}^\alpha \varphi\Big(y+\int_{\tau_s^h}^s \mathfrak b_h(r,y)dr\Big)-\nabla P_{t-\tau_s^h}^\alpha \varphi(y)+\nabla P_{t-\tau_s^h}^\alpha \varphi(y)-\nabla P_{t-s}^\alpha \varphi(y)]\\
&-\big[\nabla P_{t-\tau_s^h}^\alpha \varphi\Big(z+\int_{\tau_s^h}^s \mathfrak b_h(r,z)dr\Big)-\nabla P_{t-\tau_s^h}^\alpha \varphi(z)+\nabla P_{t-\tau_s^h}^\alpha \varphi(z)-\nabla P_{t-s}^\alpha \varphi(z)]\\
=&\int_0^1 d\lambda \nabla^2 P_{t-\tau_s^h}^\alpha \varphi\Big(y+\lambda \int_{\tau_s^h}^s \mathfrak b_h(r,y)dr\Big)\int_{\tau_s^h}^s \mathfrak b_h(r,y)dr\\
 &-\int_0^1 d\theta [ \partial_r\nabla P_{r}^\alpha \varphi(y)-\partial_r\nabla P_{r}^\alpha \varphi(z)]|_{r=t-\tau_s^h+\theta({\tau_s^h-s})}](s-\tau_s^h)\\
 &-\int_0^1 d\lambda \nabla^2 P_{t-\tau_s^h}^\alpha \varphi\Big(z+\lambda \int_{\tau_s^h}^s \mathfrak b_h(r,z)dr\Big)\int_{\tau_s^h}^s \mathfrak b_h(r,z)dr.
\end{align*}
Using \eqref{BD_RHO_SC_BESOV} for the first inequality, \eqref{BD_RHO_SC_BESOV_SCHEME}, \eqref{CTR_PONCTUEL_BH_INT}, 
 \eqref{BD_INFTY_SC_BESOV_SCHEME}, \eqref{CTR_PONCTUEL_BH_INT_HOLDER} for the second, {as well as $
\frac{(s-\tau_s^h)^{\frac{\gamma}{\alpha}+\frac{1-\beta-\rho}{\alpha}}}{(t-s)^{\frac 2\alpha}}
\le \frac{(s-\tau_s^h)^{\frac{\gamma-\beta}{\alpha}}}{(t-s)^{\frac {1+\rho}\alpha} } $ deduced from $(s-\tau_s^h)\le t-s $} for the third, we get:
\begin{align*}
|\Delta(s,y,z) |\lesssim &\Big|\int_0^1 d\lambda \Big[\nabla^2 P_{t-\tau_s^h}^\alpha \varphi\Big(y+\lambda \int_{\tau_s^h}^s \mathfrak b_h(r,y)dr\Big)-\nabla^2 P_{t-\tau_s^h}^\alpha \varphi\Big(z+\lambda \int_{\tau_s^h}^s \mathfrak b_h(r,z)dr\Big)\Big]\Big| \ \Big|\int_{\tau_s^h}^s \mathfrak b_h(r,y)dr\Big|\\
&+(t-s)^{-(1+\frac 1\alpha+{\frac \rho\alpha}-\frac{\beta}{\alpha}+\frac d{p\alpha}) }|z-y|^\rho \|\varphi\|_{B_{p,q}^\beta}(s-\tau_s^h)\\ 
&+\Big|\int_0^1 d\lambda \nabla^2 P_{t-\tau_s^h}^\alpha \varphi\Big(z+\lambda \int_{\tau_s^h}^s \mathfrak b_h(r,z)dr\Big)\Big| \ \Big|\int_{\tau_s^h}^s [\mathfrak b_h(r,z)- \mathfrak b_h(r,y)]dr\Big|\\
\lesssim &|y-z|^\rho 
{(t-s)^{-\frac 2\alpha+\frac{\beta}{\alpha}-\frac{\rho}{\alpha}-\frac{d}{\alpha p}}\|\varphi\|_{B_{p,q}^\beta}} (s-\tau_s^h)^{\frac{\gamma}{\alpha}+\frac{1-\beta}{\alpha}}\|b\|_{L^\vartheta-B_{p,q}^\beta}.\\
&+(t-s)^{-(1+\frac 1\alpha+{\frac\rho\alpha}-\frac{\beta}{\alpha}+\frac d{p\alpha}) }|z-y|^\rho \|\varphi\|_{B_{p,q}^\beta}(s-\tau_s^h)\\ 
&+(t-s)^{-(\frac{2}{\alpha}-\frac{\beta}{\alpha}+\frac d{p\alpha})}\|\varphi\|_{B_{p,q}^\beta}  |z-y|^\rho (s-\tau_s^h)^{\frac{\gamma}{\alpha}+\frac{1-\beta-\rho}{\alpha}}\|b\|_{L^\vartheta-B_{p,q}^\beta}\notag\\
\lesssim& (t-s)^{-(\frac 1\alpha+{\frac \rho\alpha}-\frac{\beta}{\alpha}+\frac d{p\alpha}) }|z-y|^\rho \|\varphi\|_{B_{p,q}^\beta}\left( \frac{h}{t-s}\right)^{\frac{\gamma-\varepsilon}{\alpha}}.
\end{align*}
Plugging this estimation together with \eqref{holder-forward-gammah}, \eqref{BD_INFTY_SC_BESOV} and \eqref{BD_INFTY_SC_BESOV_SCHEME} in \eqref{THE_SENSI_TO_CONCLUDE}, we conclude that
\begin{align}
\mathcal T_{1,p',q'}^{-\beta}(\mathfrak q_{x,{d}}^{h,s,t})\le& \int_0^t \frac{dv}{v}v^{\frac{\beta}{\alpha}q'}\Big(\int_{\R^d} dz \Big|\int_{\R^d}\bar p_\alpha(v,z-y) |z-y|^{\rho}\Big(s^{-\frac\rho\alpha}(\bar p_\alpha(s,y-x)+\bar p_\alpha(s,z-x))(t-s)^{-\frac 1\alpha+\frac \beta\alpha-\frac d{p\alpha}}\notag\\
&+\bar p_\alpha(s,z-x)(t-s)^{-\frac 1\alpha+\frac{\beta}{\alpha}-\frac \rho\alpha-\frac d{p\alpha}}
\Big) \frac{h^{\frac{\gamma-\varepsilon}{\alpha}}}{(t-s)^{\frac{\gamma-\varepsilon}\alpha}}dy\Big|^{p'}\Big)^{\frac{q'}{p'}}\|\varphi\|_{B_{p,q}^{\beta}}^{q'}, \rho\in (-\beta,-\beta+\gamma),\notag\\
\lesssim&\Big[t^{\frac{\beta+\rho}{\alpha}}\Big(s^{-(\frac \rho\alpha+\frac{d}{p\alpha})}(t-s)^{-\frac 1\alpha+\frac \beta\alpha-\frac d{p\alpha}}+s^{-\frac d{\alpha p}}(t-s)^{-\frac 1\alpha+\frac{\beta}{\alpha}-\frac \rho\alpha-\frac d{p\alpha}}\Big)\frac{h^{\frac{\gamma-\varepsilon}{\alpha}}}{(t-s)^{\frac{\gamma-\varepsilon}\alpha}}\Big]^{q'}\|\varphi\|_{B_{p,q}^{\beta}}^{q'}.
\end{align}

 \subsubsection{Proof of \eqref{holder-time-gamma_PP_QP_NO_NORM}, forward time regularity of the Besov norm of the heat kernel} \label{THE_FINAL_EST}
We recall that we here aim at proving that, \textcolor{black}{for all $\eps>0$ meant to be small}
, {$0<t<t'\le T$} such that  $|t-t'|\le t/2$, 
		\begin{equation*}
			\left\Vert \Gamma (0,x,t,\cdot)-\Gamma (0,x,t',\cdot) \right\Vert_{B_{p',q'}^\rho} \leq C \frac{(t'-t)^\frac{\gamma-\eps}{\alpha}}{t^\frac{\gamma-\eps+\rho+\frac dp}{\alpha}},\ \rho\in (-\beta,-\beta+\gamma).
		\end{equation*}
Let us write from the Duhamel representation \eqref{duhamel-Diff} of the density that:
\begin{align}
\Gamma(0,x,t,y)-\Gamma(0,x,t',y)
			=& p_\alpha(t,y-x)-p_\alpha(t',y-x)\notag\\
			&-\int_{0}^{ t}\E_{{0},x}\left[b(s,X_s)\cdot\Big(\nabla_y  p_\alpha(t-s,y-X_s)-\nabla_y  p_\alpha(t'-s,y-X_s)\Big)\right]\d s\notag\\
			&+\int_{t}^{ t'}\E_{{0},x}\left[b(s,X_s)\cdot\Big(\nabla_y  p_\alpha(t'-s,y-X_s)\Big)\right]\d s:=\sum_{i=1}^3 \Delta^i(0,x,t,t',y).\label{THE_DIFF-BIS}
\end{align}
By the fundamental theorem of analysis, we get {using \eqref{SING_STABLE_HK} and $t'-t\le \frac t2$}:
\begin{align}\label{CT_MT_HT}
\|\Delta^1(0,x,t,t',\cdot)\|_{B_{p',q'}^\rho}=\|p_\alpha(t,\cdot-x)-p_\alpha(t',\cdot-x)\|_{B_{p',q'}^\rho}\lesssim \frac{(t'-t)^\frac{\gamma-\eps}{\alpha}}{t^\frac{\gamma-\eps+\rho+\frac dp}{\alpha}}.
\end{align}
Let us now handle the term $\Delta^3 $. Write first:
\begin{align*}
\|\mathcal F^{-1}(\phi)\star\Delta^3(0,x,t,t',\cdot)\|_{L^{p'}}&\lesssim \|\int_{t}^{ t'}\int \Gamma(0,x,s,z) b(s,z)\cdot \nabla  p_\alpha(t'-s,\cdot-z)\d z  \d s\|_{L^1}\\
&\lesssim\int_{t}^{t'} \int_{\R^d}\Big|\int \Gamma(0,x,s,z) b(s,z)\cdot \nabla  p_\alpha(t'-s,y-z) \Big| dy ds\\
&\lesssim \int_t^{t'} \int_{\R^d}\Vert b (s,\cdot)\Vert_{B_{p,q}^\beta}  \Vert \Gamma(0,s,x,\cdot)\nabla_y p_{\alpha} (t'-s,y-\cdot)\Vert_{B_{p',q'}^{-\beta}}\d s dy
\end{align*} 
where we have used \eqref{dual-ineq} for the last inequality. 

To proceed with the analysis we will need the following two controls established in Lemma 5 from \cite{fito:isso:meno:25}:
\begin{trivlist}
			\item[-] $\forall 0< s  < t$, $\forall (x,y)\in (\R^d)^2$,  
			$\forall \zeta \in (-\beta,\textcolor{black}{-\beta+\gamma})$,
			\begin{align}\label{besov-estimate-gamma-and-stable}
				\Vert& \Gamma(0,x,s,\cdot) \nabla_y p_{\alpha} (t-s,y-\cdot) \Vert_{\B^{-\beta
				}_{p',q'}
				} \lesssim 
				\frac{\bar{p}_{\alpha} (t,x-y)}{(t-s)^{\frac{1}{\alpha}}} t^{\frac{\beta
				}{\alpha}}\left[ \frac{1}{s^{\frac{d }{\alpha  p}}}+\frac{1}{(t-s)^{\frac{d }{\alpha  p}}} \right] \left[
				  \frac{t^{\frac{\zeta}{\alpha}}}{s^{\frac{\zeta }{\alpha}}}+\frac{t^{\frac{\zeta}{\alpha}}}{(t-s)^{\frac{\zeta }{\alpha}}}  \right] .
			\end{align}
			\item[-]
			 $\forall 0< s  < t$, $\forall (x,y,y')\in (\R^d)^3$, s.t. $|y'-y|\le (t-s)^{\frac 1\alpha} $, $\forall \zeta \in (-\beta,\textcolor{black}{-\beta+\gamma}),{\tilde \rho} \in (-\beta,1]$, 
			
			\begin{align}\label{besov-estimate-gamma-and-stable_NO_NORM_Holder}
			&\Vert \Gamma(0,x,s,\cdot) (\nabla_y p_{\alpha} (t-s,y-\cdot)-\nabla_y p_{\alpha} (t-s,y'-\cdot)) \Vert_{\B^{-\beta
			}_{p',q'}} \notag\\
			\lesssim& \bar{p}_{\alpha} (t,x-y)\frac{|y-y'|^{{\tilde\rho}}}{(t-s)^{\frac 1\alpha+\frac{{\tilde \rho}}\alpha}} t^{\frac{\beta
			}{\alpha}}\left[ \frac{1}{s^{\frac{d }{\alpha  p}}}+\frac{1}{(t-s)^{\frac{d }{\alpha  p}}} \right] \left[  
			\frac{t^{\frac{\zeta}{\alpha}}}{s^{\frac{\zeta }{\alpha}}}+\frac{t^{\frac{\zeta}{\alpha}}}{(t-s)^{\frac{\zeta }{\alpha}}}  \right] .\end{align}
\end{trivlist}
We mention that the control \eqref{besov-estimate-gamma-and-stable} will mainly allow to handle the thermic part and to derive as well that \eqref{besov-estimate-gamma-and-stable_NO_NORM_Holder} extends globally to
\begin{align}\label{besov-estimate-gamma-and-stable_NO_NORM_Holder_GLOBAL}
			&\Vert \Gamma(0,x,s,\cdot) (\nabla_y p_{\alpha} (t-s,y-\cdot)-\nabla_y p_{\alpha} (t-s,y'-\cdot)) \Vert_{\B^{-\beta
			}_{p',q'}} \notag\\
			\lesssim& \big(\bar{p}_{\alpha} (t,x-y)+\bar{p}_{\alpha} (t,x-y')\big)\frac{|y-y'|^{{\tilde \rho}}}{(t-s)^{\frac 1\alpha+\frac{{\tilde \rho}}\alpha}} t^{\frac{\beta}{\alpha}}\left[ \frac{1}{s^{\frac{d }{\alpha  p}}}+\frac{1}{(t-s)^{\frac{d }{\alpha  p}}} \right] \left[%1+  
			\frac{t^{\frac{\zeta}{\alpha}}}{s^{\frac{\zeta }{\alpha}}}+\frac{t^{\frac{\zeta}{\alpha}}}{(t-s)^{\frac{\zeta }{\alpha}}}  \right],
			\end{align}
since in the off-diagonal regime $|y-y'|>(t-s)^{\frac 1\alpha} $, the above estimate is indeed a direct consequence of \eqref{besov-estimate-gamma-and-stable} and the triangular inequality. Now, from \eqref{besov-estimate-gamma-and-stable}, we get {also using that for $s\in[t,t']$, $t'-s\le t'-t\le \frac{t}{2}\le \frac s 2$}:
\begin{align}
&\|\mathcal F^{-1}(\phi)\star\Delta^3(0,x,t,t',\cdot)\|_{L^{p'}}\lesssim\|\Delta^3(0,x,t,{t',}\cdot)\|_{L^{1}}\notag\\
&\lesssim \int_t^{t'}  \Vert b (s,\cdot)\Vert_{B_{p,q}^\beta}\int_{\R^d}  \frac{\bar{p}_{\alpha} (t',x-y)}{(t'-s)^{\frac{1}{\alpha}}} (t')^{\frac{\beta}{\alpha}}\left[ \frac{1}{s^{\frac{d }{\alpha  p}}}+\frac{1}{(t'-s)^{\frac{d }{\alpha  p}}} \right] \left[
				  \frac{(t')^{\frac{\zeta}{\alpha}}}{s^{\frac{\zeta }{\alpha}}}+\frac{(t')^{\frac{\zeta}{\alpha}}}{(t'-s)^{\frac{\zeta }{\alpha}}}  \right]dydt\notag\\
				  &\lesssim \|b\|_{L^\vartheta-B_{p,q}^\beta}(t')^{\frac{\beta+\zeta}{\alpha}}\Big(\int_t^{t'}  (t'-s)^{-(\frac 1\alpha+\frac{d}{\alpha p}+\frac{\zeta}{\alpha})\vartheta'}  ds\Big)^{\frac{1}{\vartheta'}}\lesssim (t'-t)^{\frac{\alpha-1-\zeta-(\frac dp+\frac \alpha\vartheta)}\alpha}\notag\\
				  &\lesssim (t'-t)^{\frac{\gamma-\zeta-2\beta}\alpha}\lesssim (t'-t)^{\frac{\gamma-\varepsilon-\beta}\alpha},\label{CTR_DELTA_3_NON_THERM_RT}
\end{align}
since the smallest $\zeta $, e.g. $\zeta=-\beta+\varepsilon$, with $\varepsilon>0 $ meant to be small, will here give the largest exponent for $t'-t $. 
Hence, this is not a \textit{critical} term w.r.t. the statement. Let us now consider the thermic part. We keep the same previous notations for the lower and upper cuts. Only dealing with the case $q\ne 1$ (the case $q=1$ is easier), we write:
\begin{align}
\Big[\mathcal T_{p',q'}^{\rho}\Big(\Delta^3(0,x,t,t',\cdot)\Big)\Big]^{q'}
=&\int_0^{t} \frac{dv}{v}v^{(1-\frac{\rho}{\alpha}) q'}\|\partial_v p_\alpha(v,\cdot)\star \Delta^3(0,x,t,t',\cdot)\|_{L^{p'}}^{q'}\notag\\
&+\int_{t}^T \frac{dv}{v}v^{(1-\frac{\rho}{\alpha}) q'}\|\partial_v p_\alpha(v,\cdot)\star \Delta^3(0,x,t,t',\cdot)\|_{L^{p'}}^{q'}\notag\\
=:&\big(\mathcal T_{1,p',q'}^{\rho}+\mathcal T_{2,p',q'}^{\rho}\big)(\Delta^3(0,x,t,t',\cdot)),\label{CUT_FOR_BESOV_NORM_HK_FOR_TR}
\end{align}
{For the upper-cut  using the $L^1 $-norm, similarly to the non-thermic part, we derive from \eqref{estigradplu}, \eqref{CTR_DELTA_3_NON_THERM_RT} and $t'-t\le\frac t2$:}
\begin{align*}
\mathcal T_{2,p',q'}^{\rho}(\Delta^3(0,x,t,t',\cdot))&\lesssim \int_{t}^T \frac{dv}{v}v^{(1-\frac{\rho}{\alpha}) q'}\|\partial_v \bar p_\alpha(v,\cdot)\|_{L^{p'}}^{q'}\|\Delta^3(0,x,t,t',\cdot)\|_{L^1}^{q'}\\
&\lesssim \int_{t}^T \frac{dv}{v}v^{-(\frac{\rho}{\alpha}+\frac d{\alpha p}) q'}(t'-t)^{(\frac{\gamma-\varepsilon-\beta}\alpha)q'}\lesssim \left(\frac{(t'-t)^{\frac{\gamma-\varepsilon-\beta}\alpha}}{t^{\frac{\rho}{\alpha}+\frac d{p\alpha}}}\right)^{q'}\lesssim \Big( \frac{(t'-t)^{\frac{\gamma-\varepsilon}\alpha}}{t^{\frac{\beta+\rho}{\alpha}+\frac d{p\alpha}}}\Big)^{q'}.
\end{align*}
For the lower cut, write:
\begin{align*}
\mathcal T_{1,p',q'}^{\rho}(\Delta^3(0,x,t,t',\cdot))\lesssim& \int_{0}^t \frac{dv}{v}v^{(1-\frac{\rho}{\alpha}) q'}\Big(
\int_t^{t'} ds \| \partial_v p_\alpha(v,\cdot) \star  \int \Gamma(0,x,s,z) b(s,z)\cdot\nabla  p_\alpha(t'-s,\cdot-z)dz\|_{L^{p'}}\Big)^{q'}\\
\lesssim& \int_{0}^t \frac{dv}{v}v^{(1-\frac{\rho}{\alpha}) q'}\Big(
\int_t^{t'} ds \Big(\int_{\R^d} dy \Big|\int_{\R^d}\partial_v p_\alpha(v,y-w) \times\\
& \Big[\int_{\R^d} \Gamma(0,x,s,z) b(s,z)\cdot[ \nabla  p_\alpha(t'-s,w-z))-\nabla  p_\alpha(t'-s,y-z)] dz\Big]dw\Big|^{p'}\Big)^{\frac{1}{p'}}\Big)^{q'}\\
\lesssim& \int_{0}^t \frac{dv}{v}v^{(1-\frac{\rho}{\alpha}) q'}\Big(
\int_t^{t'} ds \Big(\int_{\R^d} dy \Big(\int_{\R^d} |\partial_v p_\alpha(v,y-w)| \\
&\times \|b(s,\cdot)\|_{B_{p,q}^{\beta}} \|\Gamma(0,x,s,\cdot) [ \nabla  p_\alpha(t'-s,w-\cdot))-\nabla  p_\alpha(t'-s,y-\cdot)\|_{B_{p',q'}^{-\beta}}dw\Big)^{p'}\Big)^{\frac{1}{p'}}\Big)^{q'},
\end{align*}
using the Minkowski inequality and a cancellation argument for the first two inequalities and \eqref{dual-ineq} for the last one. Use now \eqref{besov-estimate-gamma-and-stable_NO_NORM_Holder_GLOBAL} { with $\tilde \rho>\rho $}, to derive:
\begin{align*}
\mathcal T_{1,p',q'}^{\rho}(\Delta^3(0,x,t,t',\cdot))\lesssim& \int_{0}^t \frac{dv}{v}v^{(1-\frac{\rho}{\alpha}) q'}(t')^{\frac{\beta}{\alpha}q'}\\
			&\times \Big(
\int_t^{t'} ds\|b(s,\cdot)\|_{B_{p,q}^{\beta}} \frac 1{(t'-s)^{\frac 1\alpha+\frac{\tilde \rho}\alpha}}\left[ \frac{1}{s^{\frac{d }{\alpha  p}}}+\frac{1}{(t'-s)^{\frac{d }{\alpha  p}}} \right] \left[
			\frac{(t')^{\frac{\zeta}{\alpha}}}{s^{\frac{\zeta }{\alpha}}}+\frac{(t')^{\frac{\zeta}{\alpha}}}{(t'-s)^{\frac{\zeta }{\alpha}}}  \right]\\
			&\times \Big(\int_{\R^{d}} dy \Big(\int_{\R^d} dw v^{-1}\bar p_\alpha(v,y-w) |y-w|^{\tilde \rho}
(\bar p_\alpha(t',w-x)+\bar p_\alpha(t',y-x))  \Big)^{p'} \Big)^{\frac 1{p'}} \Big)^{q'}\\
\lesssim &\int_{0}^t \frac{dv}{v}v^{-\frac{\rho}{\alpha} q'}(t')^{q'\frac{\beta}{\alpha}}\\
			&\times \Big(
\int_t^{t'} ds\|b(s,\cdot)\|_{B_{p,q}^{\beta}} \frac 1{(t'-s)^{\frac 1\alpha+\frac{\tilde \rho}\alpha}}\left[ \frac{1}{s^{\frac{d }{\alpha  p}}}+\frac{1}{(t'-s)^{\frac{d }{\alpha  p}}} \right] \left[
			\frac{(t')^{\frac{\zeta}{\alpha}}}{s^{\frac{\zeta }{\alpha}}}+\frac{(t')^{\frac{\zeta}{\alpha}}}{(t'-s)^{\frac{\zeta }{\alpha}}}  \right]\\
&\times\Big(\|(\bar p_\alpha (v,\cdot)|\cdot|^{\tilde \rho})\star \bar p_{\alpha}(t',\cdot-x) \|_{L^{p'}} +\|\bar p_\alpha(t',\cdot-x)\|_{L^{p'}}\|\bar p_\alpha (v,\cdot)|\cdot|^{\tilde \rho}\|_{L^1}\Big)\Big)^{q'}\\
\lesssim &\int_{0}^t \frac{dv}{v}v^{\frac{\tilde \rho-\rho}{\alpha} q'}(t')^{q'(\frac{\beta+\zeta-\frac dp}{\alpha})}\|b\|_{L^\vartheta-B_{p,q}^\beta}^{q'}\Big(\int_{t}^{t'} \frac{ds}{(t'-s)^{(\frac 1\alpha+\frac{\tilde \rho+\zeta}{\alpha}+\frac{d}{\alpha p})\vartheta'}}\Big)^{\frac{q'}{\vartheta'}}\\\lesssim&\Big(t^{\frac{\tilde \rho-\rho}{\alpha}+\frac{\beta+\zeta-\frac dp}{\alpha}}(t'-t)^{\frac{\alpha-1 -(\tilde \rho+\zeta)-(\frac dp+\frac{\alpha}{\vartheta})}{\alpha}}\Big)^{q'}
\end{align*}
choosing $\tilde \rho=-\beta+\varepsilon/2=\zeta, \tilde \rho>\rho $ with $\varepsilon<\gamma $ so that the above time singularities remain integrable. Recalling that $\tilde \rho>\rho, \zeta>-\beta $, this choice therefore  yields:
\begin{align*}
\mathcal T_{1,p',q'}^{\rho}(\Delta^3(0,x,t,t',\cdot))\lesssim \Big((t'-t)^{\frac{\gamma-\varepsilon}{\alpha}}t^{-\frac{d}{\alpha p}}\Big)^{q'},
\end{align*}
which together with the previous control on the upper cut, \eqref{CUT_FOR_BESOV_NORM_HK_FOR_TR} and \eqref{CTR_DELTA_3_NON_THERM_RT} gives the expected control for the term $\Delta^3(0,x,t,t',\cdot)$. Namely,
\begin{align}\label{CT_LT_HT}
\|\Delta^3(0,x,t,t',\cdot)\|_{B_{p',q'}^{\rho}}\lesssim (t'-t)^{\frac{\gamma-\varepsilon}{\alpha}}t^{-(\frac d{\alpha p}+\frac{\beta+\rho}{\alpha})}.
\end{align}

We now proceed similarly for the term $\Delta_2 $. For technical reasons we are going to split this term into two {contributions}. Namely,
\begin{align*}
\Delta^2(0,x,t,t',\cdot)=&\int_{0}^{ t-(t'-t)}\E_{{0},x}\left[b(s,X_s)\cdot\Big(\nabla_y  p_\alpha(t-s,y-X_s)-\nabla_y  p_\alpha(t'-s,y-X_s)\Big)\right]\d s\\
&+\int_{t-(t'-t)}^{ t}\E_{{0},x}\left[b(s,X_s)\cdot\Big(\nabla_y  p_\alpha(t-s,y-X_s)-\nabla_y  p_\alpha(t'-s,y-X_s)\Big)\right]\d s\\
=:&(\Delta^{21}+\Delta^{22})(0,x,t,t',\cdot).
\end{align*}
The contribution $\|\Delta^{22}(0,x,t,t',\cdot)\|_{B_{p',q'}^{\rho}} $ can be analyzed {like} the former term $\|\Delta^{3}(0,x,t,t',\cdot)\|_{B_{p',q'}^{\rho}} $. On the other hand, for $s\in [0,t-(t'-t)] $, $t-s$ and $t'-s$ are both greater than $t'-t$. This is a somehow natural constraint in order to compare the two heat kernels in time.
\begin{align*}
&\|\mathcal F^{-1}(\phi)\star\Delta^{21}(0,x,t,t',\cdot)\|_{L^{p'}}\lesssim \|\Delta^{21}(0,x,t,t',\cdot)\|_{L^1}\\
\lesssim& \|\int_{0}^{ t-(t'-t)}\int_{{\R^d}} \Gamma(0,x,s,z) b(s,z)\cdot(\nabla  p_\alpha(t'-s,\cdot-z)-\nabla  p_\alpha(t-s,\cdot-z))\d z  \d s\|_{L^1}\\
\lesssim&\int_{0}^{t-(t'-t)} \int_{\R^d}\Big|\int \Gamma(0,x,s,z) b(s,z)\cdot(\nabla  p_\alpha(t'-s,y-z)-\nabla  p_\alpha(t-s,y-z)) dz\Big| dy ds\\
\lesssim &\int_0^{t-(t'-t)} \int_{\R^d}\Vert b (s,\cdot)\Vert_{B_{p,q}^\beta}  \Vert \Gamma(0,s,x,\cdot)(\nabla_y p_{\alpha} (t'-s,y-\cdot)-\nabla_y p_{\alpha} (t-s,y-\cdot))\Vert_{B_{p',q'}^{-\beta}}\d s dy
\end{align*} 

We now need some auxiliary estimates that can be proved similarly to controls of Lemma 5 in \cite{fito:isso:meno:25}. Precisely, it is established there that:
\begin{trivlist}
\item[-]   $ \forall h\leq s \leq \tau_t^h-h$, $r\in \textcolor{black}{[}\tau_s^h,\tau_s^h+h)$, $\forall (x,y)\in (\R^d)^2$, $\forall \zeta \in (-\beta,\textcolor{black}{-\beta+\gamma)}$, $\forall \delta\in [0,1) $, 
			\begin{align}\label{besov-estimate-gammah-stable-sensi-holder-time}
				&\left\Vert \Gamma^h(0,x,\tau_s^h,\cdot)\left[\nabla_y p_{\alpha} (t-s,y-\cdot) -\nabla_y p_{\alpha} (t-r,y-\cdot) \right] \right\Vert_{ \B_{p',q'}^{-\beta}} \nonumber \\& \qquad\lesssim (s-r)^\delta\frac{\bar p_{\alpha}(t,y-x) }{(t-s)^{\frac{1}{\alpha}+\delta}} t^{\frac{\beta}{\alpha}}\left[\frac{1}{s^{\frac{d }{\alpha  p}}}+ \frac{1}{(t-s)^{\frac{d }{\alpha  p}}} \right] \left[
				\frac{t^{\frac{\zeta}{\alpha}}}{s^{\frac{\zeta }{\alpha}}}+\frac{t^{\frac{\zeta}{\alpha}}}{(t-s)^{\frac{\zeta }{\alpha}}}  \right] .
			\end{align}
			
\item[-]  $ \forall h\leq s \leq \tau_t^h-h$, $r\in \textcolor{black}{[}\tau_s^h,\tau_s^h+h)$, $\forall (x,y,y')\in (\R^d)^3$, $\forall \zeta,{\tilde \rho} \in (-\beta,1]^2$, $|y-y'|\le (t-r)^{\frac 1\alpha} $, $\forall \delta\in [0,1),\forall \theta \in\{0,1\} $,
			\begin{align}\label{besov-estimate-stable-derivees_temps_esp_sensi_holder_esp}
				&\left\Vert \bar {p}_{{\alpha}} (\tau_s^h,x-\cdot)\left[\partial_t^\theta\nabla_y p_{\alpha} (t-r,y-\cdot) -\partial_t^\theta\nabla_y p_{\alpha} (t-r,y'-\cdot) \right] \right\Vert_{ \B_{p',q'}^{-\beta}} \nonumber \\& \qquad\lesssim |y-y'|^{{\tilde \rho}} \frac{\bar p_{\alpha}(t,y-x) }{(t-s)^{\frac{1}{\alpha}+\theta+\frac{{\tilde \rho}}{\alpha}}} t^{\frac{\beta}{\alpha}}\left[\frac{1}{s^{\frac{d }{\alpha  p}}}+ \frac{1}{(t-s)^{\frac{d }{\alpha  p}}} \right] \left[%1+
				 \frac{t^{\frac{\zeta}{\alpha}}}{s^{\frac{\zeta }{\alpha}}}+\frac{t^{\frac{\zeta}{\alpha}}}{(t-s)^{\frac{\zeta }{\alpha}}}  \right] .
			\end{align}			
\end{trivlist}
The constraint about the last time step appears in order that $t-r$,  $t-s$ and $t-\tau_s^h $ are comparable. Passing to the limit thanks to the controls on the Euler scheme and the error or reproducing the arguments of \cite{fito:isso:meno:25} to prove \eqref{besov-estimate-gammah-stable-sensi-holder-time} and \eqref{besov-estimate-stable-derivees_temps_esp_sensi_holder_esp}, one gets:
\begin{trivlist}
\item[-]   $ 0\leq s \leq t-(t'-t)$, $\forall (x,y)\in (\R^d)^2$, $\forall \zeta \in (-\beta,\textcolor{black}{-\beta+\gamma)}$, $\forall \delta\in [0,1) $, 
			\begin{align}\label{besov-estimate-gamma-stable-sensi-holder-time}
				&\left\Vert \Gamma(0,x,s,\cdot)\left[\nabla_y p_{\alpha} (t-s,y-\cdot) -\nabla_y p_{\alpha} (t'-s,y-\cdot) \right] \right\Vert_{ \B_{p',q'}^{-\beta}} \nonumber \\& \qquad\lesssim (t'-t)^\delta\frac{\bar p_{\alpha}(t,y-x) }{(t-s)^{\frac{1}{\alpha}+\delta}} t^{\frac{\beta}{\alpha}}\left[\frac{1}{s^{\frac{d }{\alpha  p}}}+ \frac{1}{(t-s)^{\frac{d }{\alpha  p}}} \right] \left[
				\frac{t^{\frac{\zeta}{\alpha}}}{s^{\frac{\zeta }{\alpha}}}+\frac{t^{\frac{\zeta}{\alpha}}}{(t-s)^{\frac{\zeta }{\alpha}}}  \right] .
			\end{align}
			
\item[-]  $ \forall 0\leq s < t-(t'-t)$, $\forall (x,y,y')\in (\R^d)^3$, $\forall \zeta,{\tilde \rho} \in (-\beta,1]^2$, $|y-y'|\le (t-s)^{\frac 1\alpha} $, $\forall \delta\in [0,1),\forall \theta \in\{0,1\} $,
			\begin{align}\label{besov-estimate-stable-derivees_temps_esp_sensi_holder_espbis}
				&\left\Vert \Gamma (0,x,s,\cdot)\left[\partial_t^\theta\nabla_y p_{\alpha} (t-s,y-\cdot) -\partial_t^\theta\nabla_y p_{\alpha} (t-s,y'-\cdot) \right] \right\Vert_{ \B_{p',q'}^{-\beta}} \nonumber \\& \qquad\lesssim |y-y'|^{{\tilde \rho}} \frac{\bar p_{\alpha}(t,y-x) }{(t-s)^{\frac{1}{\alpha}+\theta+\frac{{\tilde \rho}}{\alpha}}} t^{\frac{\beta}{\alpha}}\left[\frac{1}{s^{\frac{d }{\alpha  p}}}+ \frac{1}{(t-s)^{\frac{d }{\alpha  p}}} \right] \left[%1+
				 \frac{t^{\frac{\zeta}{\alpha}}}{s^{\frac{\zeta }{\alpha}}}+\frac{t^{\frac{\zeta}{\alpha}}}{(t-s)^{\frac{\zeta }{\alpha}}}  \right].
			\end{align}			
\end{trivlist}
Observe as well that \eqref{besov-estimate-stable-derivees_temps_esp_sensi_holder_espbis} could be proven similarly to \eqref{besov-estimate-gamma-and-stable_NO_NORM_Holder} considering an additional Taylor expansion in time for the gradient of the stable heat kernel. Since \eqref{besov-estimate-gamma-and-stable} could be extended as well to this setting (i.e. with a potential additional time derivative and the corresponding time singularities in the r.h.s.), we derive the following global estimate, i.e. independently on the condition on $y-y' $,
			\begin{align}\label{besov-estimate-stable-derivees_temps_esp_sensi_holder_esp_GLB}
				&\left\Vert \Gamma (0,x,s,\cdot)\left[\partial_t^\theta\nabla_y p_{\alpha} (t-s,y-\cdot) -\partial_t^\theta\nabla_y p_{\alpha} (t-s,y'-\cdot) \right] \right\Vert_{ \B_{p',q'}^{-\beta}} \nonumber \\& \qquad\lesssim |y-y'|^{{\tilde \rho}} \frac{(\bar p_{\alpha}(t,y-x)+\bar p_{\alpha}(t,y'-x))}{(t-s)^{\frac{1}{\alpha}+\theta+\frac{{\tilde \rho}}{\alpha}}} t^{\frac{\beta}{\alpha}}\left[\frac{1}{s^{\frac{d }{\alpha  p}}}+ \frac{1}{(t-s)^{\frac{d }{\alpha  p}}} \right] \left[
				 \frac{t^{\frac{\zeta}{\alpha}}}{s^{\frac{\zeta }{\alpha}}}+\frac{t^{\frac{\zeta}{\alpha}}}{(t-s)^{\frac{\zeta }{\alpha}}}  \right].
			\end{align}			
We are now in position to continue the analysis. We get from \eqref{besov-estimate-gamma-stable-sensi-holder-time},
\begin{align}
&\|\mathcal F^{-1}(\phi)\star\Delta^{21}(0,x,t,t',\cdot)\|_{L^{p'}}\lesssim\|\Delta^{21}(0,x,t,t',\cdot)\|_{L^1}\notag\\
\lesssim &\int_0^{t-(t'-t)} \int_{\R^d}\Vert b (s,\cdot)\Vert_{B_{p,q}^\beta}  (t'-t)^\delta\frac{\bar p_{\alpha}(t,y-x) }{(t-s)^{\frac{1}{\alpha}+\delta}} t^{\frac{\beta}{\alpha}}\left[\frac{1}{s^{\frac{d }{\alpha  p}}}+ \frac{1}{(t-s)^{\frac{d }{\alpha  p}}} \right] \left[%1+ 
				\frac{t^{\frac{\zeta}{\alpha}}}{s^{\frac{\zeta }{\alpha}}}+\frac{t^{\frac{\zeta}{\alpha}}}{(t-s)^{\frac{\zeta }{\alpha}}}  \right] dy\d s\notag \\
\lesssim & (t'-t)^\delta \Vert b (s,\cdot)\Vert_{L^\vartheta-B_{p,q}^\beta} t^{\frac{\beta+\zeta}{\alpha}}\Big( \int_0^{t-(t'-t)}   \frac{1 }{(t-s)^{\vartheta'(\frac{1}{\alpha}+\delta)}} \left[\frac{1}{s^{\frac{d }{\alpha  p}}}+ \frac{1}{(t-s)^{\frac{d }{\alpha  p}}} \right]^{\vartheta'} \left[%1+ 
				\frac{1}{s^{\frac{\zeta }{\alpha}}}+\frac{1}{(t-s)^{\frac{\zeta }{\alpha}}}  \right]^{\vartheta'}\d s\Big)^{\frac1{\vartheta'}}		\notag\\
				\lesssim& 	(t'-t)^\delta \Vert b (s,\cdot)\Vert_{L^\vartheta-B_{p,q}^\beta} t^{\frac{\beta+\zeta}{\alpha}}\lesssim (t'-t)^{\frac{\gamma-\varepsilon-\beta}{\alpha}},	\label{DELTA_21_L1}
\end{align} 
taking $\delta= (-\beta+\gamma-\varepsilon)/\alpha $ and taking $\zeta =-\beta+\varepsilon/2$, with $\varepsilon\in (0,\gamma) $ to keep integrable singularities in the time integral.

We now keep the same previous notations as in \eqref{CUT_FOR_BESOV_NORM_HK_FOR_TR} for the corresponding upper and lower cuts. 
For the upper-cut we derive, using the $L^1 $-norm, similarly to the non-thermic part that, i.e. from \eqref{CTR_DELTA_3_NON_THERM_RT}:
\begin{align}
\mathcal T_{2,p',q'}^{\rho}(\Delta^{21}(0,x,t,t',\cdot))&\lesssim \int_{t}^T \frac{dv}{v}v^{(1-\frac{\rho}{\alpha}) q'}\|\partial_v \bar p_\alpha(v,\cdot)\|_{L^{p'}}^{q'}\|\Delta^{21}(0,x,t,t',\cdot)\|_{L^1}^{q'}\notag\\
&\lesssim \int_{t}^T \frac{dv}{v}v^{-(\frac{\rho}{\alpha}+\frac d{\alpha p}) q'}(t'-t)^{\frac{\gamma-\varepsilon-\beta}{\alpha} q'}
\lesssim \left(\frac{(t'-t)^{\frac{\gamma-\varepsilon}\alpha}}{t^{\frac{\beta+\rho}{\alpha}+\frac d{p\alpha}}}\right)^{q'}.\label{DELTA_21_UC}
\end{align}
For the lower cut, write using a Taylor expansion in time:\begin{align*}
&\mathcal T_{1,p',q'}^{\rho}(\Delta^{21}(0,x,t,t',\cdot))\\
\lesssim& \int_{0}^t \frac{dv}{v}v^{(1-\frac{\rho}{\alpha}) q'}\Big(
\int_0^{t-(t'-t)} ds \| \partial_v p_\alpha(v,\cdot) \star  \int \Gamma(0,x,s,z) \\
&\times b(s,z)\cdot [\int_{0}^1 d\lambda \partial_{r_\lambda}\nabla  p_\alpha(r_\lambda-s,\cdot-z)|_{r_\lambda=t+\lambda(t'-t)}(t'-t)]dz\|_{L^{p'}}\Big)^{q'}\\
\lesssim& (t'-t)^{q'} \int_{0}^t \frac{dv}{v}v^{(1-\frac{\rho}{\alpha}) q'}\Big(
\int_0^{t-(t'-t)} ds \Big(\int_{\R^d} dy \Big|\int_{\R^d}\partial_v p_\alpha(v,y-w) \\
& \times\Big[\int \Gamma(0,x,s,z) b(s,z)\cdot\Big\{\int_{0}^1 d\lambda [\partial_{r_\lambda}\nabla  p_\alpha(r_\lambda-s,w-z))-\partial_{r_\lambda}\nabla  p_\alpha(r_\lambda-s,y-z)]\Big\}_{r_\lambda=t+\lambda(t'-t)} dz\Big]dw\Big|^{p'}\Big)^{\frac{1}{p'}}\Big)^{q'}\\
\lesssim&(t'-t)^{q'} \int_{0}^t \frac{dv}{v}v^{(1-\frac{\rho}{\alpha}) q'}\Big(
\int_0^{t-(t'-t)} ds\int_0^1 d\lambda \Big(\int_{\R^d} dy \Big(\int_{\R^d} |\partial_v p_\alpha(v,y-w)| \\
&\times \|b(s,\cdot)\|_{B_{p,q}^{\beta}} \|\Gamma(0,x,s,\cdot) [\partial_{r_\lambda} \nabla  p_\alpha(r_\lambda-s,w-\cdot))-\partial_{r_\lambda}\nabla  p_\alpha(r_\lambda-s,y-\cdot)\|_{B_{p',q'}^{-\beta}}dw\Big)^{p'}\Big)_{r_\lambda=t+\lambda(t'-t)}^{\frac{1}{p'}}\Big)^{q'},
\end{align*}
using the Minkowski inequality and a cancellation argument for the first two inequalities and \eqref{dual-ineq} and the Minkowski inequality {to deal with} the integral in  $\lambda $ for the last one. Use now \eqref{besov-estimate-stable-derivees_temps_esp_sensi_holder_esp_GLB} {with $\tilde \rho\in (\rho,1]$}, to derive:
\begin{align*}
&\mathcal T_{1,p',q'}^{\rho}(\Delta^{21}(0,x,t,t',\cdot))\\
\lesssim&(t'-t)^{q'} \int_{0}^t \frac{dv}{v}v^{(1-\frac{\rho}{\alpha}) q'}(t')^{\frac{\beta}{\alpha}q'}\\
			&\times \Big(
\int_0^{t-(t'-t)} ds\|b(s,\cdot)\|_{B_{p,q}^{\beta}} \int_0^1d\lambda \frac 1{(r_\lambda-s)^{\frac 1\alpha+1+\frac{\tilde \rho}\alpha}}\left[ \frac{1}{s^{\frac{d }{\alpha  p}}}+\frac{1}{(r_\lambda-s)^{\frac{d }{\alpha  p}}} \right] \left[  
			\frac{(t')^{\frac{\zeta}{\alpha}}}{s^{\frac{\zeta }{\alpha}}}+\frac{(t')^{\frac{\zeta}{\alpha}}}{(t'-s)^{\frac{\zeta }{\alpha}}}  \right]\\
			&\times \Big(\int_{\R^{d}} dy \Big(\int_{\R^d} dw v^{-1}\bar p_\alpha(v,y-w) |y-w|^{\tilde \rho}
(\bar p_\alpha(r_\lambda,w-x)+\bar p_\alpha(r_\lambda,y-x))  \Big)_{r_\lambda=t+\lambda(t'-t)}^{p'} \Big)^{\frac 1{p'}} \Big)^{q'}\\
\lesssim &(t'-t)^{q'}\int_{0}^t \frac{dv}{v}v^{-\frac{\rho}{\alpha} q'}(t')^{\frac{\beta}{\alpha}q'}\\
			&\times \Big(
\int_0^{t-(t'-t)} ds\|b(s,\cdot)\|_{B_{p,q}^{\beta}} \frac 1{(t'-s)^{\frac 1\alpha+1+\frac{\tilde \rho}\alpha}}\left[ \frac{1}{s^{\frac{d }{\alpha  p}}}+\frac{1}{(t'-s)^{\frac{d }{\alpha  p}}} \right] \left[
			\frac{(t')^{\frac{\zeta}{\alpha}}}{s^{\frac{\zeta }{\alpha}}}+\frac{(t')^{\frac{\zeta}{\alpha}}}{(t'-s)^{\frac{\zeta }{\alpha}}}  \right]\\
&\times\Big(\|(\bar p_\alpha (v,\cdot)|\cdot|^{\tilde \rho})\star \bar p_{\alpha}(t',\cdot-x) \|_{L^{p'}} +\|\bar p_\alpha(t',\cdot-x)\|_{L^{p'}}\|\bar p_\alpha (v,\cdot)|\cdot|^{\tilde \rho}\|_{L^1}\Big)\Big)^{q'}\\
\lesssim &(t'-t)^{q'}\int_{0}^t \frac{dv}{v}v^{\frac{\tilde \rho-\rho}{\alpha} q'}(t')^{(\frac{\beta+\zeta-\frac d{\alpha p}}{\alpha})q'}\|b\|_{L^\vartheta-B_{p,q}^\beta}^{q'}\Big(\int_0^{t-(t'-t)} \frac{ds}{(t'-s)^{(\frac 1\alpha+1+\frac{\tilde \rho+\zeta}{\alpha}+\frac{d}{\alpha p})\vartheta'}}\Big)^{\frac{q'}{\vartheta'}}\\
			\lesssim&\Big(t^{\frac{\tilde \rho-\rho}{\alpha}+\frac{\beta+\zeta-\frac dp}{\alpha}}(t'-t)^{\frac{\alpha-1 -(\tilde \rho+\zeta)-(\frac dp+\frac{\alpha}{\vartheta})}{\alpha}}\Big)^{q'},
\end{align*}
recalling for the last inequality that the exponent in the {last} time integral is smaller than $-1$. Eventually choosing $\tilde \rho $ and $\zeta $ as above, i.e. $\tilde \rho=-\beta+\varepsilon/2=\zeta $, 
\begin{align*}
&\mathcal T_{1,p',q'}^{\rho}(\Delta^{21}(0,x,t,t',\cdot))\lesssim \Big(t^{-\frac d{p\alpha}}(t'-t)^{\frac{\gamma-\varepsilon}{\alpha}}\Big)^{q'}
\end{align*}
Using as well  the previous bounds established in \eqref{DELTA_21_L1}  and \eqref{DELTA_21_UC}, we eventually get :
\begin{align*}
\|\Delta^{21}(0,x,t,t',\cdot)\|_{B_{p',q'}^\rho}\lesssim (t'-t)^{\frac{\gamma-\varepsilon}{\alpha}}t^{-\frac{d}{p\alpha}}.
\end{align*}
Together with \eqref{CT_LT_HT} which is valid for $\|(\Delta^{22}(0,x,t,t',\cdot))\|_{B_{p',q'}^\rho}$ in the l.h.s, we derive: 
\begin{align*}
\|\Delta^{2}(0,x,t,t',\cdot)\|_{B_{p',q'}^\rho}\lesssim (t'-t)^{\frac{\gamma-\varepsilon}{\alpha}}t^{-(\frac{d}{p\alpha}+\frac{\beta+\rho}{\alpha})}.
\end{align*} 
From the above inequality, \eqref{CT_LT_HT}, \eqref{CT_MT_HT} and \eqref{THE_DIFF-BIS} the proof is complete.

\bibliographystyle{alpha}
\bibliography{BIBLI}

\end{document}